\documentclass[11pt,reqno]{article}

\usepackage[margin=1in]{geometry}
\usepackage{amsmath,amsthm,amssymb,amsfonts}

\makeatletter
\newcommand{\subjclass}[2][2010]{%
  \let\@oldtitle\@title%
  \gdef\@title{\@oldtitle\footnotetext{#1 \emph{Mathematics subject classification.} #2}}%
}
\makeatother

\usepackage[colorlinks=true, pdfstartview=FitV, linkcolor=blue,citecolor=blue, urlcolor=blue]{hyperref}

\usepackage[abbrev,lite,nobysame]{amsrefs}
\usepackage{times}
\usepackage[usenames,dvipsnames]{color}

\usepackage{esint}

\usepackage{mathtools,enumitem,mathrsfs}
\mathtoolsset{showonlyrefs=true}

\usepackage[compact]{titlesec}

\usepackage[title]{appendix}

\usepackage{comment}

\newcommand{\ep}{\epsilon}
\newcommand{\eps}{\epsilon}

\newcommand{\leqc}{\lesssim}

\newcommand{\grad}{\nabla}
\newcommand{\MalD}{\cD}
\newcommand{\Wbf}{\mathbf{W}}

\newcommand{\Lbf}{\mathbf{L}}

\newcommand{\Vc}{{\mathcal V}}
\newcommand{\Dc}{{\mathcal D}}

\newcommand{\norm}[1]{\left|\left| #1 \right|\right|}
\newcommand{\abs}[1]{\left| #1 \right|}
\newcommand{\set}[1]{\left\{ #1 \right\}}
\newcommand{\brak}[1]{\left\langle #1 \right\rangle} 

\newcommand{\R}{\mathbb{R}}
\newcommand{\Real}{\R}

\newcommand{\Z}{\mathbb{Z}}

\newcommand{\T}{\mathbb{T}}

\renewcommand{\S}{\mathbb{S}}

\newcommand{\tensor}{\otimes}

\newcommand{\cM}{\mathcal{M}}

\newcommand{\cD}{\mathcal{D}}

\newcommand{\Hbf}{{\bf H}}

\newcommand{\dee}{\mathrm{d}}
\newcommand{\ds}{\dee s}
\newcommand{\dt}{\dee t}

\newcommand{\dx}{\dee x}

\newcommand{\dy}{\dee y}

\newcommand{\dr}{\dee r}

\DeclareMathOperator{\Div}{\mathrm{div}}
\DeclareMathOperator{\Id}{\mathrm{Id}}

\DeclareMathOperator{\curl}{\mathrm{curl}}

\usepackage{bbm}
\newcommand{\1}{\mathbbm{1}}

\renewcommand{\P}{\mathbf{P}}

\newcommand{\E}{\mathbf{E}}
\newcommand{\EE}{\mathbf E}
\newcommand{\PP}{\mathbf P}

\newtheorem{theorem}{Theorem}[section]
\newtheorem{proposition}[theorem]{Proposition}
\newtheorem{corollary}[theorem]{Corollary}
\newtheorem{lemma}[theorem]{Lemma}
\newtheorem*{lemma*}{Lemma}

\newtheorem{assumption}{Assumption}
\newtheorem{system}{System}

\theoremstyle{definition}
\newtheorem{definition}[theorem]{Definition}
\newtheorem{remark}[theorem]{Remark}

\numberwithin{equation}{section}
    
\begin{document}

\title{Almost-sure enhanced dissipation and uniform-in-diffusivity exponential mixing  for advection-diffusion by stochastic Navier-Stokes}
\subjclass{Primary: 37A25, 76D06, 60H15, 37H15, 76F20}
\author{Jacob Bedrossian\thanks{\footnotesize Department of Mathematics, University of Maryland, College Park, MD 20742, USA \href{mailto:jacob@math.umd.edu}{\texttt{jacob@math.umd.edu}}. J.B. was supported by NSF CAREER grant DMS-1552826 and NSF RNMS \#1107444 (Ki-Net)} \and Alex Blumenthal\thanks{\footnotesize Department of Mathematics, University of Maryland, College Park, MD 20742, USA \href{mailto:alex123@math.umd.edu}{\texttt{alex123@math.umd.edu}}. This material is based upon work supported by the National Science Foundation under Award No. DMS-1604805.} \and Sam Punshon-Smith\thanks{\footnotesize Division of Applied Mathematics,  Brown University, Providence, RI 02906, USA \href{mailto:punshs@brown.edu}{\texttt{punshs@brown.edu}}. This material is based upon work supported by the National Science Foundation under Award No. DMS-1803481.}}

\maketitle

\begin{abstract}
We study the mixing and dissipation properties of the advection-diffusion equation with diffusivity $0 < \kappa \ll 1$ and advection by a class of random velocity fields on $\mathbb T^d$, $d=\{2,3\}$, including solutions of the 2D Navier-Stokes equations forced by sufficiently regular-in-space, non-degenerate white-in-time noise. We prove that the solution almost surely mixes exponentially fast uniformly in the diffusivity $\kappa$. Namely, that there is a deterministic, exponential rate (independent of $\kappa$) such that all mean-zero $H^1$ initial data decays exponentially fast in $H^{-1}$ at this rate with probability one. This implies almost-sure enhanced dissipation in $L^2$. Specifically that there is a deterministic, uniform-in-$\kappa$, exponential decay in $L^2$ after time $t \gtrsim \abs{\log \kappa}$. Both the $\mathcal{O}(\abs{\log \kappa})$ time-scale and the uniform-in-$\kappa$ exponential mixing are optimal for Lipschitz velocity fields and, to our knowledge, are the first rigorous examples of velocity fields satisfying these properties (deterministic or stochastic). This work is also a major step in our program on scalar mixing and Lagrangian chaos necessary for a rigorous proof of the Batchelor power spectrum of passive scalar turbulence. 
\end{abstract}

\setcounter{tocdepth}{2}
{\small\tableofcontents}

\section{Introduction}\label{sec:Intro}

The evolution of a passive scalar $g_t$ under an incompressible fluid motion $u_t$ is a fundamental problem in physics and engineering; see e.g. \cite{Provenzale1999,ShraimanSiggia00,W00,WF04,LDT11} and the references therein. In applications, the scalar $g_t$ is typically the temperature distribution or a chemical concentration that can be treated as a passive tracer. 
Here we study the advection-diffusion equation with diffusivity $0 < \kappa \ll 1$, 
\begin{align}
  & \partial_t g_t + u_t \cdot \grad g_t = \kappa \Delta g_t \label{def:AD} \\ 
  & g_0 = g,   
\end{align}
on the periodic box $\T^d = [0,\pi]^d$ where $g$ is a mean-zero $L^2$ function and $u_t$ is an incompressible velocity field evolving under any one of a variety of stochastic fluid models, for example, the stochastically-forced 2D Navier-Stokes equations. We set $u_0 = u$, the initial condition of the fluid evolution (assumed to be in a sufficiently regular Sobolev space).

Understanding the mixing and dissipation of $g_t$ under various fluid motions $(u_t)$ is a central question in both physics and engineering applications, and has recently received significant attention from the mathematics community, for example \cite{Bressan03,LDT11,LLNMD12,S13,IyerXu14,YZ17,MD18,TZ18,CZDE18,CRS19,GGM19,IF19,BBPS18,BBPS19,Thiffeault2012-xs} and the references therein (also see below for more discussion).
One case, crucial for many physical applications, not studied in the mathematics community (until \cite{BBPS19}) is that of velocity fields evolving under ergodic, nonlinear dynamics. In \cite{BBPS19}, we showed that if $(u_t)$ evolves according to the stochastically-forced Navier-Stokes equations, then in the absence of diffusivity (i.e., \eqref{def:AD} with $\kappa =0$), the passive scalar mixes exponentially fast almost surely with respect to the noise on the fluid equation. Specifically, we show exponential decay in any negative Sobolev norm
\begin{align}
	\|g_t\|_{H^{-s}} := \sup_{\|f\|_{H^s} =1 }\left|\int f \,g_t\, \dx\right| \leq D e^{-\gamma t}\|g\|_{H^s}, \label{eq:k=0-mixing}
\end{align}
where $D(s,u,\omega)$ is a random constant with finite moments (independent of $g$, but depending on $u$
and the noise sample $\omega$), and $\gamma>0$ is a \emph{deterministic} constant (independent of $g$, $u$ and $\omega$). The use of negative Sobolev norms to measure mixing is standard in the literature and their decay corresponds to mixing in the sense of ergodic theory (see discussions in \cite{Thiffeault2012-xs} and the references therein; see also \cite{Z18}). 
It is easy to check that Lipschitz velocity fields that satisfy standard moment estimates cannot mix scalars faster than \eqref{eq:k=0-mixing} (see \cite{BBPS18,BBPS19} and Remark \ref{rmk:ExpOpt}). 

The mixing in \eqref{def:AD} arises due to the chaotic nature of the Lagrangian trajectories, a phenomenon referred to as {\em chaotic mixing}.  
Chaos in the Lagrangian flow map is often referred to as \emph{Lagrangian chaos} (to distinguish it from the property of $u_t$ itself being chaotic; see discussions in \cite{BMOV05}). 
In our first work \cite{BBPS18}, we proved positivity of the top Lyapunov exponent (a hallmark of sensitivity with respect to initial conditions) for the Lagrangian flow.
This provides a local hyperbolicity to the flow, and this was subsequently upgraded to the \emph{global} almost-sure, exponential mixing statement in \eqref{eq:k=0-mixing} by our second work \cite{BBPS19} (the work \cite{BBPS19} uses \cite{BBPS18} as a lemma).  
We emphasize that the mixing mechanism here is not turbulence or small scales in the velocity field $u_t$-- indeed, the fields we work with are, at minimum, 
$C^2$ spatially regular and it is not directly relevant whether or not $u_t$ is chaotic. 
See e.g. \cite{YuanEtAl00,AntonsenEtAl96,GalluccioVulpiani94,AmonEtAl96,HV05,TAO05,NV11}, the reviews \cite{Ottino1990,CrisantiEtAl1991,ArefEtAl17}, and the references therein for more discussion in the physics literature on chaotic mixing and Lagrangian chaos. 

The primary goal of the current paper is to prove that the almost-sure exponential mixing estimate \eqref{eq:k=0-mixing} holds also for \eqref{def:AD} for $0 < \kappa \ll 1$ \emph{uniformly in $\kappa$}, that is, for $\gamma$ \emph{independent} of $\kappa$ and random constant $D$ that satisfies uniform estimates in $\kappa$ (see Theorem \ref{thm:UniMix} below). It is important to note that $\kappa > 0$ is a {\em singular perturbation} of $\kappa = 0$, and to our knowledge, there is \emph{no general method} in the literature by which one can deduce uniform exponential mixing from the knowledge that one has exponential mixing at $\kappa = 0$, for either deterministic or stochastic velocities.
Indeed, the only uniform-in-diffusivity mixing we are aware of are only at a polynomial rate and are all essentially shear flows: inviscid damping in the Navier-Stokes equations near Couette flow \cite{BMV14,BGM15I}; the recent work \cite{CZ19} on passive scalars in strictly monotone shear flows; and Landau damping in Vlasov-Poisson with weak collisions \cite{Tristani17,B17}. 
In fact, it is known that the introduction of diffusion can limit the mixing rate in certain contexts \cite{MD18}.

When $\kappa >0$, the scalar additionally dissipates in $L^2$ due to the diffusivity:
\begin{align}
\frac{1}{2}\frac{\dee}{\dt} \norm{g_t}_{L^2}^2 = - \kappa \norm{\grad g_t}_{L^2}^2. 
\end{align}
From this balance it is clear that the creation of small scales due to mixing could accelerate the $L^2$ dissipation rate. This effect is usually called \emph{relaxation enhancement} or \emph{enhanced dissipation}. The first general, mathematically rigorous study of this effect in deterministic, constant-in-time velocity fields was the foundational work \cite{CKRZ08} (see e.g. \cite{Bajer2001,LatiniBernoff01,BernoffLingevitch94,RhinesYoung83} for some of the earlier work in the physics literature). The effect is now being actively studied both for passive scalars \cite{Zlatos2010,BCZGH,BCZ15,CZDE18,IF19} and also in the context of hydrodynamic stability of shear flows and vortices (see e.g. \cite{BMV14,BGM15I,BGM15II,WZZ_2DK,Gallay18}  and the references therein). 
In \cite{CZDE18,IF19}, it was shown that if a deterministic flow is exponentially mixing for $\kappa = 0$, then one sees exponential $L^2$ dissipation after $t \gtrsim \abs{\log \kappa}^{2}$. The uniform-in-$\kappa$ exponential mixing we deduce for \eqref{def:AD} in Theorem \ref{thm:UniMix} allows to obtain the rapid exponential $L^2$ dissipation after $t \gtrsim \abs{\log \kappa}$ in Theorem \ref{thm:ED} (note that for stochastic velocities, this time scale is random).
This time-scale is easily seen to be optimal for Lipschitz fields that satisfy standard moment estimates (Theorem \ref{thm:optTimeScale}).
We emphasize here that if uniform-in-$\kappa$ mixing were available for deterministic fields, then corresponding optimal improvements of \cite{CZDE18,IF19} could be proved with simpler arguments than those in \cite{CZDE18,IF19} (similarly, some of the results of \cite{CKRZ08}). However, such mixing estimates are currently unavailable.

In addition to the intrinsic interest, the results herein are a crucial step in our program on Lagrangian chaos and scalar mixing required for our proof of Batchelor's Law for the 
power spectrum of passive scalar turbulence in the forthcoming article \cite{BBPS19_BL}.
First conjectured in 1959 \cite{Batchelor59}, Batchelor's Law predicts that the distribution of $\EE \abs{\hat{g}_t(k)}^2$ behaves like $|k|^{-d}$ for statistically stationary passive scalars subject to random sources in the $\kappa \to 0$ limit with the Reynolds number of the fluid held fixed (the so-called Batchelor regime of passive scalar turbulence). Batchelor's law is the analogue of Kolmogorov's prediction of the $-5/3$ power law spectrum in 3D Navier-Stokes \cite{Frisch1995}.
Theorem \ref{thm:UniMix} below provides the quantitative information on the low-to-high frequency cascade required to verify this power spectrum law. 
This will be, to the best of our knowledge, the first-ever proof of a  
power spectrum law for the turbulent regime of a fundamental physical model. 
See, e.g., \cite{GibsonSchwarz63,AO03,AK04,AntonsenEtAl96,DSY10}, our forthcoming preprint \cite{BBPS19_BL}, and the references therein for more information. In particular, note that neither the validity or scope of Batchelor's law is completely settled 
in the physics literature (see discussions in \cite{MD96,AO03,DSY10}) and our results will show that the Batchelor spectrum is universal for a variety of different settings.

\subsection{Stochastic Navier-Stokes}
For simplicity, we first state our main results for the most physically interesting and mathematically challenging cases: the stochastic 2D Navier-Stokes equations and the 3D hyperviscous Navier-Stokes equations (on $\mathbb T^d$, $d=2,3$ respectively).
In Section \ref{sec:Finite} we discuss the setting used to study finite dimensional models, which allow for smoother (in both space and time) velocity fields.  

We define the natural Hilbert space on velocity fields $u:\T^d \to \R^d$ by
\begin{align}
  \Lbf^2 := \set{u \in L^2(\mathbb T^d;\Real^d)\,: \,\int u\, \dx = 0,\quad \Div u = 0},   
\end{align}
with the natural $L^2$ inner product. Let $W_t$ be a cylindrical Wiener process on $\Lbf^2$ with respect to an associated canonical stochastic basis $(\Omega,\mathscr{F},(\mathscr{F}_t),\P)$ and $Q$ a positive Hilbert-Schmidt operator on $\Lbf^2$, diagonalizable with respect the Fourier basis on $\Lbf^2$. We will assume that $Q$ satisfies the following regularity and non-degeneracy assumption (see Section \ref{sec:Finite} for more discussion):
\begin{assumption} \label{a:Highs}
  There exists $\alpha$ satisfying $\alpha > \frac{5d}{2}$ and a constant $C$ such that
\[
	\frac{1}{C}\|(-\Delta)^{-\alpha/2}u\|_{\Lbf^2} \leq \|Qu\|_{\Lbf^2} \leq C\|(-\Delta)^{-\alpha/2}u\|_{\Lbf^2}.
\]
\end{assumption}
\noindent
We define our primary phase space of interest to be velocity fields with sufficient Sobolev regularity: 
\[
	\Hbf := \set{u \in H^\sigma(\T^d, \R^d)\, :\, \int u \, \dx = 0,\quad \Div u = 0}, \quad \text{where}\quad \,\sigma \in (\alpha-2(d-1), \alpha  - \tfrac{d}{2}). 
\]
Note we have chosen $\alpha$ sufficiently large to ensure that $\sigma > \frac{d}{2} + 3$ so that $\Hbf \hookrightarrow C^3$. 

We consider $(u_t)$ evolving in $\Hbf$, which we refer to as the {\em velocity process}, by one of the two following stochastic PDEs: 
\begin{system}[2D Navier-Stokes equations]\label{sys:NSE}
\begin{equation}
\begin{dcases}
\,\partial_t u_t + u_t \cdot \grad u_t =- \grad p_t + \nu \Delta u_t + Q \dot W_t \\ 
\,\Div u_t = 0 \, , 
\end{dcases}
\end{equation}
where $u_0 = u \in \Hbf$. Here, the viscosity $\nu >0$ is a fixed constant.
\end{system}
\begin{system}[3D hyper-viscous Navier-Stokes]\label{sys:3DNSE}
\begin{equation}
\begin{dcases}
\,\partial_t u_t + u_t \cdot \grad u_t =- \grad p_t + \nu' \Delta u_t - \nu \Delta^{2} u_t + Q \dot W_t \\ 
\,\Div u_t = 0,
\end{dcases}
\end{equation}
where $u_0 = u \in \Hbf$. Here, the viscosity $\nu' \geq 0$, and hyperviscosity $\nu>0$ are fixed constants.
\end{system}

Since we will need to take advantage of the ``energy estimates'' produced by the vorticity structure of the Navier-Stokes equations in $2D$, we find it notationally convenient to define the following dimension dependent norm
\begin{equation}\label{eq:vorticity-norm}
\|u\|_{\Wbf} := \begin{cases}
\|\curl u\|_{\Lbf^2} & d=2\\
\|u\|_{\Lbf^2} & d=3.
\end{cases}
\end{equation}

The following well-posedness theorem is classical (See Section \ref{sec:Prelim}).   

\begin{proposition} \label{prop:WP}
For both Systems \ref{sys:NSE}, \ref{sys:3DNSE} and all initial data $u\in \Hbf$,
there exists a $\P$-a.s. unique, global-in-time, $\mathscr{F}_t$-adapted mild solution $(u_t)$ satisfying $u_0 = u$. 
Moreover, $(u_t)$ defines a Feller Markov process on $\Hbf$ and the corresponding Markov semigroup has a unique stationary probability measure $\mu$ on $\Hbf$.
\end{proposition}


\subsection{Main results}

The first result is uniform-in-$\kappa$ exponential mixing for passive scalars.  
It is important to emphasize that the methods we employ in Theorem \ref{thm:UniMix} are inherently stochastic. This is not simply because they rely directly on the results of  \cite{BBPS19,BBPS18}, 
but also because the extension from $\kappa = 0$ to $\kappa > 0$ 
requires the use of the stochastic nature of Systems \ref{sys:NSE}--\ref{sys:3DNSE}. 
A general method for extending exponential mixing at $\kappa = 0$ to uniform-in-$\kappa$ mixing does not, to our knowledge, currently exist.
Here and for the remainder of the paper, implicit constants will \emph{never depend on $\omega$, $\kappa$, $t$, $(u_t)$, or $(g_t)$}. See Section \ref{sec:Note} for notation conventions. 

\begin{theorem}[Uniform mixing] \label{thm:UniMix}
For each of Systems \ref{sys:NSE} -- \ref{sys:3DNSE}, there exists a \emph{deterministic} $\kappa_0 > 0$ such that for all $s>0$, $p \geq 1$
there exists a \emph{deterministic} $\gamma = \gamma(s, p) > 0$ (depending only on $s,p$ and the parameters $Q$, $\nu$ etc) which satisfies the following properties.  
For all $\kappa \in [0,\kappa_0]$, and for all $u \in \Hbf$ 
there is a $\PP$-a.s. finite random constant $D_{\kappa}(\omega,u): \Omega \times \Hbf \to [1,\infty)$ (also depending on $p, s$) such that the solution to \eqref{def:AD} with $(u_t)$ given by the corresponding System \ref{sys:NSE} or \ref{sys:3DNSE} with initial data $u$, satisfies for all $g \in H^s$ (mean-zero), 
\begin{align}
\norm{g_t}_{H^{-s}} \leq D_\kappa(\omega,u) e^{-\gamma t} \norm{g}_{H^s} \, ,
\end{align}
where  $D_\kappa(\omega, u)$ satisfies the following $\kappa$-independent bound: there exists a $\beta \geq 2$ (independent of $u$, $p$, $s$) such that for all $\eta > 0$,  
\begin{align}
\EE D_\kappa^p(\cdot,u) \lesssim_{\eta,p} (1 + \norm{u}_{\Hbf})^{p \beta} \exp \left( \eta \norm{u}_{\Wbf}^2 \right) \, . \label{ineq:Dkap1}
\end{align}
\end{theorem}

Theorem \ref{thm:UniMix} implies enhanced dissipation as well.
Indeed, the Sobolev interpolation $\norm{g}_{L^2} \leq \norm{g}_{H^1}^{1/2}\norm{g}_{H^{-1}}^{1/2}$ relates the dissipation rate directly to the ratio of the $L^2$ and $H^{-1}$ norms:  
\begin{align}
\frac{\dee}{\dt}\norm{g_t}_{L^2}^2 = - 2 \kappa \norm{\grad g_t}_{L^2}^2 \leq -2 \kappa \frac{\norm{g_t}_{L^2}^2}{\norm{g_t}_{H^{-1}}^2} \norm{g_t}_{L^2}^2. 
\end{align}
Theorem \ref{thm:UniMix} in turn provides a quantitative lower bound on the dissipation rate that is integrated and combined with parabolic regularity to deduce the following enhanced dissipation (See Section \ref{sec:ED} for more details). 
The recent quantitative works of \cite{CZDE18,IF19} and the earlier more qualitative works \cite{CKRZ08,Zlatos2010} required much more subtle arguments because there is not yet an analogue of Theorem \ref{thm:UniMix} for any deterministic velocity fields.
Theorem \ref{thm:ED} also provides stronger results than those of \cite{CZDE18,IF19} in terms of both the rate of decay and the characteristic time-scale of enhanced dissipation.

\begin{theorem}[Enhanced dissipation] \label{thm:ED}
In the setting of Theorem \ref{thm:UniMix}, for any $p \geq 2$, let $\gamma = \gamma(1,p)$ be as in Theorem \ref{thm:UniMix}.
For all $\kappa \in (0,\kappa_0]$, and for all $u \in \Hbf$ 
there is a $\PP$-a.s. finite random constant $D_{\kappa}^\prime(\omega,u): \Omega \times \Hbf \to [1,\infty)$ (also depending on $p$) such that the solution to \eqref{def:AD}, satisfies for all $g \in H^s$ (mean-zero) and $u\in \Hbf$,
\begin{align}
\norm{g_t}_{L^2} \leq D_{\kappa}^\prime(\omega,u) \kappa^{-1} e^{-\gamma t} \norm{g}_{L^2}, \label{ineq:ED}
\end{align}
where $D_\kappa^\prime$ also satisfies the following $\kappa$-independent bound for $\beta$ sufficiently large (independent of $u$, $p$, $\kappa$) and for all $\eta > 0$, 
\begin{align}
\EE \left(D_\kappa'(\cdot,u)\right)^p \lesssim_{\eta,p} (1 + \norm{u}_{\Hbf})^{p \beta} \exp \left( \eta \norm{u}_{\Wbf}^2 \right) \, . \label{ineq:Dpkapps}
\end{align}
\end{theorem}
\begin{remark}
Note that Theorem \ref{thm:ED} implies the following: 
\[
	\|g_t\|_{L^2} \lesssim D_{\kappa}' e^{-\delta |\log\kappa|^{-1}t}\|g\|_{L^2},
\]
where the implicit constant does not depend on $\kappa$ and $D_\kappa'$ satisfies \eqref{ineq:Dpkapps}.  
Both results give the same characteristic time-scale of decay ($\tau_{ED} \sim \abs{\log \kappa}$) but Theorem \ref{thm:ED} gives faster drop off past that time.
\end{remark}
The next estimate shows that the $\log \kappa$ dissipation time-scale is optimal for $H^1$ data. 
This estimate is a simple consequence of the regularity of the velocity field, which implies small scales in the passive scalar cannot be generated faster than exponential.
The estimate is basically trivial for bounded, deterministic velocity fields; for unbounded stochastic velocity fields that can make large deviations, the dissipation time scale is a stopping time, and the estimate is less trivial. Lower bounds on this time show that the $|\log{\kappa}|$ timescale is optimal. See Section \ref{sec:EDopt} for a proof.
\begin{theorem}[Optimality of the $\abs{\log \kappa}$ time-scale]\label{thm:optTimeScale}
In the setting of Theorem \ref{thm:UniMix}, let
\[
\tau_\ast = \inf\set{t: \norm{g_t}_{L^2} < \tfrac{1}{2}\|g\|_{L^2}} \, .
\]
Then, there exists a $\kappa_0 > 0$ a sufficiently small universal constant such that
for all $\kappa\in (0,\kappa_0]$, one has
\[
	\tau_\ast \geq \delta(g,u,\omega) |\log{\kappa}| \quad\text{with probability}\quad 1 \, , 
\]
where $\delta (g, u, \omega) \in (0,1)$ is a $\kappa$-independent random constant
with the property that for all $\beta \geq 1, p \geq 1$ and $\eta > 0$, 
\begin{align}
\E \delta^{-p} \leqc_{p,\eta, \beta} \frac{\|g\|_{L^2}^p}{\|g\|_{H^1}^p} \left(1 + \norm{u}_{\Hbf}\right)^{p \beta} \exp \left( \eta \norm{u}_{\Wbf}^2 \right) \, . \label{ineq:deltap}
\end{align}
\end{theorem}

\begin{remark} \label{rmk:ExpOpt}
The proof of Theorem \ref{thm:optTimeScale} shows that the $H^{-1}$ exponential decay of Theorem \ref{thm:UniMix} is sharp even in the presence of diffusion\footnote{The case without diffusion follows almost immediately from the multiplicative ergodic theorem (see \cite{BBPS18}), however, it requires an additional check to ensure that the random constant $\overline{D}$ possesses good moment bounds (Lemma \ref{lem:BVH1explo}).}
That is, for all $p \geq 1$, there exists an almost-surely finite random constant $\overline{D}(\omega,u)$ (independent of $\kappa$) and a deterministic $\mu = \mu(p) > 0$ (independent of $u,\kappa$) such that for all $g \in H^1$, and $t < \tau_\ast$ (as in Theorem \ref{thm:optTimeScale}), 
\begin{align}
\norm{g_t}_{H^{-1}} \geq \overline{D}(\omega,u) e^{-\mu t} \frac{\norm{g}_{L^2}^2}{\norm{g}_{H^1}}. 
\end{align}
Moreover, the random constant satisfies $\EE(\overline{D})^{-p} \lesssim _{\eta,p} (1 + \norm{u}_{\Hbf})^{p \beta} \exp \left( \eta \norm{u}_{\Wbf}^2 \right)$ as in e.g. \eqref{ineq:Dkap1}. 
\end{remark}

\subsection{Finite dimensional models and $C_t^k C^\infty_x$ examples} \label{sec:Finite}
Assumption \ref{a:Highs} essentially says that the forcing is $QW_t$ has high spatial regularity, but cannot be $C^\infty$. The non-degeneracy requirement on $Q$ can be weakened to a more mild non-degeneracy at only high-frequencies (see \cite{BBPS18}), but fully non-degenerate noise simplifies some arguments.
As discussed in \cite{BBPS18,BBPS19}, non-degenerate noise is used to prove strong Feller for the infinite dimensional Furstenberg criterion [Theorem 4.7, \cite{BBPS18}] on which \cite{BBPS19}, and hence this work, depends critically. 
It is also used in \cite{BBPS19} and here to access geometric ergodicity in a wider variety of spaces than that currently available in asymptotically strong Feller frameworks of \cite{HM06,HM11} (see discussions in \cite{BBPS19} for more details). 
In \emph{all other places} in \cite{BBPS18,BBPS19} and here, non-degenerate noise is used only to reduce the length and complexity of the works. However, for velocity fields evolving according to finite dimensional models, degenerate noise is easily treated by H\"ormander's theorem. 
This provides a robust way to produce examples of $C_t^k C^\infty_x$ random fields satisfying Theorems \ref{thm:UniMix} and \ref{thm:ED}. 

To make this more precise: 
in all cases considered in this work, the additive noise term $Q\dot{W}_t$ can be represented in terms of a Fourier basis $\{e_m\}_{m\in \mathbb{K}}$ on $\Lbf^2$ by
\[
	Q\dot{W}_t = \sum_{m\in\mathbb{K}} q_m e_m \dot{W}^m_t 
\]
where $\mathbb{K} := \Z^d_0 \times \{1,\ldots,d-1\}$ and $\{W^m_t\}_{m\in \mathbb{K}}$ are a collection of iid one-dimensional Wiener processes with respect to $(\Omega,\mathscr{F},(\mathscr{F}_t),\P)$ see (Section \ref{sec:Prelim} for more details and the precise definition of the Fourier basis). 

In this notation, we can consider the following weaker non-degeneracy condition: 
\begin{assumption}[Low mode non-degeneracy] \label{a:lowms}
Define $\mathcal{K}_0 \subset \mathbb K$ to be the set of $m \in \mathbb K$ such that $q_m \neq 0$. Assume $m \in \mathcal{K}_0$ if $\abs{m}_{\infty} \leq 2$ (for $m = (k,i)$, $k = (k_i)_{i = 1}^d \in \Z^d$ we write $|m|_{\infty} = \max_{i} |k_i|$). 
\end{assumption}

We write $\Hbf_{\mathcal K_0} \subset \Hbf$ for the subspace spanned by the Fourier modes $m \in \mathcal K_0$ and $\Hbf_N \subset \Hbf$ for the subspace spanned by the Fourier modes satisfying $\abs{m}_\infty \leq N$.
Consider the Stokes system (with very degenerate forcing) and Galerkin-Navier-Stokes systems defined as follows. 

\begin{system} \label{sys:2DStokes} 
The \emph{Stokes system} in $\mathbb T^d$ ($d = 2,3$) is defined, for $u_0 = u \in \Hbf_{\mathcal K}$, by
\begin{equation} \label{eq:stokes-2d}
\begin{cases}
\,\partial_t u_t =- \grad p_t + \Delta u_t + Q \dot W_t \\ 
\,\Div u_t = 0 
\end{cases},
\end{equation}
where $Q$ satisfies Assumption \ref{a:lowms} and $\mathcal K_0$ is finite.
\end{system}

\begin{system} \label{sys:Galerkin}
The \emph{Galerkin-Navier-Stokes system} in $\mathbb T^d$ ($d = 2,3$) is defined, for $u_0 = u \in \Hbf_N$, by
\begin{equation}
\begin{cases}
\,\partial_t u_t + \Pi_{\leq N}\left(u_t\cdot \grad u_t + \grad p_t \right) = \nu\Delta u_t + Q \dot W_t \\ 
\,\Div u_t = 0
\end{cases} \, , 
\end{equation}
where $Q$ satisfies Assumption \ref{a:lowms}; $N \geq 3$ is an arbitrary integer; $\Pi_{\leq N}$ denotes the projection to 
Fourier modes with $| \cdot|_{\infty}$ norm $\leq N$; $\Hbf_N$ denotes the span of the first $N$ Fourier modes; and $\nu > 0$ is fixed and arbitrary.
\end{system}

Note that velocity fields $u_t$ evolving according to Systems \ref{sys:2DStokes} and \ref{sys:Galerkin} are
spatially $C_x^\infty$ and, at best, $\frac12$-H\"older in time. We are also 
able to treat a class of evolutions with non-white-in-time forcing, referred to as `OU tower noise' 
in \cite{BBPS19}.
This is basically an external forcing given by the projection of an Ornstein-Uhlenbeck process on $\Real^M$. 
\begin{system}\label{sys:Markov}
\begin{subequations} \label{eqn:finiteD}
The \emph{(generalized) Galerkin-Navier-Stokes system with OU tower noise} in $\mathbb T^d$ ($d = 2,3$) is defined, for
$u_0 \in \Hbf_N$, by the stochastic ODE
\begin{align}
  \partial_t u_t & + X(u,u) = \nu \Delta u_t + Q Z_t \\
  \partial_t Z_t & = -\mathcal{A}Z_t + \Gamma \dot{W}_t \, , 
\end{align} 
\end{subequations}
where $Z_t \in \Hbf_M$, the operator $\mathcal{A} : \Hbf_M \to \Hbf_M$ is diagonalizable and has a strictly positive spectrum, and the bilinear term $X(u,u):\Hbf_N \times \Hbf_N \to \Hbf_N$ satisfies $u \cdot X(u,u) = 0$ and $\forall j$,  $X(e_j,e_j) = 0$. 
Note that $(u_t)$ is not Markov, but $(u_t,Z_t)$ is Markov and one must also specify the initial condition for the $(Z_t)$ process, i.e. $Z_0 = Z$, when considering this setting. 
\end{system}

All of our results extend to each of Systems \ref{sys:2DStokes}, \ref{sys:Galerkin}, and \ref{sys:Markov}. 

\begin{theorem} \label{thm:FiniteDregs}
Consider any of Systems \ref{sys:2DStokes}--\ref{sys:Markov}.
Assume that $Q$ satisfies Assumption \ref{a:lowms} and that the parabolic H\"ormander condition is satisfied for $(u_t)$ or $(u_t,Z_t)$ (see e.g. \cite{Hairer11notes}).  
Then, Theorems \ref{thm:UniMix}, \ref{thm:ED}, and \ref{thm:optTimeScale} all hold (in the case of System \ref{sys:Markov}, the estimates on the random constants in \eqref{ineq:Dkap1}, \eqref{ineq:Dpkapps}, and \eqref{ineq:deltap} all contain an additional factor of $\exp\left(\eta \abs{Z}^2 \right)$, i.e. the initial condition for the $Z_t$ process). 
\end{theorem}

\begin{remark}
Note that for all $k \geq 0$, one can arrange so that solutions to System \ref{sys:Markov} satisfy $(u_t) \in L^p (\Omega;C^k_{t,loc} C^\infty_x)$ for all $p < \infty$. See \cite{BBPS19} for more details.  
\end{remark}

\begin{remark}
We have chosen to include Theorem \ref{thm:FiniteDregs} to emphasize that our methods do not fundamentally require non-$C^\infty_x$ velocity fields,
nor do they require velocity fields that are directly subjected to white-in-time forcing.
The difficulty in treating infinite dimensional models with smooth-in-space, $C^k_t$ forcing of `OU tower' type is the lack of an adequate extension of H\"ormander's theorem to infinite dimensions (though, note that the theory of Hairer and Mattingly \cite{HM11} applies to OU tower forcing).
In addition, it would also be interesting to extend our works \cite{BBPS18,BBPS19} and this work to the non-white-in-time, uniformly bounded forcing studied in \cite{KNS18,KNS19,JNPS19}. 
\end{remark}

\section{Outline}\label{sec:outline}

We will henceforth only discuss the proof for the infinite dimensional stochastic Navier-Stokes Systems \ref{sys:NSE}--\ref{sys:3DNSE}.
Essentially the same proof applies to the systems in Section \ref{sec:Finite} but each step is vastly simplified by the finite dimensionality (see \cite{BBPS19} for a brief discussion about the small changes required to treat System \ref{sys:Markov}). 

The vast majority of the work in this paper is to prove Theorem \ref{thm:UniMix}, which we
outline here. The proofs of Theorems \ref{thm:ED} and \ref{thm:optTimeScale} are 
discussed in Section \ref{sec:ED}. 

\subsection{Uniform mixing by uniform geometric ergodicity of two-point Lagrangian process}

The proof is based on the representation of the advection-diffusion equation as a Kolmogorov equation of the corresponding stochastic Lagrangian process. 
To do this, let $\widetilde W_t$ denote a standard $d$-dimensional Wiener process with respect to a separate stochastic basis $(\widetilde{\Omega}, \widetilde{\mathscr{F}}, \widetilde{\mathscr{F}}_t, \widetilde{\mathbf{P}})$. This naturally gives rise to an augmented probability space $\Omega\times\widetilde{\Omega}$ with the associated product sigma-algebra $\mathscr{F}\otimes \widetilde{\mathscr{F}}$, and product measure $\P\times \widetilde{\P}$. In a slight abuse of notation, we will write $\widetilde{\E}$ for the expectation with respect to $\widetilde{P}$ alone, and write $\E$ denote expectation with respect to the full product measure $\P\times \widetilde{\P}$.

Define the stochastic Lagrangian flow $\phi^t_\kappa(x)$ to solve the SDE
\begin{align}
  \frac{\dee}{\dt} \phi^t_\kappa (x)  = u_t(\phi_\kappa^t(x)) + \sqrt{2\kappa}\dot{\widetilde W_t} \quad \phi^0_\kappa(x)  = x \, .
\end{align}
The fact that $u_t$ is incompressible implies that $x\mapsto \phi^t_\kappa$ is almost surely volume preserving. The solution $g_t$ to the advection diffusion equation \eqref{def:AD} is represented by this stochastic flow in the sense that
\begin{align}
g_t = \widetilde{\EE} g \circ (\phi^t_\kappa)^{-1}.
\end{align}
By incompressibiliy, it follows that for $f\in L^2$, $f : \T^d \to \R$, we have 
\begin{align}\label{eq:diffusivityRelnOutline}
\int g_t (x) f(x) \dx = \widetilde \EE \int g(x) f \big( \phi^t_\kappa(x) \big) \dx \, .
\end{align}

By choosing $f,g \in H^s$, the $H^{-s}$ decay of $g_t$ as in Theorem \ref{thm:UniMix}
follows once we deduce \eqref{eq:diffusivityRelnOutline} decays exponentially fast $\mathbb P$-a.e.. 
We will show this by obtaining $H^{-s}$ decay for observables advected by the Lagrangian flow
$\phi^t_\kappa$ for almost every $W_t, \widetilde W_t$-realization.
This, in turn, will be deduced using geometric ergodicity of the \emph{two-point process}
$(u_t, x_t^\kappa, y_t^\kappa)$ on $\Hbf \times \T^d \times \T^d$ defined
by $x_t^\kappa = \phi^t_\kappa(x), y_t^\kappa = \phi^t_\kappa(y)$ for $x,y \in \T^d$, $x\neq y$. 
Note that each of $x_t^\kappa, y_t^\kappa$ is driven by the \emph{same} noise paths $W_t$, $\widetilde{W}_t$.
Throughout, we write $x_t := x_t^0, y_t := y_t^0$ for two-point process when $\kappa = 0$.

The methodology of studying the two-point process 
follows our previous work \cite{BBPS19} on almost-sure $H^{-s}$ decay for Lagrangian 
flow in the absence of diffusivity (i.e., $\kappa = 0$), to which we
refer the reader for more detailed discussion and motivation
(see also \cite{dolgopyat2004sample, baxendale1988large}). 

Let us make these ideas more precise.
Let $P^{(2), \kappa}_t$ denote the Markov semigroup associated to the $\kappa$-two point process, that is, for measurable $\varphi:\Hbf \times \mathbb T^d \times \mathbb T^d \to \Real$, 
\begin{align}
P^{(2),\kappa}_t \varphi(u,x,y) = \EE_{(u,x,y)}\varphi(u_t,x_t,y_t) \, , 
\end{align}
whenever the RHS is defined. 
Define $\Dc = \{ (x,x) : x \in \T^d\} \subset \T^d \times \T^d$; in our setting, 
the complement $\Hbf \times \Dc^c$ is
the natural state space for the two-point process (see \cite{BBPS19} for a discussion of this point). 
Below, given a function $V : Z \to [1,\infty)$ on a metric space $Z$, we write $C_V$ the space of continuous observables $\phi:Z \to \Real$ such that
\[
\| \phi \|_{C_V} = \sup_{z \in Z} \frac{|\phi(z)|}{V(z)} \,  < \infty. 
\]
We will deduce Theorem \ref{thm:UniMix} from $\kappa$-uniform geometric ergodicity 
of the two-point process, stated precisely
below as Theorem \ref{thm:unifTwoPtMixing}. Its proof
occupies the majority of this paper, and is outlined in Sections \ref{sec:QHarris} -- \ref{sec:sktPsiPuni} below.
Note that this implies $\mu \times dx \times dx$ is the unique stationary measure for the two-point process on $\Hbf \times \Dc^c$. 
\begin{theorem}\label{thm:unifTwoPtMixing}
There exists $\kappa_0 > 0$ such that for all $\kappa \in [0,\kappa_0]$, there is a function $\mathcal V_\kappa : \Hbf \times \mathcal D^c \to [1,\infty)$
and $\kappa$-independent constants $C > 0, \gamma > 0$ such that for all $\psi \in C_{\mathcal{V}_\kappa}$ with $\int_{\Hbf \times \mathbb T^d \times \mathbb T^d} \psi(u,x,y) d\mu(u) dx dy = 0$, we have 
\[
| P^{(2), \kappa}_t \psi (u, x, y)| \leq C e^{- \gamma t} \mathcal V_\kappa(u, x, y) \norm{\psi}_{C_{\mathcal{V}^\kappa}}
\]
for all $t \geq 0, u \in \Hbf, (x,y) \in \Dc^c$. In general, the Lypaunov function $\Vc_\kappa$ depends on $\kappa$, but  
satisfies the following uniform-in-$\kappa$ estimate: for $\beta$ sufficiently large (independent of $\kappa$) 
and $\forall \eta > 0$, we have
\[
\iint \Vc_\kappa (u,x,y) \dx \dy \lesssim_\eta ( 1 + \| u \|_\Hbf^2)^\beta \exp\left(\eta\norm{u}_{\Wbf}^2\right)
\]
for all $u \in \Hbf$.
\end{theorem}


By repeating the Borel-Cantelli argument in Section 7 of \cite{BBPS19}, to which 
we refer the reader for details, Theorem \ref{thm:unifTwoPtMixing} implies the following $H^{-s}$ decay result uniformly in $\kappa$.
\begin{corollary} \label{cor:BC}
Let $\kappa \in [0,\kappa_0]$ and $\gamma, \beta, \eta > 0$ be as in Theorem \ref{thm:unifTwoPtMixing}. 
Fix $s, p > 0$. 
There exists a random constant $\tilde D_\kappa : \Omega \times \widetilde \Omega \times \Hbf \to [1,\infty)$ and $\gamma' \in (0,\gamma)$ (depending on $p$ and $s$, but not on $\kappa$)  such that for all $H^s$, mean zero scalars $f, g : \T^d \to \R$, we have
\[
\left | \int g (x)  f(\phi^t_\kappa(x)) \, \dx \right| \leq \tilde D_\kappa(\omega, \widetilde \omega, u) e^{- \gamma' t} \norm{f}_{H^s} \norm{g}_{H^s} \, 
\]
where the random constant $\tilde D_\kappa$ satisfies the moment estimate (uniformly in $\kappa$) for $\beta$ sufficiently large (independent of $u$, $p$, $\kappa$) and $\eta > 0$, 
\begin{align}\label{eq:momentEst}
\EE (\tilde D_\kappa(\cdot, \widetilde \cdot, u))^p \lesssim_{p,\eta} 
( 1 + \| u \|_\Hbf^2)^{\beta p} \exp\left(\eta \norm{u}_{\Wbf}^2\right) 
\end{align}
\end{corollary}

\begin{proof}[\textbf{Proof of Theorem \ref{thm:UniMix}  assuming Corollary \ref{cor:BC}} ]
Theorem \ref{thm:UniMix}  follows with $D_\kappa (u,\omega) := \widetilde \EE \tilde D_\kappa(\omega, \widetilde \cdot, u)$, since by \eqref{eq:diffusivityRelnOutline},  
\begin{align}
\left |  \int g_t (x) f(x)  \, \dx \right| =  \left |  \widetilde{\EE} \int g (x) f \big( \phi^t_\kappa(x)\big) \, \dx \right|
  & \leq \widetilde{\EE} D_\kappa(\omega, \widetilde \cdot, u) e^{- \gamma' t} \norm{f}_{H^s} \norm{g}_{H^s} \\ 
& = D_\kappa(\omega,u) e^{- \gamma' t} \norm{f}_{H^s} \norm{g}_{H^s} \, .
\end{align}
For fixed $u \in \Hbf$, moment estimates in $\EE$ for $D^\kappa$ follow from \eqref{eq:momentEst}
and Jensen's inequality with respect to $\widetilde \EE$. This completes the proof of Theorem \ref{thm:UniMix}.
\end{proof}

The rest of the paper is now dedicated to proving Theorem \ref{thm:unifTwoPtMixing} (with the exception of Section \ref{sec:ED}). 

\subsection{Uniform geometric ergodicity: a `quantitative' Harris's Theorem} \label{sec:QHarris}

To prove Theorem \ref{thm:unifTwoPtMixing}, we will run $P^{(2),\kappa}_t$ through the following mildly `quantitative' version of Harris's Theorem (Theorem \ref{thm:Harris})
 on geometric ergodicity for Markov chains, which keeps track of dependence of the constants appearing in the geometric decay of observables
in terms of the `inputs'. 
Since we use this result at several points throughout this paper, we state it below at a high level of generality.

Let $Z$ be a complete, separable metric space and $(z_n)$ a discrete-time Markov chain on $Z$ generating a Markov semigroup $\mathcal P^n$.
Geometric ergodicity of $(z_n)$ is usually proved by combining two properties: 
a \emph{minorization condition} which allows to couple trajectories initiated from a controlled
subset of phase space (sometimes called a \emph{small set}), 
and a \emph{drift condition} ensuring that trajectories visit this controlled 
subset with a high relative frequency. 

The latter can be formulated as follows: 
\begin{definition}[Drift condition] \label{defn:drift}
We say that a function $V : Z \to [1,\infty)$ satisfies 
a \emph{drift condition} for the $(z_n)$ chain if there exist constants $\gamma \in (0,1), K > 0$ for which
\[
\mathcal P V(z) \leq \gamma V(z) + K \, .
\]
\end{definition}
\noindent Functions $V$ satisfying Definition \ref{defn:drift} are commonly referred to as \emph{Lyapunov functions}.

Minorization in our context will be checked using the following standard result, regarding
suitably chosen sublevel sets $\{ V \leq R\}$ as our `controlled' regions of phase space. Here we also need to check dependence on parameters. 
\begin{proposition}[Quantitative minorization] \label{prop:minorize}
Let $V : Z \to [1,\infty)$ satisfy the drift condition with $\gamma,K$ as in Definition \ref{defn:drift} for the chain $(z_n)$. 
Assume that the Markov operator $\mathcal P$ is given as $\mathcal P = \mathcal P_{1/2} \circ \mathcal P_{1/2}$ for some Markov operator $\mathcal P_{1/2}$ satisfying the following two properties: 
\begin{itemize}
\item[(a)]  $\exists z_\ast \in Z$ such that $\forall \zeta > 0$, $\exists \eps > 0$ such that the following holds for all bounded, measurable $\phi:Z \to \R$: 
\begin{align}
\sup_{z \in B_{\eps(z^\ast)} }\abs{\mathcal P_{1/2} \phi(z) - \mathcal P_{1/2} \phi(z^\ast)} < \zeta. 
\end{align}
\item[(b)] Let $\eps := \eps$ be as in part (a) with $\zeta = \frac12$.
Suppose that there exists $R > 2K/(1-\gamma)$ and $\eta = \eta(R) > 0$ such that
\begin{align}
\inf_{z \in \{ V \leq R\}} \mathcal P_{1/2}(z,B_{\eps}(z_\ast)) > \eta >0 \, .
\end{align}
\end{itemize}
Then, the following minorization condition holds: for any $z_1, z_2 \in \{ V \leq R\}$,
we have that
\begin{align}\label{eq:minorize}
\| \mathcal P(z_1, \cdot) - \mathcal P(z_2, \cdot) \|_{TV} < \alpha \,  ,
\end{align}
where $\alpha := 1 - \frac{\eta}{2} \in (0,1)$. 
\end{proposition}
\begin{remark}
Note that condition (a) is commonly called \emph{strong Feller} at $z_\ast$ and condition (b) is called \emph{topological irreducibility}.  
\end{remark}

Crucially, Proposition \ref{prop:minorize} guarantees that the constants appearing in the minorization
condition \eqref{eq:minorize} are controlled by `inputs' $\epsilon, \eta(R) > 0$.
Verifying that these constants can be chosen independently of the diffusivity $\kappa > 0$ is one of the steps in our proofs below. 

Proposition \ref{prop:minorize} follows from standard arguments-- see, e.g., the proof of [Theorem 4.1, \cite{FM95}]. 
However, since quantitative dependence on parameters is of central importance in the
proof of our main results, for the sake of completeness 
we sketch the proof of Proposition \ref{prop:minorize} in Section \ref{sec:Minor}.



The following version of Harris's theorem below now describes geometric ergodicity for Markov
chains satisfying Definition \ref{defn:drift} and \eqref{eq:minorize}. Its proof is evident from a careful 
reading of any of the several proofs of 
Harris's theorem now available; see, e.g., the book of Meyn \& Tweedie \cite{meyn2012markov} or the proof of Hairer \& Mattingly \cite{hairer2011yet}. 

\begin{theorem}[Quantitative Harris's Theorem] \label{thm:Harris}
Assume that the Markov chain $(z_n)$ satisfies a drift condition with Lyapunov function $V$ in the sense of Definition \ref{defn:drift}, as well
as the conditions of Proposition \ref{prop:minorize}. Then, 
the Markov chain $(z_n)$ admits a unique invariant measure $\mu$ on $Z$ such that the following holds: 
there exists constants $C_0 > 0, \gamma_0 \in (0,1)$, depending only on $\gamma, K, \alpha, R$ as above, with the
property that
\[
\left| \mathcal P^n \psi(z) - \int \psi d \mu \right| \leq C_0 \gamma_0^n V(z) \| \psi \|_V
\]
for all $z \in Z, n \geq 0$ and $\psi : Z \to \R$ with $\| \psi \|_V < \infty$.
\end{theorem}

We note that there are many works studying quantitative dependence in Harris's
Theorem in a much more precise way; see, e.g., \cite{meyn1994computable, bakry2008rate, down1995exponential}. All we are using in this work is the comparatively simpler 
fact that the constants $C_0, \gamma_0$ can be uniformly controlled in terms of the drift and minorization parameters
$\gamma, K, \alpha, R$.

\subsection{Checking minorization for $P^{(2),\kappa}_T$} \label{sec:chkMin}
We intend to apply the quantitative Harris's Theorem (Theorem \ref{thm:Harris}) to $\mathcal P = P^{(2), \kappa}_T$ on $\Hbf \times \Dc^c$ for some fixed, $\kappa$-independent 
$T > 0$. This will imply Theorem \ref{thm:unifTwoPtMixing}. The most difficult step is the construction of the Lyapunov function $\Vc_\kappa$ satisfying Definition \ref{defn:drift} for $P^{(2), \kappa}_T$.
Before turning to this, however, let us indicate how the hypotheses of Proposition \ref{prop:minorize}
will be checked once a suitable $\Vc_\kappa$ has been constructed.

Generally speaking, Markov kernels may degenerate in some regions of state space, and so it is usually expected that minorization conditions such as \eqref{eq:minorize} only hold on certain subsets of state space bounded away from these degeneracies. Typically, then, the Lyapunov function $V$ is built so that suitable sublevel sets $\{ V \leq R\}$ avoid such degeneracies. 
In our setting, for the two point process on $\{(u,x,y) \in \Hbf \times \Dc^c\}$, Markov kernels degenerate in two places: where $\| u \|_\Hbf \gg 1$, and where $d(x, y) \ll 1$. The latter degeneracy is due to the fact that the set $\Dc = \{ (x, x) : x \in \T^d \} \subset \T^d \times \T^d$ is almost surely invariant for the two point process.
In view of these considerations, the following property is natural and ensures sublevel sets are bounded away from these degenerate regions of state space. 

 \begin{definition}\label{defn:coercive}
 We say that a $\kappa$-dependent family of functions 
 $\Vc_\kappa : \Hbf \times \T^d \times \T^d \to [1,\infty)$ is \emph{uniformly coercive} if $\forall  R > 0$, $\exists R' > 0$ (independent of $\kappa$) and $\exists \kappa_0 = \kappa_0(R) > 0$ such that $\forall \kappa \in (0,\kappa_0)$ the following holds 
 \[
 \{ \mathcal V_\kappa \leq R \} \subset \hat {\mathcal C}_{R'} := \{ \| u \|_{\Hbf} \leq R' \} \cap \{ d(x,y) \geq 1/R'\}. 
 \]
 \end{definition}
As long as the Lyapunov function $\Vc_\kappa$ in our drift condition is uniformly coercive, 
it suffices to check that the hypotheses of Proposition \ref{prop:minorize} (b) hold on a 'small' set of the form $\hat{\mathcal C}_R$ for a fixed $R$ sufficiently large relative only to the parameters $\gamma$, $K$ in Definition \ref{defn:drift} (both independent of $\kappa$). 
See Remark \ref{rmk:kapR} for more discussion.




We now turn to the task of verifying the hypotheses (a) and (b) of Proposition \ref{prop:minorize}. 
Item (a) is deduced from the following \emph{uniform strong Feller} regularity, 
which implies that minorization holds across balls of possibly small (yet $\kappa$-uniform) radius. 

\begin{lemma}[Uniform strong Feller] \label{lem:USF}
For all $T, R, \zeta > 0$, there exists $\epsilon = \epsilon(T,\zeta, R)$ (independent of $\kappa$) and there exists $\kappa_0 > 0$ such that the following holds for all $\kappa \in [0,\kappa_0]$. 
Let $\phi : \Hbf \times \mathcal D^c \to \R$ be an arbitrary bounded measurable function and let $z_\ast \in \hat{\mathcal C}_R$. Then, 
\begin{align}
 \sup_{z \in B_\eps(z_*)} \abs{P^{(2), \kappa}_T \phi(z) -  P^{(2), \kappa}_T \phi(z_*) } < \zeta. 
\end{align}
\end{lemma}
A straightforward adaptation of the methods in \cite{BBPS19} implies that for fixed $\kappa > 0$, the $\kappa$-two point process $P_T^{(2),\kappa}$ is
strong Feller, hence transition kernels vary continuously in the TV metric \cite{seidler2001note}.
Lemma \ref{lem:USF} is stronger, and is a kind of
 TV \emph{equicontinuity} for transition kernels, with uniform control on moduli 
of continuity in $\kappa \in [0,\kappa_0]$ and across the small sets $\hat {\mathcal C}_R, R > 0$.
The proof is essentially a careful re-examination of the methods in \cite{BBPS19} to keep track of 
dependence on the $\kappa$ parameter. A brief sketch is given in Section \ref{sec:USFetc}. 

Turning to hypothesis (b) in Proposition \ref{prop:minorize}: 
fix a reference point of the form
$z_\ast = (0,x_\ast, y_\ast) \in \Hbf \times \Dc^c$, where
$x_\ast, y_\ast \in \T^d$ are such that $d(x_\ast, y_\ast) > 1/10$. 
Fix $\eps = \eps (\zeta)$ for $\zeta = \frac12$ as in Lemma \ref{lem:USF}. 
Item (b) in Proposition \ref{prop:minorize} is checked at $z_\ast$ from the following.
\begin{lemma}[Uniform topological irreducibility]\label{lem:getControlled}
Let $T, R > 0$ be arbitrary, and let $\eps = \eps(T, \frac{1}{2}, R) > 0$ be as in Lemma \ref{lem:USF} with $\zeta = \frac{1}{2}$. 
Then, there exists $\kappa_0' = \kappa_0'(R, T), \eta = \eta(R,T)$ such that the following holds for all $\kappa \in [0,\kappa_0']$. 
For all $z = (u,x,y) \in \hat{\mathcal C}_R$, 
we have
\[
P_{T}^{(2), \kappa}(z, B_\epsilon(z_*)) \geq \eta \, .
\]
\end{lemma}
Note that in Lemma \ref{lem:getControlled}, the 
value of the upper bound $\kappa_0'$ depends  
on $\eps = \eps(T, 1/2, R)$, as well as $T$ and $R$. This is an artifact of the proof: 
since the primary case of interest is $\kappa \ll 1$, 
we treat the $\sqrt{\kappa} \widetilde W_t$ term as a perturbation and 
control trajectories exclusively with the $W_t$ noise applied to the velocity field process
(following the scheme set out for $\kappa = 0$ in [Proposition 2.7, \cite{BBPS19}]). 
A proof sketch in our setting is given in Section \ref{sec:USFetc}.


By Proposition \ref{prop:minorize}, Lemmata \ref{lem:USF} and \ref{lem:getControlled} imply the minorization condition as 
in Proposition \ref{prop:minorize} for $\mathcal P := P^{(2), \kappa}_{1}$ when we set $T = 1/2$.

\subsection{Drift condition for $P^{(2), \kappa}_T$} \label{sec:GeoP2kOut}
We now turn to the more significant task of deriving a drift condition with a Lyapunov function $\mathcal V_\kappa$ satisfying
the $\kappa$-uniform coercivity condition in Definition \ref{defn:coercive}.

The family of Lyapunov functions $\Vc_\kappa$ we construct for the two-point process 
will serve the role of bounding the dynamics away from the 'degnerate' regions $\| u \|_{\Hbf} \gg 1$ and
$d(x,y) \ll 1$. Control of the first is done entirely on the Navier-Stokes process $(u_t)$
as follows.


\begin{lemma}[Lemma 2.9, \cite{BBPS19}] \label{lem:Lyapu}
There exists $\mathcal Q > 0$, depending only on the noise coefficients $\{ q_m\}$ in the 
noise term $Q W_t$ and the dimension
$d$, with the following property. 
Let $0 < \eta < \eta^* = \nu / \mathcal Q$, $\beta \geq 0$, and define
\begin{align}
V_{\beta,\eta}(u) = (1 + \norm{u}_{\Hbf}^2)^{\beta}\exp\left(\eta \norm{u}_{\Wbf }^2 \right) \label{def:V}
\end{align}
where ${ \| \cdot \|_{\Wbf} }$ is as in \eqref{eq:vorticity-norm}. Then \eqref{def:V} satisfies the drift condition as in Definition \ref{defn:drift} for the $(u_t)$ process.
\end{lemma}
Lemma \ref{lem:Lyapu} is taken verbatim from \cite{BBPS19}. 
In fact, a more powerful estimate than that in Definition \ref{defn:drift} holds (a so-called super-Lyapunov property): see Lemma \ref{lem:TwistBd} in Section \ref{sec:NSEsup} for details. Obviously, these drift conditions do not depend on the $\kappa$ parameter, which only drives the Lagrangian flow itself. 

\subsubsection*{Motivation: controlling dynamics near $\mathcal D$}

To bound the dynamics away from small neighborhoods $\{ d(x,y) \ll 1\}$ of the diagonal, we
seek to build $\Vc_\kappa$ with an infinite singularity along $\Hbf \times \Dc$. 
We again follow our previous approach from \cite{BBPS19}, where
a Lypaunov function for $P^{(2)}_t$ at $\kappa = 0$ was built using the linearized approximation when
$x_t \approx y_t$. 
As proved in our earlier work \cite{BBPS18}, this linearization satisfies the following $\PP$-a.e.:
\begin{align}\label{eq:lyapExponent}
0 < \lambda_1 = \lim_{t \to \infty} \frac{1}{t} \log | D_x \phi^t| \qquad \text{ for all } (u,x) \in \Hbf \times \T^d \, , 
\end{align}
where the \emph{Lyapunov exponent} $\lambda_1 > 0$ is a (deterministic) constant independent of the initial $(u, x) \in \Hbf \times \T^d$.
This guarantees that nearby particles separate exponentially fast with high probability.

 With this intuition in mind, following the reasoning given in [Section 2 of \cite{BBPS19}], it is natural to seek a Lyapunov function of the form
$\mathcal V_\kappa = V_{\beta, \eta}(u) + h_{p, \kappa}(u,x,y)$, where 
 $h_{p, \kappa} (u,x,y) : \Hbf \times \mathcal D^c \to \R_{> 0}$ is of the form 
\begin{align}\label{eq:defineh}
h_{p, \kappa} (u,x,y) = \chi(|w|) d(x,y)^{-p} \psi_{p, \kappa} \left(u,x,\frac{w}{|w|}\right) \, 
\end{align}
for some $p > 0$. Here, $w = w(x,y)$ denotes the minimal 
displacement vector in $\R^d$ from $x$ to $y$, noting $|w| = d(x,y)$, and
$\chi : \R_{\geq 0}\to [0,1]$ is a smooth cutoff satisfying $\chi|_{[0,1/10]} \equiv 1$ and $\chi|_{[1/5, \infty)} \equiv 0$. 
We regard $\psi_{p, \kappa}$ as a function on the space $\Hbf \times P \T^d$, where $P \T^d = \T^d \times P^{d-1}$ is the projective bundle over $\T^d$. 

A natural candidate for $\psi_{p, \kappa}$ is (if it exists) the dominant, positive-valued eigenfunction 
of the `twisted' Markov semigroups $\hat P^{\kappa, p}_t$, defined 
for observables 
$\psi : \Hbf \times P \T^d \to \R$, by
\begin{align}\label{eq:feynmanFormulaOutline}
\hat{P}^{\kappa,p}_t \psi(u,x,v) = \EE_{(u,x,v)} \abs{D_x \phi_t^\kappa v}^{-p}\psi(u_t,x_t^\kappa,v_t^\kappa) \, , 
\end{align}
whenever the RHS exists. Here, for $\kappa > 0$, we let $(u_t, x_t^\kappa, v_t^\kappa)$ denote the \emph{projective process} 
\footnote{Equivalently, we can think of $(u_t, x_t^\kappa, v_t^\kappa)$ as evolving on the 
sphere bundle $\Hbf \times S \T^d$, where $S \T^d \cong \T^d \times \S^{d-1}$. In this parametrization,
$v_t^\kappa$ evolves according to the random ODE
\[
\dot v_t^\kappa  = (1 - v_t^\kappa \otimes v_t^\kappa) D u_t(x_t^\kappa) v_t^\kappa \, .
\]}
on $\Hbf \times P \T^d$: the \emph{one-point} process  $x_t^\kappa$ on $\T^d$ is as before, and $v_t^\kappa \in P^{d-1}$ is defined for initial $v \in P^{d-1}$ to be the projective representative of $D_x \phi^t_\kappa v$. We write $\hat P^\kappa_t$ for the $p = 0$ Markov semigroup corresponding to $(u_t, x_t^\kappa, v_t^\kappa)$. 

 In \cite{BBPS19}, we showed that for $\kappa = 0$, the dominant eigenfunction $\psi_{p,0}$ exists, is unique up to scaling, and satisfies $\hat P^{0,p}_t \psi_{p,0} = e^{- \Lambda(p,0) t} \psi_{p,\kappa}$ where $\Lambda(p,0) > 0$ for all $p$ sufficiently small-- in fact, $\Lambda(p,0) = p \lambda_1 + o(p^2)$, $\lambda_1$ as in \eqref{eq:lyapExponent}, and so our ability to build a drift condition is 
directly the result of a positive Lyapunov exponent (see also Remark \ref{rmk:momentLyapExponent}).
Once $\psi_{p, 0}$ was been constructed, a careful infinitesimal 
generator argument is then applied to pass from the linearized process
$(u_t, x_t, v_t)$ to nonlinear process
$(u_t, x_t, y_t)$ [Section 6.3 of \cite{BBPS19}].
In what remains we denote $\psi_{p} := \psi_{p,0}$, $\hat P^{p} := \hat P^{0,p}$, and $\Lambda(p) := \Lambda(p,0)$. 

In our context, we seek to show that the dominant eigenfunctions $\psi_{p, \kappa}$ for $\hat P_t^{\kappa, p}$, if they exist, result in analogous drift and 
uniform coercivity conditions 
with constants uniformly controlled in $\kappa$. 
The quality of these conditions depends on 
  (A) $\kappa$-uniform control on $\psi_{p, \kappa}$ from above and below, to ensure $\kappa$-uniform coercivity and to control error in the linearization approximation; and (B) a $\kappa$-uniform lower bound on the value $\Lambda(p,\kappa)$
  for which $\hat P^{\kappa, p}_t \psi_{p,\kappa} = e^{- \Lambda(p,\kappa) t}\psi_{p,\kappa}$,
 ensuring $\kappa$-uniform parameters in the resulting drift condition. 

 The primary challenge in achieving these points is the fact 
 that $\kappa \to \hat P^{\kappa, p}_t$ is a singular (not operator-
norm continuous) perturbation 
for $p \geq 0, t > 0$, and so it 
 is a subtle and technically challenging problem to obtain $\kappa$-uniform
control over $\psi_{p, \kappa}$ and $\Lambda(p, \kappa)$.
This is the aim of Proposition \ref{prop:psiPUniform} below, which summarizes the $\kappa$-uniform controlled needed
on this eigenproblem.
 
 \subsubsection*{Technical formulation of the eigenproblem for $\hat P^{\kappa, p}_t$}
 
 In what follows, $\beta, \eta > 0$ are fixed admissible parameters for Lemma \ref{lem:Lyapu}, and $V := V_{\beta, \eta}$. 
A finite number of times in the coming proofs, we will assume $\beta$ is taken sufficiently large, 
but always in a $\kappa$-independent way. 

We define
$C_V^1$ to be the set of Fr\'echet-differentiable observables $\psi : \Hbf \times P \T^d \to \R$ for which
\[
\| \psi \|_{C_V^1} := \| \psi\|_{C_V} + \sup_{(u,x,v) \in \Hbf \times P \T^d} \frac{\| D \psi(u,x,v)\|_{\Hbf^\ast}}{V(u)} < \infty \, ,
\]
where $\Hbf^\ast$ is shorthand for the dual space to $\Hbf \times T_{(x,v)} (P \T^d)$.

For reasons discussed in \cite{BBPS19} (see also, e.g., \cite{HM08}), for the 
purposes of $C_0$ semigroup theory one usually 
restricts to the following separable subspace of observables well-approximated by 
smooth, finite-dimensional observables. We define the (norm-closed)  subspace $\mathring C_V^1 \subset C_V^1$ to be the $C_V^1$-closure and $\mathring C_V \subset C_V$ to be the $C_V$-closure of the space of smooth cylinder functions
\begin{align}
\mathring C_0^\infty (\Hbf \times P \T^d) := \{ \psi | \psi(u,x,v) = \phi(\Pi_{\mathcal K} u, x, v), \mathcal K \subset \mathbb K, 
\phi \in C_0^\infty\} \, , \label{def:CylSpc}
\end{align}
where $\Pi_{\mathcal K}$ denotes the orthogonal projection onto $\Hbf_{\mathcal K} \cong \R^{\abs{\mathcal K}}$.

The following statement lists all required properties of the dominant eigenfunctions
for $\hat P^{\kappa, p}_t$ under the singular perturbation $\kappa \to 0$.
The result is crucial to our method for dealing with this singularity and its proof occupies a substantial
portion of the paper. The proof is outlined in Section \ref{sec:sktPsiPuni} below.

\begin{proposition}\label{prop:psiPUniform}
There exist $\kappa_0, p_0 > 0$ for which the following holds. 
\begin{itemize}
\item[(a)] There exists $T_0 > 0$ such that for all $(\kappa, p) \in [0,\kappa_0] \times [0,p_0]$, 
the (positive) operator $\hat P^{\kappa, p}_{T_0}$ admits a simple, dominant, isolated, positive, real eigenvalue 
$e^{- T_0 \Lambda(p, \kappa)}$ in $\mathring C_V^1$ such that $\Lambda(p, \kappa) > 0$, and have the following
property: for each fixed $p > 0$ 
\begin{align}\label{eq:unifLambdaBoundsOutline}
\lim_{\kappa \to 0}\Lambda(p,\kappa) = \Lambda(p) > 0. 
\end{align}
\item[(b)]
With $\pi_{p,\kappa}$ denoting the (rank 1) spectral projector corresponding to
the dominant eigenvalue of $\hat P^{\kappa,p}_{T_0}$, let $\psi_{p,\kappa} = \pi_{p,\kappa}({\bf 1})$, where
${\bf 1}$ denotes the unit constant function on $\Hbf \times P \T^d$. The family $\{ \psi_{p,\kappa}\}$ has the following
properties.
	\begin{itemize}
		\item[(i)] For all $t > 0$, we have
		\[
		\hat P^{p, \kappa}_t \psi_{p, \kappa} = e^{- \Lambda(p, \kappa) t} \psi_{p, \kappa} \, .
		\]
		\item[(ii)] We have $\psi_{p,\kappa} \in \mathring C_V^1$, with $\| \psi_{p,\kappa}\|_{C_V^1}$ bounded
		from above uniformly in $\kappa, p$.
      \item[(iii)]  For all $p,\kappa$ sufficiently small, $\psi_{p,\kappa} \geq 0$ and there holds the convergence
       \begin{align}
         \lim_{\kappa \to 0} \norm{\psi_{p,\kappa} - \psi_p}_{C_V} = 0 \,. 
       \end{align}
       Finally, for $p$ sufficiently small, $\forall R > 0$, $\exists \kappa_0 = \kappa_0(R)$ such that  
		\[
		\inf_{\kappa \in [0,\kappa_0]} \inf_{\substack{(u,x,y) \in \Hbf \times P \T^d \\ \| u \|_\Hbf \leq R}} \psi_{p, \kappa}(u,x,v) > 0 \, .
        \]
	\end{itemize}
\end{itemize}
\end{proposition}

\begin{remark}\label{rmk:momentLyapExponent}
The value $\Lambda(p, \kappa)$ is referred to as the \emph{moment Lyapunov exponent}
in the random dynamical systems literature \cite{arnold1984formula}, 
and governs large deviation-scale fluctuations in the 
convergence of Lyapunov exponents. Indeed, $\hat P_t^{\kappa, p}$ is the
Feynman-Kac semigroup \cite{varadhan1984large} with respect to the potential $H(u,x,v) = \langle v, D u(x) v\rangle$; 
see \eqref{FKRep}. 

As in [Lemma 5.8 of \cite{BBPS19}], one can show that
\[
\Lambda(p, \kappa) = - \lim_{t \to \infty} \frac{1}{t} \log \E |D_x \phi^t v|^{-p}
\]
holds for all initial $u \in \Hbf$ and $(x,v) \in P \T^d$. This, in turn, implies the asymptotic
\[
\Lambda(p, \kappa) = p \lambda_1^\kappa + o(p)
\]
where $\lambda_1^\kappa$ is the Lyapunov exponent
\[
\lambda_1^\kappa = \lim_{t \to \infty} \frac{1}{t} \log |D_x \phi_\kappa^t| 
\]
for the $\kappa$-driven Lagrangian flow $\phi^t_\kappa$. We remark that
what is really needed for our purposes is continuity of $\kappa \mapsto \Lambda(p, \kappa)$,
which does not follow from continuity of $\kappa \mapsto \lambda^1_\kappa$ (the latter
is a straightforward corollary of Proposition \ref{prop:outlineSpecPic} (a) applied
to $p = 0$). 
\end{remark}

\subsection{Proof outline of Proposition \ref{prop:psiPUniform}} \label{sec:sktPsiPuni}

The proof has two main components. The first is to establish
spectral properties of the semigroups $\hat P_t^{\kappa, p}$ by viewing these, for \emph{fixed} $\kappa > 0$, 
 as norm-continuous perturbations in the parameter $p > 0$ of the semigroups $\hat P_t^\kappa$. 
This part of the proof is a careful re-working of the arguments in \cite{BBPS19} to ensure
that the relevant quantities do not depend on the parameter $\kappa$. The following is a summary
of the spectral picture derived.

\begin{proposition}\label{prop:outlineSpecPic}
There exist $\kappa_0, p_0, T_0 > 0, c_0 \in (0,1)$ such that the following holds for any $\kappa \in [0,\kappa_0], p \in [0,p_0]$. 
\begin{itemize}
\item[(a)] The semigroup $\hat P_t^{\kappa, p}$ is a $C_0$-semigroup on $\mathring C_V$. For any fixed $t > 0$, the norm $\| \hat P_t^{\kappa, p}\|_{C_V}$ is bounded uniformly in $\kappa$. Additionally, for any $t > 0$, the operator $\hat P_t^{\kappa, p}$ has a 
simple, dominant, isolated eigenvalue $e^{- \Lambda(p, \kappa) t}$, and satisfies
\begin{align}\label{eq:specPicCVOutline}
\sigma(\hat P_t^{\kappa, p}) \setminus \{ e^{- \Lambda(p, \kappa) t}\} \subset B_{c_0^t}(0) \, .
\end{align}
\item[(b)] We have that $\hat P_{T_0}^{\kappa, p}$ is a bounded linear operator $C_V^1 \to C_V^1$ sending
$\mathring C_V^1$ into itself, with $\| \hat P^{\kappa, p}_{T_0}\|_{C_V^1}$ bounded uniformly in $\kappa$. Regarded as an operator in this space, the value $e^{- \Lambda(p, \kappa) T_0}$ 
is a simple, dominant, isolated eigenvalue for $\hat P_{T_0}^{\kappa, p}$, and satisfies
\[
\sigma(\hat P_{T_0}^{\kappa, p}) \setminus \{ e^{- \Lambda(p, \kappa) T_0}\} \subset B_{c_0^{T_0}}(0) \, .
\]
\end{itemize}
\end{proposition} 


\subsubsection{Proof of Proposition \ref{prop:outlineSpecPic} following \cite{BBPS19}}

We provide a brief sketch of the arguments and highlight where one must be most careful about $\kappa$-dependence. 
Basic properties, such as $C_0$ continuity on $\mathring C_V$ and uniform bounds in the $C_V$ and $C_V^1$ norms follow essentially the same as those in \cite{BBPS19}; see Section \ref{sec:BasicProj} for more details. 

At $p = 0$, the uniform spectral picture for $\hat P_t^{\kappa}$ in $C_V$
is derived by applying the quantitative Harris theorem (Theorem \ref{thm:Harris}) to the projective process
 $(u_t, x_t^\kappa, v_t^\kappa)$. A $\kappa$-uniform spectral gap follows by verifing the minorization and drift conditions
 with constants independent of $\kappa > 0$.
 Since the $P \T^d$ factor is compact, it suffices to use $V = V_{\beta, \eta}$ as the Lyapunov function in Definition \ref{defn:drift} (via Lemma \ref{lem:Lyapu}). 
 The only thing to check here is the minorization condition using Proposition \ref{prop:minorize}. The following is sufficient for our purposes.
See Sections \ref{sec:USF} and \ref{sec:UTI} for sketches of parts (a) and (b) respectively. 

\begin{proposition}\label{prop:unifMinorCondProjective} \
\begin{itemize}
\item[(a)] (Uniform strong Feller) For all $\kappa$ sufficiently small, the following holds. 
For any $\zeta > 0$ there exists $\eps = \eps(\zeta, R) > 0$, independent of $\kappa$,
so that for all bounded measurable $\phi : \Hbf \times P \T^d \to \R$ and $(u,x,v) \in \Hbf \times P \T^d$, $\| u \|_\Hbf \leq R$,
we have
\[
\sup_{(u', x', v') \in B_\eps(u,x,v)} \left| \hat P^{\kappa}_1 \phi(u,x,v) - \hat P^\kappa_1 (u',x',v')\right| < \zeta \, .
\]
\item[(b)] (Uniform topological irreducibility) Fix $\zeta = 1/2$ and let $\eps = \eps(\frac12, R)$ be as in part (a). 
Fix a reference point $(0,x_\ast,v_\ast) \in \Hbf \times P \T^d$. 
Then, there exists $\kappa_0''= \kappa_0''(\eps, R), \eta = \eta(\eps, R) > 0$ so that for all $\kappa \in [0,\kappa_0'']$, the following holds:
for all $(u,x,v) \in \Hbf \times P \T^d, \| u \|_\Hbf \leq R$, we have
\[
\hat P_1^{\kappa}((u,x,v), B_\eps(0,x_\ast,v_\ast) ) \geq \eta
\]
\end{itemize}
\end{proposition}

Having verified the uniform spectral gaps for $\hat P_t^\kappa$ semigroup,
the proof of Proposition \ref{prop:outlineSpecPic} (a) is completed using a spectral perturbation argument carried out in Section \ref{subsubsec:specPicCV} and the convergence 
\[
\lim_{p \to 0} \sup_{\kappa \in [0,\kappa_0]} \| \hat P_t^{\kappa, p} - \hat P_t^\kappa\|_{C_V \to C_V} = 0
\]
for any fixed $t > 0$ (see Lemma \ref{lem:unifConvCVPto0}).


Next, we sketch the proof of Proposition \ref{prop:outlineSpecPic} (b). Checking $\kappa$-uniform boundedness in $C_V^1$ and propagation of $\mathring C_V^1$ again proceeds 
more-or-less verbatim from arguments in \cite{BBPS19}; see Section \ref{sec:BasicProj} for more details. As in \cite{BBPS19}, we are only able to show $\hat P_t^{\kappa, p}$ is bounded in $C_V^1$ for $t \geq T_0$ ($T_0 > 0$ a $\kappa$-independent constant), which is why we state the $C_V^1$ spectral picture for $\hat P_{T_0}^{\kappa, p}$. Following a standard argument in [Proposition 4.7; \cite{BBPS19}], the $\kappa$-uniform spectral gap in $C^1_V$ is obtained from the $C_V$ spectral gap from Proposition \ref{prop:outlineSpecPic}(a) and the following $\kappa$-uniform 
gradient-type bound similar to those pioneered by Hairer \& Mattingly \cite{HM06,HM11} for ergodicity with degenerate noise. 

\begin{lemma}[Uniform Lasota-Yorke regularity]\label{lem:unifGradBound}
There exists $\kappa_0$ such that the following holds uniformly in $\kappa \in [0,\kappa_0]$. 
For all $\beta' \geq 2$ sufficiently large and all admissible $\eta' > 0$ for Lemma \ref{lem:Lyapu},
there exist $C_1 > 0, \varkappa > 0$ such that the following holds for all $\kappa \in [0,\kappa_0]$.
For all $\psi \in C_V$ and $t > 0$, we have
\[
\| D\widehat P_t^\kappa \psi \|_{\Hbf^*} \leq C_1 V_{\beta', \eta'} \left( \sqrt{\widehat P_t^\kappa | \psi|^2} + e^{-\varkappa t}\sqrt{\widehat P_t^\kappa\|D\psi\|_{\Hbf^*}^2} \right) 
\]
pointwise on $\Hbf \times P \T^d$.
\end{lemma}
The proof of Lemma \ref{lem:unifGradBound} is analogous to that in [Proposition 4.6; \cite{BBPS19}]; we provide a sketch in Section \ref{subsec:CV1proj} below.


Finally, the $\kappa$-uniform spectral gap in $\mathring C_V^1$ for $\hat P^{\kappa, p}_{T_0}$ is obtained by a spectral perturbation argument (see Section \ref{sec:CV1twist}) and the fact that
\[
\lim_{p \to 0} \sup_{\kappa \in [0,\kappa_0]} \| \hat P_{T_0}^{\kappa, p} - \hat P_{T_0}^\kappa \|_{C_V^1 \to C_V^1} = 0. 
\]
This completes the proof of Proposition \ref{prop:outlineSpecPic}; see Section \ref{sec:prfSpecPic} for more details.  

\subsubsection{Overcoming the singular perturbation $\kappa \mapsto \hat P_t^{\kappa, p}$}\label{subsubsec:singular}

We now move on to completing the proof of Proposition \ref{prop:psiPUniform}, 
which requires that we contend with the potentially singular
nature of $\kappa \mapsto \hat P_t^{\kappa, p}$. 
This is a significant deviation from our previous work \cite{BBPS19}, which
considers only the $\kappa = 0$ case. 

More precisely, the mapping $\kappa \mapsto \hat P_t^{\kappa, p}$ 
is not, to the best of our knowledge, continuous with respect to the
 operator norm derived from any of the usual 
topologies on observables $\psi : \Hbf \times P \T^d \to \R$. 
From the perspective of smooth dynamics, this is unsurprising. For deterministic
maps, Markov semigroups on observables are called Koopman operators, and 
for parametrized families of (deterministic) maps, 
these Koopman operators typically vary discontinuously in the parameter 
with respect to most useful operator norms. 
For an example related to our setting, where the parameter dictates the amplitude of noise, 
see \cite{baladi2002almost}. 

At least, we have the following strong operator continuity: 

\begin{lemma}
Assume $\beta > 0$ to be taken sufficiently large. 
There exists $p_0 > 0$ such that the following holds for any $\psi \in C_V( \Hbf \times P \T^d)$:
\begin{align}\label{ineq:SOTintro}
\lim_{\kappa \to 0} \sup_{p \in [- p_0, p_0]} \| \hat P_t^{p, \kappa} \psi - \hat P_t^p \psi \|_{C_V} = 0 \, .
\end{align}
\end{lemma}
\noindent For proof, see Lemma \ref{lem:SOT-convergence}.

The continuity in \eqref{ineq:SOTintro} is not strong enough to immediately extend the $\hat{P}_t^p$ spectral gap to a $\kappa$-uniform spectral gap on $\hat{P}_t^{\kappa,p}$. 
In order to leverage \eqref{ineq:SOTintro}, we instead pass to the limit in the eigenfunction/value problem. 
To roughly summarize: estimates on dominant spectral projectors (Lemma \ref{cor:specProjUnifCTRL}) 
and arguments using the scale of compactly-embedded spaces $\Hbf^{\sigma'}$ and the uniform $C^1_V$ estimates
imply that $\set{\psi_{p,\kappa}}_{\kappa \in (0,1)}$ is 
 suitably `locally sequentially pre-compact' in $C_V$ using a version of Arzela-Ascoli (Lemma \ref{lem:compact}). This pre-compactness together with \eqref{ineq:SOTintro} ultimately allows to pass to the limit in the eigenvalue problem 
 $\hat P^{\kappa, p}_t \psi_{\kappa, p} = e^{- t \Lambda(p, \kappa)} \psi_{\kappa, p}$, obtaining the following.

\begin{proposition}\label{prop:convSpecObjectsOutline}
Let $p \in [0,p_0]$ be fixed. Then,
\begin{gather}\label{eq:ctyLambdaOutline}
\lim_{\kappa \to 0}\| \psi_{p, \kappa} - \psi_{p}\|_{C_V} = 0 \quad \text{ and }\quad 
\lim_{\kappa \to 0}\Lambda(p, \kappa) = \Lambda(p). 
\end{gather}
\end{proposition}
See Section \ref{sec:ProofConv} for the detailed proof. With Proposition \ref{prop:convSpecObjectsOutline} in hand, it is now straightforward to check the remaining items in Proposition \ref{prop:psiPUniform}; see Section \ref{sec:ProofConv} for such details.

\subsubsection*{Verifying the drift condition: infinitesimal generator argument}

Assuming Proposition \ref{prop:psiPUniform}, let us sketch how the drift condition for the `nonlinear'
$(u_t, x_t^\kappa, y_t^\kappa)$ process is derived, thereby completing the proof of 
Theorem \ref{thm:unifTwoPtMixing}. 
 Let $p \in (0,p_0]$ be fixed once and for all, and let $\kappa > 0$ be sufficiently small so that, as in Proposition \ref{prop:psiPUniform}(a), we have $\Lambda(p, \kappa) \geq \frac12 \Lambda(p, 0)$ uniformly in $\kappa$. 
Our Lyapunov function $\Vc_\kappa$ is of the form
\begin{align}
\Vc_\kappa(u,x,y) = V_{\beta + 1, \eta}(u) + h_{p, \kappa}(u,x,y) \, , \label{def:Vc}
\end{align}
where $h_{p, \kappa}$ is as in \eqref{eq:defineh}. 
Observe that Proposition \ref{prop:psiPUniform} (b)(iii) ensures that $\Vc_\kappa$ as above is
 uniformly coercive as in Definition \ref{defn:coercive}.

\newcommand{\Lc}{\mathcal L}

To conclude the drift condition for $\Vc_\kappa$ as in Definition \ref{defn:drift}, we 
apply the analogue of the infinitesimal generator argument used for the $\kappa = 0$
case in \cite{BBPS19}, again
carefully ensuring $\kappa$-independence of relevant quantities. 
Brushing aside details for the moment, for $\kappa \geq 0$ let $\Lc_{(2), \kappa}$ denote
the (formal) infinitesimal generator of the $(u_t, x_t^\kappa, y_t^\kappa)$ process. We show that in fact $h_{p, \kappa}$ is in the domain of this generator, and that
\[
\Lc_{(2), \kappa} h_{p, \kappa} \leq - \Lambda(p, \kappa) h_{p, \kappa} + C_0 V_{\beta + 1, \eta}.  
\]
The first term is good and reflects the strong exponential separation of nearby trajectories (equivalently, repulsion
from the diagonal), while the second is an error arising from the linearized approximation of the velocity field (the constant $C_0$ being independent of $\kappa$).
This uniform control in the linearization error makes critical use of the uniform $C_V^1$ control on $\psi_{p, \kappa}$ as in Proposition 
\ref{prop:psiPUniform}(b)(ii), while verifying that $\psi_{p, \kappa}$ is in the domain of $\Lc_{(2), \kappa}$ uses $\psi_{p, \kappa} \in \mathring C_V^1$ and
Proposition \ref{prop:psiPUniform} (b)(i). See Section \ref{sec:Udrift} where this argument is carried out in more detail.

The linearization error is overcome as follows: formally, 
a stronger version of the drift condition for $V_{\beta + 1, \eta}$ (see Remark \ref{rmk:SuperL})
implies that 
for any $\xi > 0$ there exists $C_\xi > 0$ such that
\[
\Lc V_{\beta + 1, \eta} \leq - \xi V_{\beta + 1, \eta} + C_\xi \, ,
\]
where $\Lc$ is the generator of the $(u_t)$ process (we do not justify this inequality precisely as written, but instead an integrated version that is almost equivalent; for details, see Section \ref{sec:Udrift} below  and the proof of [Proposition 2.13; \cite{BBPS19}]). Taking $\xi \geq C_0 + \frac12 \Lambda(p, 0)$ 
ensures that the $- \xi V_{\beta + 1, \eta}$ term successfully absorbs the linearization error
$C_0 V_{\beta + 1, \eta}$, verifying the desired drift condition.
With this established, Theorems \ref{thm:unifTwoPtMixing} and \ref{thm:UniMix} now follow.
See Proposition \ref{prop:drift} in Section \ref{sec:Udrift} for mathematical details. 

\begin{remark}[Setting the parameters ]\label{rmk:kapR}
Let us lastly point out how to set parameters consistently in a non-circular manner. 
  Notice that Proposition \ref{prop:psiPUniform} (b) (iii) has the same ordering in the quantifiers of $R$ and $\kappa$ as Definition \ref{defn:coercive}. We choose parameters like this: first we fix $p,\kappa$ small to obtain a $\kappa$-independent drift condition for $\Vc_{\kappa}$ as defined in \eqref{def:Vc} -- that is, \eqref{def:Vc} satisfies Definition \ref{defn:drift} for $\gamma$, $K$ both independent of $\kappa$. Then, $\Vc_\kappa$ satisfies Definition \ref{defn:coercive} by Proposition \ref{prop:psiPUniform} (b) (iii). Then, choose $R$ sufficiently large to satisfy Proposition \ref{prop:minorize} based on these parameters. Then, chooose $\kappa_0$ sufficiently small based on Definition \ref{defn:coercive} and Lemma \ref{lem:getControlled} of Section \ref{sec:chkMin} to obtain minorization. 
\end{remark}

\subsection{Notation} \label{sec:Note} 

We use the notation $f \lesssim g$ if there exists a constant $C > 0$ such that $f \leq Cg$ where $C$ is independent of the parameters of interest. Sometimes we use the notation $f \approx_{a,b,c,...} g$ to emphasize the dependence of the implicit constant on the parameters, e.g. $C = C(a,b,c,...)$. We denote $f \approx g$ if $f \lesssim g$ and $g \lesssim f$.
In this work, such implicit constants \emph{never depend on $\omega$, $\kappa$, $t$, $(u_t)$ (the velocity), or $(g_t)$ (the passive scalar)}.

Throughout, $\R^d$ is endowed with the standard Euclidean inner product $(\cdot, \cdot)$ and corresponding norm $|\cdot|$. We continue
to write $|\cdot|$ for the corresponding matrix norm. When the domain of the $L^p$ space is omitted it is understood to be $\T^d$: $\norm{f}_{L^p} = \norm{f}_{L^p(\T^d)}$. We use the notations $\EE X = \int_{\Omega} X(\omega) \PP(d\omega)$ and $\norm{X}_{L^p(\Omega)} = \left(\EE \abs{X}^p \right)^{1/p}$. When $(z_t)$ is a Markov process, we write $\EE_z, \P_z$ for the expectation and probability, respectively, conditioned on the event $z_0 = z$. We use the notation $\norm{f}_{H^s} = \sum_{k \in \Z^d} \abs{k}^{2s} \abs{\hat{f}(k)}^2$ (denoting $\hat{f}(k) = \frac{1}{(2\pi)^{d/2}} \int_{\T^d} e^{-ik\cdot x}f(x) dx$ the usual complex Fourier transform).  We occasionally use Fourier multiplier notation $\widehat{m(\grad) f}(\xi) := m(i\xi) \hat{f}(\xi)$. Additionally, we will often use $r_0$ to denote a number in $(\frac{d}{2}+1,3)$ such that the Sobolev embedding $H^{r_0} \hookrightarrow W^{1,\infty}$ holds.

We denote $P \T^d \cong \T^d \times P^{d-1}$  for projective bundle.
We often abbreviate $T_v P\T^d = T_{(x,v)} P \T^d$ for the tangent space of $P \T^d$ at $(x,v)$ as the $\T^d$ factor is flat. 
We are often working with the Hilbert spaces $\Wbf \times T_v P \T^d$ and $\Hbf \times T_v P \T^d$.
For these spaces we denote the inner product $\brak{\cdot,\cdot}_{\Wbf}$ (respectively $\Hbf$) and correspondingly for the norms as the finite-dimensional contribution to the inner product is unambiguous. For linear operators $A: \Wbf \times T_v P \T^d \to \Wbf \times T_v P \T^d$ we similarly denote the operator norm $\norm{A}_{\Wbf}$ and for linear operators $A:\Wbf \times T_v P\T^d \to \Real$ we use the notation $\norm{A}_{\Wbf^\ast}$ (analogously for $\Hbf$).
For $\mathcal{K} \subset \mathbb K$, define $\Pi_{\mathcal{K}}: \Wbf \times P \T^d \to \mathcal{K} \times P\T^d$ to be the orthogonal projection onto the subset of modes in $\mathcal{K}$.
For $n \in \mathbb N$, $\Pi_n$ denotes the orthogonal projection onto the modes with $k \in \mathbb K$, $\abs{k} \leq n$.

\section{Preliminaries}\label{sec:Prelim}

\subsection{Proof of Proposition \ref{prop:minorize}} \label{sec:Minor}






For completeness, we provide a proof of our criterion for minorization, Proposition \ref{prop:minorize}. 
\begin{proof}[\textbf{Proof of Proposition \ref{prop:minorize}}]
Let $z_1,z_2 \in \set{V \leq R}$ be as in the statement, and 
let $z_\ast$ be as in hypothesis (a) of Proposition \ref{prop:minorize}. 
Fix $\zeta = \frac{1}{2}$ and the corresponding value of $\eps$ as in hypothesis (b). 
By hypothesis (b), we have
\[
	\mathcal P_{1/2}(z_i,A) \geq \eta \,\hat{\nu}_{z_i}(A) \, , \quad \hat \nu_{z_i}(A) := \frac{\mathcal P_{1/2} (z_i,A\cap B_\ep(z_*))}{\mathcal P_{1/2}(z_i,B_\ep(z_*))}
\]
Consequently we can write $\mathcal P_{1/2}(z_i,\cdot)$ as a convex combination of probability measures 
\[
	\mathcal P_{1/2}(z_i,\cdot) = \eta \hat{\nu}_{z_i}(\cdot) + (1-\eta)\tilde{\nu}_{z_i}(\cdot) \, .
\]
Using $\mathcal P(x,A) = \int_Z \mathcal P_{1/2}(y,A) \mathcal P_{1/2}(x,dy)$, we estimate
\begin{equation}
\begin{aligned}
	&\abs{ \mathcal P(z_1,A) - \mathcal P(z_2,A)} \\
	&\hspace{.5in}\leq \eta \int_{B_\ep(z_*)}\int_{B_{\ep}(z_*)} |\mathcal P_{1/2}(w_1,A) - \mathcal P_{1/2}(w_2,A)|\hat \nu_{z_1}(\dee w_1)\hat \nu_{z_2}(\dee w_2) + (1-\eta) \\
\end{aligned}
\end{equation}
Using hypothesis (a) and our choice of $\eps$, there holds
\begin{equation}
\begin{aligned}
	&\abs{ \mathcal P(z_1,A) - \mathcal P(z_2,A)} & \leq 1-\frac{\eta}{2},
\end{aligned}
\end{equation}
which provides the desired minorization with $\alpha = 1 - \frac{\eta}{2}$. 
\end{proof}

\subsection{Stochastic Navier-Stokes and the super-Lyapunov property} \label{sec:NSEsup}

Following the convention used in \cite{E2001-lg,BBPS18,BBPS19}, we define a natural real Fourier basis on $\Lbf^2$ by defining for each $m = (k,i) \in \mathbb{K} := \Z^d_0 \times \{1,\ldots,d-1\}$
 \[\label{eq:Fourier-Basis}
 e_m(x) = \begin{cases}
 c_d\gamma_k^i\sin(k\cdot x), \quad& k \in \Z^d_+\\
 c_d\gamma_k^i\cos(k\cdot x),\quad& k\in \Z^d_-,
 \end{cases}
 \]
where $\Z^d_0 := \Z^d \setminus \set{0,\ldots, 0}$, $\Z_+^d = \{k\in \Z^d_0 : k^{(d)} >0\}\cup\{k\in \Z^d_0 \,:\, k^{(1)}>0, k^{(d)}=0\}$ and $\Z_-^d = - \Z_+^d$, and for each $k\in \Z_0^d$, $\{\gamma_k^i\}_{i=1}^{d-1}$ is a set of ${d-1}$ orthonormal vectors spanning the plane perpendicular to $k \in \R^d$ with the property that $\gamma_{-k}^i = - \gamma_{k}^i$. The constant $c_d = \sqrt{2}(2\pi)^{-d/2}$ is a normalization factor so that $e_m(x)$ are a complete orthonormal basis of $\Lbf^2$. Note that in dimension $d=2$  $\mathbb{K} = \Z^d_0$, hence $\gamma_k^1 = \gamma_k$ is just a vector in $\R^2$ perpendicular to $k$ and is therefore given by $\gamma_k = \pm k^{\perp}/|k|$. We assume that $Q$ can be diagonalized with respect to $\{e_m\}$ with eigenvalues $\{q_m\}\in \ell^2(\mathbb{K})$ defined by
\[
Qe_m = q_m e_m,\quad m = (k,i)\in \mathbb{K}.
\]
Note that Assumption \ref{a:Highs} is equivalent to
\[
	|q_m| \approx |k|^{-\alpha},\quad m = (k,i)
\]
We will write the Navier-Stokes system as an abstract evolution equation on $\Hbf$ by 
\begin{equation}\label{eq:NS-Abstract}
	\partial_t u + B(u,u) + Au = Q\dot{W} = \sum_{m\in\mathbb{K}} q_m e_m \dot{W}^m, 
  \end{equation}
where
\begin{align}
B(u,v) &= \left(\Id - \grad (-\Delta)^{-1} \grad \cdot \right)\grad \cdot (u \otimes v) \\
Au &= \begin{cases} -\nu \Delta u \quad&  \textup{ if } d=2 \\
-\nu' \Delta u + \nu\Delta^2 u \quad& \textup{ if } d=3. 
\end{cases}
\end{align}
The $(u_t)$ process with initial data $u$ is defined as the solution to \eqref{eq:NS-Abstract} in the mild sense \cite{KS,DPZ96}:
\begin{align}
u_t = e^{-tA}u - \int_0^t e^{-(t-s)A} B(u_s,u_s) ds + \underbrace{\int_0^t e^{-(t-s)A} Q \dee W(s)}_{=: \Gamma_t} \, ,  \label{eq:Mild} 
\end{align}
where the above identity holds $\P$ almost surely for all $t>0$. The random process $\Gamma_t$ is referred
to as the \emph{stochastic convolution} for this additive SPDE. For \eqref{eq:Mild}, we have the following well-posedness theorem.
\begin{proposition}[\cite{KS,DPZ96}] \label{prop:WPapp}
For each of Systems \ref{sys:NSE}--\ref{sys:3DNSE}, we have the following.
 For all initial $u \in \Hbf \cap \Hbf^{\sigma'}$ with $\sigma' < \alpha-\frac{d}{2}$ and all $T> 0, p \geq 1$, there exists a $\P$-a.s. unique solution $(u_t)$ to \eqref{eq:Mild} which is $\mathscr{F}_t$-adapted, and belongs to $L^p(\Omega;C([0,T];\Hbf \cap \Hbf^{\sigma'})) \cap L^2(\Omega;L^2(0,T;\Hbf^{\sigma'+(d-1)}))$. 

Additionally, for all $p \geq 1$ and $0 \leq \sigma' < \sigma'' < \alpha - \frac{d}{2}$,  
\begin{align}
\EE_u \sup_{t \in [0,T]} \norm{u_t}_{\Hbf^{\sigma'}}^p & \lesssim_{T,p,\sigma'} 1 + \norm{u}_{\Hbf \cap \Hbf^{\sigma'}}^p \\
\EE_u \int_0^T \norm{u_s}_{\Hbf^{\sigma' + (d-1)}}^2 ds & \lesssim_{T,\delta} 1 + \norm{u}^2_{\Hbf^{\sigma'}} \\ 
\EE_u \sup_{t \in [0,T]} \left(t^{\frac{\sigma''-\sigma'}{2(d-1)}} \norm{u_t}_{\Hbf^{\sigma''}}\right)^p &\lesssim_{p,T,\sigma',\sigma''} 1 + \norm{u}^p_{\Hbf^{\sigma'}}. \label{ineq:locRegu}
\end{align}
\end{proposition}

We now state a precise version of the super-Lyapunov property for the
drift functions $V_{\beta, \eta}(u)  := (1 + \| u \|^2_{\Hbf})^\beta \exp( \eta \| u \|_{\Wbf}^2)$. 
If $d=2$ define $\mathcal{Q} = 64 \sup_{m = (k,i) \in \mathbb K} \abs{k}\abs{q_m}$, and if $d=3$ define $\mathcal{Q} = 64 \sup_{m = (k,i) \in \mathbb K} \abs{q_m}$. Define 
$\eta_* = \nu / \mathcal Q$. 

\begin{lemma}[Lemma 3.7 in \cite{BBPS19}] \label{lem:TwistBd} Let $(u_t)$ solve either Systems \ref{sys:NSE} or \ref{sys:3DNSE}. There exists a $\gamma_* >0$, such that for all $0\leq \gamma < \gamma_*$, $T\geq 0$, $r\in (0,3)$, $C_0 \geq 0$, and $V(u) = V_{\beta,\eta}$ where $\beta \geq 0$ and $0 < e^{\gamma T}\eta < \eta_*$, there exists a constant $C = C(\gamma,T,r,C_0,\beta,\eta) >0$ such that the following estimate holds:
\begin{equation}\label{eq:Twistbd}
\EE_u \exp\left(C_0 \int_0^T\norm{u_s}_{\Hbf^r}\ds\right)\sup_{0\leq t\leq T}V^{e^{\gamma t}}(u_t) \leq C V(u).
\end{equation}

\end{lemma}
\begin{remark}
It suffices to take $\gamma_\ast = \frac{\nu}{8}$. 
\end{remark}

\begin{remark} \label{rmk:SuperL}
Note that Lemma \ref{lem:TwistBd} is strictly stronger than a drift condition. The improvement in the power of $V$ is sometimes called a \emph{super-Lyapunov property} and it provides an important strengthening of the notion of a drift condition.  To see that \eqref{eq:Twistbd} implies a drift condition, we write $P_1 \varphi (u) = \E_u \varphi(u_1)$ as the Markov semi-group for Navier-Stokes and apply Jensen's inequality with \eqref{eq:Twistbd} to deduce that $\exists C_L > 0$,
\begin{align}
P_1 V \leq  (e^{C_L}V)^{e^{-\gamma}}. \label{eq:P_1-super-Lyapunov}
\end{align}
Hence, $\forall \delta > 0$, $\exists C_\delta > 0$ such that $P_1V \leq \delta V + C_{\delta}$. 
Furthermore, the bound \eqref{eq:P_1-super-Lyapunov} can be iterated with repeated applications of Jensen's inequality (c.f. [Proposition 5.11, \cite{HM11}]) to produce
\begin{align}
P_nV \leq e^{C_L\frac{e^{-\gamma}}{1-e^{-\gamma n}}} V^{e^{-\gamma n}}.  \label{ineq:SuperLyapIter}
\end{align}
\end{remark}

\subsection{Jacobian estimates} \label{sec:Jacobian}

In the course of this paper, we require a variety of Jacobian estimates for the projective process $(u_t, x_t^\kappa, v_t^\kappa)$ 
on $\Hbf \times P \T^d$ (defined in Section \ref{sec:GeoP2kOut}).  
Analogous estimates when $\kappa = 0$ were derived in [Section 3; \cite{BBPS19}] and the same estimates apply here as well (uniformly in $\kappa$).
This is because the Lagrangian and projective processes were estimated by $L^\infty$ estimates on the velocity (and its gradients), and hence are not sensitive to the noise path of $\widetilde{W}_t$ and so do not depend on $\kappa$.
Since no real changes are needed, we will merely state the necessary lemmas here and refer the reader to [Section 3; \cite{BBPS19}] for proofs. 

Let us establish some useful shorthand notation.
Recall the projective process $(\hat{z}^\kappa_t) = (u_t,x_t^\kappa,v^\kappa_t)$ solves the abstract SDE in $\Hbf\times P\T^d$
\[
	\partial_t \hat{z}^\kappa_t = F(\hat{z}^\kappa_t) + Q\dot{W}_t + \sqrt{2\kappa} \dot{\widetilde{W}}_t.
\]
where we view $Q \dot{W}_t$ and $\dot{\widetilde{W}}_t$ as extended to $\Hbf\times T_{v^\kappa_t}P\T^d$ (we will abbreviate $T_{v} P\T^d = T_{(x,v)} P\T^d$) in the obvious manner and for each $\hat{z} = (u,x,v)\in \Hbf\times P\T^d$ we write
\[
	F(\hat{z}) = \begin{pmatrix}-B(u,u) - Au\\ u(x)\\ (I - v\tensor v) (Du(x)v)\end{pmatrix}.
\] 
The Jacobian process $J^\kappa_{s,t}$ denotes the Fr\'echet derivative of the solution $\hat z_t^\kappa$ with respect to the value at time $s<t$. Hence, $J^\kappa_{s,t}$ solves the operator-valued equation
\begin{equation}\label{eq:Jacdef}
	\partial_t J^\kappa_{s,t} = DF(\hat{z}^\kappa_t)J^\kappa_{s,t}, \quad J^\kappa_{s,s} = \Id.
\end{equation}
Additionally we let $K^\kappa_{s,t}:\Wbf\times T_{v^\kappa_t}P\T^d \to \Wbf\times T_{v^\kappa_s}P\T^d$ denote the adjoint of $J^\kappa_{s,t}$, in the sense that
\begin{align}
\brak{f, J^\kappa_{s,t} \xi }_{\Wbf} = \brak{K^\kappa_{s,t} f,  \xi }_{\Wbf}.
\end{align}
A straightforward calculation (see \cite{HM11}) shows that $K^\kappa_{s,t}$ solves the backward-in-time equation
\begin{align}
\partial_s K^\kappa_{s,t}  = -DF(\hat{z}^\kappa_s)^\ast K^\kappa_{s,t}, \quad K^\kappa_{t,t} = I \, , 
\end{align}
where $DF(\hat{z}^\kappa_s)^*: \Wbf \times T_{v^\kappa_s}P\T^d \to \Wbf \times T_{v^\kappa_s}P\T^d$ is the adjoint to $DF(\hat{z}^\kappa_s)$.

In what follows, we will find it convenient to let $\tilde{z} = (\tilde{u},\tilde{x},\tilde{v}) \in \Wbf \times T_{v^\kappa_s} P \mathbb T^d$ be an initial perturbation and denote 
\[
\tilde{z}^\kappa_t := (\tilde{u}_t, \tilde{x}^\kappa_t, \tilde{v}^\kappa_t) = J_{s,t}^\kappa \tilde{z}  \in \Wbf \times T_{v_t^\kappa}P\T^d,
\]
which readily solves the linear evolution equation
\[
	\partial_t \tilde{z}_t = DF(\hat{z}_t)\tilde{z}_t, \quad \tilde{z}_s = \tilde{z}.
\]

We now state the necessary Jacobian estimates. As usual, all constants are implicitly independent of $\kappa$.
\begin{lemma} \label{lem:PathJacobian} 
$\forall \sigma > \frac{d}{2}+1$, $\forall r \in (\frac{d}{2}+1,3)$,  $\exists C, q' > 0$ such that the following holds path-wise
\begin{subequations} 
\begin{align}
\norm{\tilde{u}_t}_{\Wbf} & \leq \norm{\tilde{u}}_{\Wbf} \exp\left(C \int_s^t \norm{u_\tau}_{\Hbf^{r}} \dee\tau \right) \label{def:JHbd} \\ 
\norm{J_{s,t}^\kappa}_{\Hbf^\sigma \to \Hbf^\sigma} & \lesssim \exp\left(C \int_s^t \norm{u_\tau}_{\Hbf^{r}} \dee\tau \right) \left(1 + \brak{t-s}^3 \sup_{s < \tau < t}\norm{u_\tau}_{\Hbf^\sigma}^{q'} \right). \label{def:JHSigbd}
\end{align}
\end{subequations}
\end{lemma}

\begin{lemma}[Jacobian bounds in expectation] \label{lem:JacExps}
For all $\sigma$ and all $\eta > 0$, there is a constant $C_J$ such that the following holds for all $1 \leq p < \infty$, 
\begin{align}
\sup_{s \leq t \leq 1}\EE\norm{J_{s,t}^\kappa}^p_{\Hbf^\sigma \to \Hbf^\sigma} & \leq V_{q', \eta}^p (u_s) \exp \left( p C_J \right). 
\end{align}
\end{lemma}

\begin{lemma} \label{lem:JacSmooth}
Let $\gamma \in [0,\alpha - \frac{d}{2})$ and $r \in (\frac{d}{2}+1,3)$. Then, $\exists \varkappa'$ such that the following holds path-wise for $0 \leq s \leq t \leq 1$: 
\begin{align}
(t-s)^{\frac{\gamma}{2(d-1)}}\norm{J^\kappa_{s,t}}_{\Wbf \to \Hbf^\gamma} \lesssim  \exp\left(C \int_s^t \norm{u_\tau}_{\Hbf^{r}} \dee\tau \right) \left(1 +  \sup_{\tau \in (s,t)} \norm{u_\tau}_{\Hbf^\sigma}^{\varkappa'}\right).    
\end{align}
\end{lemma}

\begin{lemma} \label{lem:PathAdjoint}
$\forall \sigma > \frac{d}{2}+1$, $\forall r \in (\frac{d}{2}+1,3)$, $\exists C, q' > 0$ such that the following hold path-wise 
\begin{align}
\norm{K_{s,t}^\kappa}_{\Wbf \to \Wbf} & \lesssim \exp\left(C \int_s^t \norm{u_\tau}_{\Hbf^{r}} \dee\tau \right) \\ 
\norm{K_{s,t}^\kappa}_{\Hbf \to \Hbf} & \lesssim \exp\left(C \int_s^t \norm{u_\tau}_{\Hbf^{r}} \dee\tau \right)\left(1 +  \brak{t-s}^3 \sup_{s < \tau < t}\norm{u_\tau}_{\Hbf^\sigma}^{q'}\right). 
\end{align}
\end{lemma}

\subsection{Malliavin calculus preliminaries}\label{subsec:MalliavinPrelims}

In order to make hypoellipticity arguments in infinite dimensions, we apply Malliavin calculus.
We will be dealing with variables $X  \in \Wbf\times \mathcal{M}$, where $\mathcal{M} = P \mathbb T^d, \mathcal{D}^c$ or trivial variations thereof, and assume that $X$ is a measurable function of a Wiener process $W = (W_t)$ on $\Lbf^2\times \R^M$. The Malliavin derivative $\MalD_h X$ of $X$ in a Cameron-Martin direction $h = (h_t) \in L^2(\R_+,\Lbf^2 \times \Real^M)$ is then defined by
\[
	\MalD_h X(W) := \lim_{\ep \to 0} \ep^{-1}\left[X\left(W + \eps \int_0^\cdot h_s\ds\right) -  X(W)\right]
\]
when the limit exists in $\Wbf\times \mathcal{M}$. If the above limit exists, we say that $X$ is {\em Malliavin differentiable}. In practice, the directional derivative $\MalD_h X$ admits a representation of the form
\begin{equation}\label{eq:D_sX-def}
\MalD_h X = \int_0^{\infty} \MalD_s X h_s\, \ds,
\end{equation}
where for a.e. $s\in \R_+$, $\MalD_s X$ is a Fr\`{e}ch\`{e}t derivative and defines a random, bounded linear operator from $\Lbf^2 \times \Real^M$ to $\Wbf \times \mathcal{M}$ (see \cite{Nualart06} for more details).
It is standard that if $X_t$ is a process adapted to the filtration $\mathscr{F}_t$ generated by $W_t$, then $\MalD_s X_t= 0$ if $s \geq t$. 

For real-valued random variables, the Malliavin derivative can be realized as a Fr\'{e}chet differential operator $\MalD: L^2(\Omega)\to L^2(\Omega; L^2(\R_+;\Lbf^2 \times \Real^M))$. The adjoint operator $\MalD^*:L^2(\Omega; L^2(\R_+;\Lbf^2 \times \Real^M)) \to L^2(\Omega)$ is referred to as the {\em Skorohod integral}, whose action on $h \in L^2(\Omega; L^2(\R_+;\Lbf^2 \times \Real^M))$ we denote by
\[
	\int_0^\infty \langle h_t,\delta W_t\rangle_{\Lbf^2} := \MalD^*h.
\]
The Skorohod integral is an extension of the usual It\^{o} integral; see \cite{Nualart06,HM11}. 
Above, we write $\langle {\cdot, \cdot} \rangle_{\Lbf^2}$ for the inner product on $\Lbf^2 \times \Real^M$, 
and throughout will suppress dependence of inner products on finite-dimensional factors. 
One moreover has the following version of It\^{o} isometry (see \cite{Nualart06} or \cite{DaPrato14}):
\[
	\E \left(\int_0^\infty \langle h_t, \delta W_t\rangle_{\Lbf^2}\right)^2\leq \EE \int_0^\infty \norm{h_{t}}_{\Lbf^2}^2 + \E \int_0^\infty \int_0^\infty \norm{\MalD_sh_{t}}_{\Lbf^2\to \Lbf^2}^2\ds\dt. \label{ineq:SkorIto}
\]

A fundamental result in the theory of Malliavin calculus is the Malliavin integration by parts formula. We stated the result for a process $(\hat z_t)$ which takes values in $\Hbf \times P\T^d$ (see e.g. \cite{DaPrato14,Nualart06}); only trivial modifications are needed to state for the other processes we apply Malliavin calculus to. 
\begin{proposition}\label{lem:MalIBP}
Let $\psi$ be a bounded Fr\'{e}ch\'{e}t differentiable function on $\Hbf\times P \T^d$ with bounded derivatives and let $h_t$ be any process satisfying 
\begin{equation}\label{eq:control-bound}
\EE \int_0^T \norm{h_t}_{\Lbf^2}^2 \dt + \E \int_0^T \int_0^T \norm{\MalD_sh_t}_{\Lbf^2\to \Lbf^2}^2\ds\dt < \infty \, .
\end{equation}
Then, the following relation holds
\[
	\E \MalD_h\psi(\hat{z}_T) = \E\left(\psi(\hat{z}_T)\int_0^T \langle h_s , \delta W_s\rangle_{\Lbf^2} \right).
\]
\end{proposition}

\section{Spectral theory for twisted Markov semigroups} \label{sec:prfSpecPic}

%
%

The primary aim of this section is to prove Proposition \ref{prop:outlineSpecPic}, which summarizes the spectral picture
we will use for the semigroups $\hat P^{\kappa, p}_t$ to construct our drift condition.
First, we outline the basic boundedness, mapping, and convergence properties of the projective $\hat P_t^\kappa$ and twisted $\hat P_t^{\kappa,p}$ Markov semigroups.
Starting with $p = 0$, in Section \ref{subsec:CVproj} we establish $\kappa$-uniform spectral gaps in $\mathring C_V$ for $\hat P_t^\kappa$ (Corollary \ref{cor:specCVuntwisted}),
while in Section \ref{subsec:CV1proj} we establish $\kappa$-uniform spectral gaps for $\hat P_{T_0}^\kappa$ in $\mathring C_V^1$, where $T_0 > 0$
is a fixed time chosen large ($\kappa$ independent). In Section \ref{subsec:miscProj}, we collect the remaining ingredients necessary
to apply our spectral perturbation arguments to conclude Proposition \ref{prop:outlineSpecPic}.

\subsection{Basic properties} \label{sec:BasicProj}

\subsubsection{Mapping and semigroup properties} 
\begin{lemma}
For all $p,\kappa \in [0,1]$, $\hat{P}_t^{\kappa,p}$ is a bounded (uniformly in $p,\kappa$) linear operator $C_V \to C_V$, satisfies the mapping $\hat P_t^{\kappa,p}\left( \mathring{C}_V \right) \subset \mathring{C}_V$, and moreover $\set{\hat{P}_t^{\kappa,p}}_{t \geq 0}$ defines a $C_0$-semigroup $\mathring{C}_V \to \mathring{C}_V$. 
\end{lemma}
\begin{proof}
Uniform boundedness in $\kappa$ for $p \neq 0$ follows from the representation
\begin{equation}\label{FKRep}
\hat P_t^{\kappa, p} \psi(u,x,v) = \E_{(u,x,v)} \exp \left(- p \int_0^t
H(u_s, x_s^\kappa, v_s^\kappa) \, ds \right) \psi(u_t, x_t^\kappa, v_t^\kappa) 
\end{equation}
of $\hat P^{\kappa, p}_t$ as a Feynman-Kac semigroup with potential
 $H(u,x,v) := \langle v, D u(x) v \rangle$, together with Lemma \ref{lem:TwistBd}. 
Since the $\sqrt{\kappa} \widetilde{W}_t$ noise applied to the Lagrangian flow is additive, 
the $\mathring{C}_V$ mapping property follows as in [Lemma 5.3 (a); \cite{BBPS19}] with no changes and the strong continuity follows as in [Proposition 5.5; \cite{BBPS19}].  
\end{proof}

\begin{lemma} \label{lem:C1VbdEtc}
There exists a time $T_0 > 0$ such that  $\forall p,\kappa \in [0,1]$, $\hat{P}_{T_0}^{\kappa,p}$ is a bounded (uniformly in $p,\kappa$) linear operator $C_V^1 \to C_V^1$ and satisfies the mapping property $\hat P_t^{\kappa,p}\left( \mathring{C}_V^1 \right) \subset \mathring{C}_V^1$.
\end{lemma}
\begin{proof}
The uniform-in-$\kappa$ boundedness follows from the representation \eqref{FKRep} and the argument in [Lemma 5.2 (a); \cite{BBPS19}]. The $\mathring{C}_V^1 \to \mathring{C}_V^1$ mapping property follows as in [Lemma 5.3 (b); \cite{BBPS19}]. 
\end{proof}

\subsubsection{Convergence results as $p \to 0$}

Next we show that $\hat P^{\kappa,p}_t \to \hat P^{\kappa}_t$ \emph{uniformly} in $\kappa$ as $p \to 0$ in various senses.
Both lemmas follow, as in [Lemma 5.2 (b); \cite{BBPS19}], from \eqref{FKRep} and Lemma \ref{lem:TwistBd}. 

\begin{lemma}\label{lem:unifConvCVPto0}
For fixed $t > 0$, the following uniform-in-$\kappa$ convergence holds:
\begin{align}
\lim_{p \to 0} \sup_{\kappa \in [0,1]} \| \hat P^{\kappa, p}_t - \hat P^\kappa_t\|_{C_V} = 0 \, .
\end{align}
\end{lemma}

\begin{lemma} \label{lem:unifConvC1VPto0}
For any fixed $T \geq T_0$, the following uniform-in-$\kappa$ convergence holds:
\begin{align}
\lim_{p \to 0} \sup_{\kappa \in [0,1]} \| \hat P_T^{\kappa, p} - \hat P^\kappa_T\|_{C_V^1} = 0 \, .
\end{align}
\end{lemma}

\subsection{Spectral picture for $\hat P_t^\kappa$ in $\mathring C_V$}\label{subsec:CVproj}

As the drift conditions are settled by Lemma \ref{lem:Lyapu}, our main task in applying Theorem \ref{thm:Harris} is to establish the uniform minorization conditions contained in Proposition \ref{prop:unifMinorCondProjective}. 

\subsubsection{Proposition \ref{prop:unifMinorCondProjective} (a): uniform strong Feller} \label{sec:USF}
The following is more than sufficient to imply Proposition \ref{prop:unifMinorCondProjective} (a). The result follows from checking uniformity in the argument used to prove [Proposition 2.12, \cite{BBPS18}] (which in turn builds from \cite{EH01}). We provide a brief sketch. 
\begin{lemma}\label{lem:ProjQSF}
There exists $a,b > 0$ such that there exists a continuous, monotone increasing, concave function $X:[0,\infty) \to [0,1]$ with $X(r) = 1$ for $r > 1$ and $X(0) = 0$ such that the following holds uniformly in $\kappa < 1$, $d_{\Hbf}(z^1,z^2) < 1$, and $t \in (0,1)$: 
\begin{align}
\abs{\hat{P}^\kappa_t \varphi(z^1) - \hat{P}^\kappa_t \varphi(z^2)} \leq X \left( \frac{d_{\Hbf}(z^1,z^2)}{t^a} \right) (1 + \norm{z^1}_{\Hbf}^b) \norm{\varphi}_{L^\infty}. 
\end{align}
\end{lemma}
\begin{proof}
It suffices to consider $v_t \in \S^{d-1}$; see [Section 6.1 of \cite{BBPS18}] for discussion. 
Define the following augmented system (denoting $\Pi_v = I - v \otimes v$), 
\begin{align}
\partial_t u_t & =-B(u_t,u_t) -Au_t + Q \dot{W}_t \\ 
\partial_t x_t & = u_t(x_t) + \sqrt{2\kappa} \dot{\widetilde{W}}_t\\ 
\partial_t v_t & = \Pi_{v_t} Du_t(x_t) v_t \\
\partial_t m_t & = \dot{M}_t,
\end{align}
where $M_t \in \Real^{2d}$ is a finite dimensional Wiener process independent from $W_t$ and $\widetilde{W}_t$,
and $m_t = (m_t^i)_{i = 1}^{2d}$ is a diffusion on $\R^{2d}$. 
We denote this augmented process by $w_t = (u_t,x_t,v_t,m_t)\in \Hbf \times \cM$, where $\cM = \T^d \times \S^{d-1} \times \Real^{2d}$, which satisfies the abstract SPDE
\begin{align}
\partial_t w_t = \widehat{F}(w_t) - Aw_t + \widehat{Q}\dot{W}_t, \label{eq:SDE-w}
\end{align}
where $\widehat{F}$ and $\widehat{Q}\dot{W}$ are given by
\[
	\widehat{F}(u,x,v,m) = 
		\begin{pmatrix}-B(u,u)\\ 
		u(x) \\ 
		\Pi_{v} Du(x)v \\ 
		0 \end{pmatrix} \, , 
		\quad 
\widehat{Q}\dot{W} = 
	\begin{pmatrix}
	Q\dot{W} \\ 
	\sqrt{2\kappa} \dot{\widetilde{W}}_t \\ 
	0 \\ 
	\dot{M}
	\end{pmatrix} 
\]
(with the obvious extended definition $Aw = (-A u,0,0,0)$). We similarly denote the associated Markov semigroup as $ \tilde{P}^\kappa_t$. 
Analogously to \cite{BBPS18}, we prove uniform strong Feller for the augmented process \eqref{eq:SDE-w}, which then implies the corresponding result for the original process.
As in \cite{BBPS18,EH01} we fix a smooth, non-negative cutoff function $\chi$ satisfying 
\begin{align}
\chi(z) & = \begin{cases}
0 \quad z < 1 \\ 
1 \quad z > 2 
\end{cases} 
\end{align}
and let $\chi_\rho(x) = \chi(x/\rho)$ for $\rho > 0$. We then define a regularized drift $F_\rho(w)$ by
\begin{align}
F_\rho(u,x,v,m) = (1-\chi_{3\rho}(\norm{u}_{\Hbf})) \widehat{F}(u,x,v,m) + \chi_{\rho}(\|u\|_{\Hbf})L(v,m),   
\end{align}
where $L(v,m)$ is a bounded vector-field on $\Hbf\times\cM$ given by
\begin{equation}
L(v,m) = \begin{pmatrix} 0 \\ \sum_{j=1}^d \hat{e}_j \frac{m^j}{\left(1 + \abs{m^j}^2\right)^{1/2} } \\  \Pi_{v}\sum_{j=1}^d\hat{e}_j\frac{m^{d+j}}{\left(1 + \abs{m^{d+j}}^2 \right)^{1/2}} \\ 0\end{pmatrix}
\end{equation}
Here, $\{\hat{e}_j\}_{j=1}^d$ the canonical basis for $\R^d$, and we are using that for each $v\in \S^{d-1}$, $\{\Pi_v e_j\}_{j=1}^d$ spans $T_v\S^{d-1}$. The cutoff/regularized process $w_t^\rho = (u_t^\rho,x_t^\rho,v_t^\rho,m_t)$ then satisfies the SPDE (replacing $\widehat{Q} \mapsto Q$ for notational simplicity), 
\begin{align}\label{eq:cut-off-eq}
\partial_t w_t^\rho = F_\rho(w_t^\rho) - A w_t^\rho + Q \dot{W}_t. 
\end{align}
Denote $\widetilde{P}^{\kappa;\rho}_t$ the Markov semigroup associated with the process \eqref{eq:cut-off-eq}.
See the discussions in \cite{BBPS18,RomitoXu11,EH01} on the utility of this cutoff.   
The main difficulty is to follow the proof of [Proposition 6.1; \cite{BBPS18}] and verify that the following gradient bound holds \emph{uniformly in $\kappa$}. 
\begin{lemma} \label{lem:GradBound}
There exists $a,b,\rho_\ast,T_\ast > 0$ all independent of $\kappa$, such that $\forall \rho \in (\rho_\ast,\infty)$, $\exists C_\rho$ (independent of $\kappa$) such that for $t < T_\ast$ and all $\varphi \in C_b^2(\mathbb H \times P \mathbb T^{d})$, we have that $w \to \widetilde{P}^{\kappa;\rho}_t\varphi(w)$ is Fr\'echet-differentiable, and satisfies
\begin{align}
\abs{D\widetilde{P}^{\kappa;\rho}_t\varphi(w)h} \leq C_\rho t^{-a} \left(1 + \norm{w}_{\Hbf}^{b} \right) \norm{\varphi}_{L^\infty} \norm{h}_{\Hbf \times T_v P \mathbb T^d} 
\end{align}
for all $h \in \Hbf \times T_v P \T^d$.
\end{lemma}
\begin{proof}
The proof of [Proposition 6.1; \cite{BBPS18}] is based on Malliavin calculus (see Section \ref{subsec:MalliavinPrelims}). 
Specifically, the main step is construct, for each $h \in \Hbf \times T_v P \T^d$, a suitably bounded control $g = (g_t)_{t \in [0,T]}$ such that the remainder
\begin{align}\label{eq:remainderDefn}
r_T = \MalD_g w_T -  D w_T h 
\end{align}
satisfies suitable estimates.  
First, the semigroup property and the Malliavin integration by parts formula (Proposition \ref{lem:MalIBP}) imply 
\begin{equation}\label{eq:MalIBPform-rem}
	D \widetilde{P}^{\kappa;\rho}_{2T}\varphi(w)h = \E\left(\widetilde{P}^{\kappa;\rho}_T\varphi(w_T)\int_0^T \langle g_t, \delta W(t)\rangle_{\Lbf^2}\right) - \E \left(D \widetilde{P}^{\kappa;\rho}_{T}\varphi(w_T)r_T\right),
\end{equation}
where the stochastic integral above is interpreted as a Skorohod integral (Section \ref{subsec:MalliavinPrelims}), since the control is not necessarily adapted.
Lemma \ref{lem:GradBound} then follows from a perturbation argument (see \cite{BBPS18}) provided we prove the analogue of [Lemma 6.3; \cite{BBPS18}]:
\begin{lemma} \label{lem:RTbd+SkrdBd}
For all $\kappa \in (0,\kappa_0)$ where $\kappa_0$ is a universal constant, and $\forall \rho > 0$, there exists constants $a_\ast, b_\ast > 0$ such that for $T$ sufficiently small (all independent of $\kappa$), there exists a control $g = (g_{t})_{t\in[0,T]}$ (in general depending on $\kappa$) satisfying the $\kappa$-uniform estimate
\begin{equation}\label{eq:Skoro-est}
\EE \int_0^T \norm{g_{t}}_{\Lbf^2}^2 \dt + \E \int_0^T \int_0^T \norm{\MalD_sg_{t}}_{\Lbf^2\to \Lbf^2}^2\ds\dt\leqc_\rho T^{-2a_\ast} (1+\norm{w}_{\Hbf})^{2b_\ast}\|h\|_{\Hbf \times T_v \cM}^2 \, 
\end{equation}
with remainder term $r_T$ as in \eqref{eq:remainderDefn} estimated by
\begin{align}
\EE\norm{r_T}_{\Hbf \times T_{v_T} \cM}^2 \lesssim_{\rho} T \norm{h}_{\Hbf \times T_v \cM}^2 \, . \label{eq:Remain-est}
\end{align}
\end{lemma} 
In order to prove this lemma we need (A) uniform-in-$\kappa$ estimates on the Jacobians and Malliavin derivatives as in Section 6.5 of \cite{BBPS18} and (B) uniform-in-$\kappa$ estimates on the partial Malliavin matrix (specifically, a $\kappa$-independent version of [Lemma 6.9; \cite{BBPS18}]).

Jacobian and Malliavin estimates analogous to those in [Section 6.5; \cite{BBPS18}] follow essentially verbatim here as well. This is for the same reason as in Section \ref{sec:Jacobian}: the estimates on the $(x_t,v_t)$ processes are done using $L^\infty$ estimates on $(u_t)$ and its derivatives, and so are insensitive to the specific noise-path of $\widetilde{W}_t$.

The uniform Jacobian and Malliavin estimates are sufficient to perform the arguments of [Section 6.5; \cite{BBPS18}] once one verifies the uniform-in-$\kappa$ non-degeneracy of the Malliavin matrix [Lemma 6.9; \cite{BBPS18}]. 
This requires more care.
The addition of  new noise directions does not change the uniform spanning property of [Lemma 6.13; \cite{BBPS18}] (the new noise directions cannot help in a $\kappa$-independent way, but they are not detrimental either).
The addition of the new directions adds additional $O(\kappa)$ or $O(\sqrt{\kappa})$ terms, for example, in [Proposition 6.10; \cite{BBPS18}]; however, these terms do not present any new difficulties beyond what is already required to treat the existing terms.

The additional noise term $\sqrt{\kappa} \widetilde W_t$ also does not significantly change the time-regularity estimates of Jacobian because the noise is additive and hence is not directly present on the Lagragian trajectories (recall time-regularity estimates of the Jacobian and its approximations play an important role in [Lemma 6.9; \cite{BBPS18}]).  
The $\sqrt{\kappa} \widetilde W_t$ term adds additional noise terms (to those already existing) to the expression for the time-derivatives of the Jacobian. On the other hand, the coefficients are controlled using the available regularity in $\Hbf$ together with BDG, similar to the noise terms that are already present.
We omit these repetitive details for brevity; see [Section 6; \cite{BBPS18}] for more detail.
\end{proof}

We are now ready to complete the proof of Lemma \ref{lem:ProjQSF}.
To see the uniform modulus of continuity, we proceed as in [Proposition 2.12; \cite{BBPS18}] and \cite{EH01}: 
\begin{align}
\abs{\widetilde{P}_t^{\kappa} \varphi(z^1) - \widetilde{P}^{\kappa}_t \varphi(z^2)} & \leq \abs{\widetilde{P}^{\kappa}_t \varphi(z^2) - \widetilde{P}^{\kappa,\rho}_t \varphi(z^2)} + \abs{\widetilde{P}^\kappa_t \varphi(z^1) - \widetilde{P}^{\kappa,\rho}_t \varphi(z^1)} \\ & \quad + \abs{\widetilde{P}^{\kappa,\rho}_t \varphi(z^1) - \widetilde{P}^{\kappa,\rho}_t \varphi(z^2)}. 
\end{align}
The first two terms are controlled noting that the moment bounds are independent of $\kappa$ because this noise only affects the degrees of freedom on the compact manifold $P \mathbb T^d$, hence by Proposition \ref{prop:WPapp}, for all $b >0$ there holds  (recall $d_{\Hbf}(z^1,z^2) < 1$ so that the sizes of $z^{j}$ are comparable), 
\begin{align}
\abs{\widetilde{P}^\kappa_t \varphi(z^j) - \widetilde{P}^{\kappa;\rho}_t \varphi(z^j)} & \lesssim \norm{\varphi}_{L^\infty}  \mathbb P \left(\sup_{0 < s < t} \norm{z_s^j}_\Hbf > \rho\right) \\ 
& \lesssim (1 + \norm{z^1}_{\Hbf}^{b})  \norm{\varphi}_{L^\infty} \frac{1}{\rho^{b}}. 
\end{align}
As in \cite{BBPS18,EH01}, an adaptation of [Lemma 7.1.5, \cite{DPZ96}] combined with Lemma \ref{lem:GradBound} implies
\begin{align}
\abs{\widetilde{P}^{\kappa;\rho}_t \varphi(z^1) - \widetilde{P}^{\kappa;\rho}_t \varphi(z^2)} \lesssim \frac{C_{\eps,\rho} d_{\Hbf}(z^1,z^2)}{t^{a}} (1 + \norm{z^1}^{b}_{\Hbf})e^{\eps \norm{z^1}_{\Wbf}^2}  \norm{\varphi}_{L^\infty}. 
\end{align}
Putting these estimates together implies 
\begin{align}
\abs{\widetilde{P}^\kappa_t \varphi(z^1) - \widetilde{P}^\kappa_t \varphi(z^2)} \leq \left(\frac{C_{\rho} d_{\Hbf}(z^1,z^2)}{t^{a}} + \frac{1}{\rho^{b}}\right) (1 + \norm{z^1}^{b}_{\Hbf}) \norm{\varphi}_{L^\infty}. 
\end{align}
Without loss of generality we can assume $C_{\rho}$ is monotone increasing, continuous in $\rho$, and satisfies $\lim_{\rho \to \infty} C_{\rho} = \infty$.
We define the modulus of continuity by
\begin{align}
X(r) := \min_{\rho \in [\rho_\ast,\infty)} \left(C_{\rho} r + \frac{1}{\rho^{b}}\right). 
\end{align}
Concavity, continuity, and monotone increasing all follow by definition and the continuity and monotonicity of $C_{\rho}$ and $\rho^{-b}$.
Finally it suffices to replace $X$ with  $\min(1, X(r))$ since the minimum of two concave, monotone, continuous functions is still concave and continuous. 
\end{proof}

\subsubsection{Proof of Proposition \ref{prop:unifMinorCondProjective}(b): uniform topological irreducibility} \label{sec:UTI}

The uniform topological irreducibility for Proposition \ref{prop:unifMinorCondProjective} (b) is proved by a standard approximate control argument; we include a sketch of the argument for completeness.
Specifically we prove the following.
\begin{lemma} \label{lem:projIR}
Fix an arbitrary $z_\ast \in \Hbf \times P\mathbb T^d$.
  For all $R > 0$, $\forall \eps > 0$, $\forall T > 0$, $\exists \kappa_0' = \kappa_0'(\eps,T)$ and $\exists \eta > 0$ such that for all $\kappa \in [0,\kappa_0']$ and $z \in \Hbf \times P\mathbb T^d$ with $\norm{z}_{\Hbf} < R$, 
\begin{align}
\hat P^{\kappa}_T \left(z , B_\eps(z_\ast) \right) > \eta. 
\end{align}
\end{lemma}
\begin{proof} 
Consider the deterministic, $\kappa = 0$, control problem on $\Hbf \times P \mathbb T^d$
\begin{subequations} \label{eq:detctrl}
\begin{align}
& \partial_t u_t + B(u_t,u_t) + Au_t = Q g_t \\
& \partial_t x_t = u_t (x_t) \\
& \partial_t v_t = Du_t(x_t) v_t.
\end{align}
\end{subequations}
Let $z = (u,x,v)$ and $z_\ast = (u',x',v')$. 
By local parabolic regularity (Proposition \ref{prop:WPapp}) it suffices to take $u \in \Hbf \cap H^{\sigma'}$ for any $\sigma < \sigma' < \alpha - \frac{d}{2}$
with $\norm{u}_{H^{\sigma'}} \lesssim R \max(1,T^{\frac{\sigma-\sigma'}{2(d-1)}})$.
For simplicity we further assume $T = 1$; the general case follows similarly.

The following lemma is standard (see the discussions in \cite{GHHM18,BBPS19} and the references therein).
\begin{lemma} \label{lem:scl1}
Let $u \in \Hbf \cap H^{\sigma'}$ for $\sigma < \sigma' < \alpha - \frac{d}{2}$ be as above.
Then $\forall \eps > 0$, $\exists \delta < \eps$ and a control $g:[0,\delta] \to \Lbf^2$ such that $\norm{u_{\delta}}_{\Hbf} \leq \eps/4$ and $\sup_{0 < t< \delta}\norm{u_t}_{\Hbf} \leq 3 \norm{u}_{\Hbf}$. Furthermore, $\sup_{0 < t < \delta}\norm{g_t}_{\Wbf}$ is bounded only in terms of $t$ and $\delta$. 
\end{lemma}

The following lemma is essentially [Lemma 7.1; \cite{BBPS18}].  
\begin{lemma}\label{lem:gctrl}
Let $a \in (0,\frac{1}{2})$ and suppose $u_a = 0$, $(x_a,v_a)= (x,v)$. There exists $C_g > 0$ such that $\forall (x,v),(x',v') \in \mathbb P\T^d$ there exists a control $g =: g^{ctr,a}$ satisfying $\sup_{t \in (a,1-a)} \norm{g^{ctr,a}_t}_{\Wbf} \leq C_g$ such that $u_{1-a} = 0$ and $(x_{1-a},v_{1-a}) = (x',v')$. 
\end{lemma}

The next lemma is essentially [Lemma 6.10; \cite{BBPS19}]. 
\begin{lemma} \label{lem:scl2}
Let $u' \in \Hbf$ be arbitrary. Then $\forall \eps > 0$, $\exists \delta \ll 1$ and a control $g:[1-\delta,1] \rightarrow \Lbf^2$ such that if $\norm{u_{1-\delta}}_{\Hbf} \leq \frac{\eps}{4}$, then there holds $\norm{u_1-u'}_{\Hbf} < \frac{\eps}{4}$, $\sup_{1-\delta \leq t \leq 1} \norm{u_t}_{\Hbf} \leq 3 \norm{u'}_{\Hbf}$, and $d(x_{1-\delta},x_1) + d(v_{1-\delta},v_1) \lesssim \delta \norm{u'}_{\Hbf}$. 
\end{lemma}

Lemmas \ref{lem:scl1}, \ref{lem:gctrl}, \ref{lem:scl2} exhibit an approximate control of the deterministic control problem \eqref{eq:detctrl}.
Let $(g_t)$ be such a deterministic control. 
As in [Lemma 7.3, \cite{BBPS18}] we have $\forall \eps$, $\exists \eta$ such that 
\begin{align}
\PP\left( \sup_{t \in (0,1)} \norm{\Gamma_t - \int_0^t e^{-(t-s)A} Qg_s ds}_{L^\infty_t(0,1;\Hbf)} < \eps \right) > \eta \, , 
\end{align}
where $\Gamma_t$ is the stochastic convolution as in \eqref{eq:Mild}. 

A remaining point is to bound the contribution of the noise term $\sqrt{2\kappa} \widetilde W_t$
applied directly to the Lagrangian flow. By a standard argument (using, e.g., the reflection principle applied to $\sup_{t \in (0,1)} \widetilde W_t^{(i)}$ for each component $\widetilde W_t^{i}$), we have the estimate
\begin{align}
\PP\left( \sup_{t \in (0,1)} \sqrt{2\kappa} |\widetilde{W}_t| > \eps \right) \lesssim \exp\left( - \frac{\eps^2}{4 d^2 \kappa} \right) 
\end{align}
for $\epsilon > 0$ fixed and all $\kappa$ sufficiently small (recall $d = 2$ or $3$). 
From here, Lemma \ref{lem:projIR} easily follows from a standard stability argument as in [Lemma 7.3; \cite{BBPS18}]. 
\end{proof}

\subsubsection{$\kappa$-uniform spectral gap for $\hat P_t^\kappa$ in $\mathring{C}_V$}

We now apply Theorem \ref{thm:Harris} with the Lypaunov function $V = V_{\beta, \eta}$ (Lemma \ref{lem:Lyapu})
and the minorization condition guaranteed by Proposition \ref{prop:unifMinorCondProjective} (c.f. Proposition \ref{prop:minorize}). 

\begin{proposition} \label{prop:CVspecGapProj}
There exist constants $C, \hat \gamma > 0$ (depending on the Lyapunov function $V$)
such that the following holds for all $\kappa > 0$ sufficiently small. 

There is a unique stationary measure $\nu^\kappa$ for the projective process in $\Hbf \times P \mathbb T^d$ and moreover, for all $\psi \in C_V$ and $t \geq 0$, we have
\begin{align}
\left\|\hat{P}^\kappa_t \psi - \int_{\Hbf \times P \mathbb T^d}  \psi \,\dee\nu^\kappa \right\|_{C_V} \leq C e^{-\hat \gamma t} \norm{\psi}_{C_V}. 
\end{align}
\end{proposition}

\begin{corollary}\label{cor:specCVuntwisted} There exists $c_0 \in (0,1)$ (independent of $\kappa$) such that, regarding $\hat P_t^\kappa$ as a $C_0$-semigroup of operators on $\mathring{C}_V$, we have that for all $t > 0$, the eigenvalue $1$ is simple, dominant and isolated, and  for all $t \geq 0$  and $\kappa$ sufficiently small 
\[
\sigma(\hat P_t^\kappa) \setminus \{ 1 \} \subseteq B_{c_0^t} (0). 
\]
\end{corollary}

\subsection{Spectral picture for $\hat P_{T_0}^\kappa$ in $\mathring{C}_V^1$}\label{subsec:CV1proj}

Following \cite{BBPS19}, a spectral gap for $\hat P^\kappa_{T_0}$ in $\mathring{C}_V^1$ will be
deduced from the uniform spectral gap in $C_V$ and the following Lasota-Yorke type gradient bound.
The proof requires checking the $\kappa$-uniformity of the analogous argument in [Proposition 4.6; \cite{BBPS19}] (which in turn follows \cite{HM06,HM11} closely with some minor variations). 
\begin{proposition}[Lasota-Yorke estimate] \label{prop:LY}
$\forall \beta' \geq 2$ sufficiently large and $\forall \eta' \in (0, \eta^\ast)$, $\exists C_1, \varkappa >0$ such that the following holds $\forall t > 0$, and $\hat{z}=(u,x,v)\in \Hbf\times P\T^d$:
\begin{align}
\|D \hat{P}^\kappa_t \psi(\hat{z})\|_{\Hbf^*} \leq C_1 V_{\beta',\eta'}(u) \left( \sqrt{\hat{P}^\kappa_t\abs{\psi}^2(\hat{z})} + e^{-\varkappa t} \sqrt{\hat{P}^\kappa_t \|D\psi\|_{\Hbf^*}^2(\hat{z})}\right).
\end{align}
\end{proposition}
\begin{proof}
The proof shares a few connections with that of Lemma \ref{lem:GradBound} above.
The proof is again based on Malliavin calculus and requires (A) uniform-in-$\kappa$ estimates on Jacobians and Malliavin derivatives; and (B) uniform-in-$\kappa$ estimates on the low-mode non-degeneracy of the Malliavin matrix (in this case, a different Malliavin matrix however).
The Jacobian and Malliavin derivative estimates carry over in a $\kappa$-uniform manner as in Section \ref{sec:Jacobian}. 

For an arbitrary control $(g_t): [0,T] \to \Lbf^2 \times \mathbb T^d $, denote the residual
\begin{align}
\rho_t = J_t \xi - \MalD_g \hat{z}_t. 
\end{align}
Then, Proposition \ref{lem:MalIBP} yields
\begin{align}
D\hat{P}_t \psi(\hat z)\xi = \EE D\psi(\hat z_t)J_t\xi = \EE D\psi(\hat z_t) \rho_t +  \EE\psi(\hat z_t) \int_0^t \brak{g_s,\delta W_s}_{\Lbf^2}. \label{ineq:DerivSwap}
\end{align}
Following the basic idea of \cite{HM06,HM11} and [Proposition 4.6; \cite{BBPS19}], the goal is to find a control $(g_t)$ such the latter Skorohod integral is uniformly bounded (for our case, in both $t$ and $\kappa$) and the former term is decaying exponentially (uniformly in $\kappa$). 

In this notation, the Malliavin matrix $\mathcal{M}$ of interest here takes the following form for $\xi\in \Wbf\times T_{v_t}P\T^d$: 
\begin{align}
\brak{\mathcal{M}_{s,t} \xi, \xi}_{\Wbf} = \sum_{k \in \mathbb K} \int_s^t q_k^2 \brak{e_k,K_{r,t}\xi}^2_{\Wbf} \dr + \sum_{k \in \set{1,..,d}} \int_s^t 2 \kappa \brak{\hat{e}_k,K_{r,t}\xi}^2_{\Wbf} \dr \, ,   \label{eq:MalNonDeg}
\end{align}
where $\set{\hat{e}_k}_{k \in \set{1,..,d}}$ denotes the canonical orthonormal basis on $\Real^d$.
One of the main steps of the proof is to verify the non-degeneracy estimate [Proposition 4.11; \cite{BBPS19}] uniformly in $\kappa$. 
The reasons why this non-degeneracy extends to \eqref{eq:MalNonDeg} in a $\kappa$-uniform way are similar to those given in the proof of Lemma \ref{lem:GradBound}.
First, the inclusion of new noise directions does not change the spanning of the brackets [Lemma 4.15; \cite{BBPS19}] (it neither helps nor hinders in a $\kappa$-independent way).
Second, the additional terms $O(\kappa)$ terms in \eqref{eq:MalNonDeg} and the additional $\sqrt{2\kappa} \dot{\widetilde{W}}_t$ in $d x_t^\kappa$ do not significantly change the latter arguments either: neither the time-regularity nor the space-regularity from the additional derivatives pose a significant new challenge in the analogues of [Lemma 4.18, Lemma 4.19; \cite{BBPS19}].
Hence, the proof of [Proposition 4.6; \cite{BBPS19}] carries over in a $\kappa$-uniform manner and we deduce Proposition \ref{prop:LY}. 
\end{proof} 

A straightforward argument (see [Proposition 4.7; \cite{BBPS19}]) combines Proposition \ref{prop:CVspecGapProj} with Proposition \ref{prop:LY} and the super-Lyapunov property (Remark \ref{rmk:SuperL}) to obtain the desired geometric ergodicity in $C_V^1$.
\begin{proposition} \label{prop:C1VSpec}
For all $V = V_{\beta,\eta}$ with $\beta$ sufficiently large and $\eta \in (0,\eta^\ast)$, we have that $\hat{P}^\kappa_{T_0}$ satisfies the following for $T_0$ sufficiently large (with $T_0$ and the implicit constant independent of $\kappa$): for $\psi \in C_V^1$, $\int \psi d\nu^\kappa = 0$, we have
\begin{align}
\norm{\hat{P}^\kappa_{n T_0} \psi}_{C_V^1} \lesssim e^{-\alpha n T_0 } \norm{\psi}_{C_V^1} \, .
\end{align}
\end{proposition}

With $T_0$ fixed once and for all, we immediately deduce the following.

\begin{corollary} \label{cor:SpecGapC1V}
There exists $c_0' \in (0,1)$ and $\kappa_0 > 0$ such that for all $\kappa \in [0,\kappa_0]$, the eigenvalue $1$ is simple, dominant, and isolated for the operator $\hat P^\kappa_{T_0}$
on $\mathring{C}_V^1$, and satisfies
\[
\sigma(\hat P^\kappa_{T_0}) \setminus \{ 1\} \subset B_{c_0'}(0). 
\]
\end{corollary}

\subsection{Spectral picture for $\hat P^{\kappa, p}_t$ in $\mathring{C}_V$ and $\mathring{C}_V^1$}\label{subsec:miscProj}

We now proceed to prove the spectral pictures for $\hat P^{\kappa, p}_t$ in $C_V$ and $C_V^1$ as in 
Proposition \ref{prop:outlineSpecPic}.

\subsubsection{Proof of Proposition \ref{prop:outlineSpecPic}(a): Spectral picture in $C_V$}\label{subsubsec:specPicCV}

Throughout, $p_0, \kappa_0 > 0$ are fixed small constants, taken smaller as need be 
in the following arguments. Let $p \in [-p_0, p_0], \kappa \in [0,\kappa_0]$.



We next establish the $\kappa$-uniform spectral gap in \eqref{eq:specPicCVOutline}. 
We first establish some preliminary resolvent estimates. Below, $\pi^\kappa$ denotes the projection
$\phi \mapsto \int \phi d \nu^\kappa$ (the latter interpreted as a constant-valued function) on $C_V$.
Recall that $\pi^\kappa$ is a spectral projection for $\hat P^{\kappa}_t$ corresponding to the dominant
eigenvalue $1$. Below, we write $\hat P^\kappa_t = \pi^\kappa + R^\kappa_t$, where
$R^\kappa_t := \hat P^\kappa_t \circ (I - \pi^\kappa)$.

\begin{lemma}\label{lem:resolventEst} \
\begin{itemize}
\item[(a)] We have $\| \pi^\kappa\|_{C_V} = \int V d \mu$.
\item[(b)] For any $z \in \mathbb C \setminus \{ 0,1\}$, we have
\begin{align}
(z - \pi^\kappa)^{-1} = z^{-1} \left( I - \frac{1}{1-z} \pi^\kappa \right) \, . \label{eq:Neumann}
\end{align}
In particular, $\forall \delta > 0,\, \exists C_\delta > 0$ such that
$\| (z - \pi^\kappa)^{-1}\|_{C_V} \leq C_\delta$ on the set
$\{ |z - 1| \geq \delta\} \cap \{ |z| \geq 3/4\}$.
\item[(c)] Fix $t > 0$ sufficiently large so that $\|R^\kappa_t \|_{C_V} \leq 1/(2 C_\delta)$ 
(independently of $\kappa$; see Proposition \ref{prop:CVspecGapProj}). 
Then, $ \| (z - \hat P^\kappa_t)^{-1}\|_{C_V} \lesssim_\delta 1$ for all
$z \in \{ |z - 1| \geq \delta\} \cap \{ |z| \geq 3/4\}$.
\end{itemize}
\end{lemma}
\begin{proof}[Proof of Lemma \ref{lem:resolventEst}]
For (a) one checks $\| \pi^\kappa \phi\|_{C_V} = \abs{\int \phi d \nu^\kappa} \leq \| \phi  \|_{C_V}  \int V d \mu$.
Equality is achieved at the function $\phi \equiv 1$. For (b), \eqref{eq:Neumann} can be deduced using a Neumann series
for $|z| > 2 \int V d \mu$ and follows for $z \in \mathbb C \setminus \{ 0,1\}$ by analytic continuation. The estimate in (c)
follows from Proposition \ref{prop:CVspecGapProj} and the relation
\[
(z - \hat P^\kappa_t)^{-1} = (I - (z - \pi^\kappa)^{-1} R_t^\kappa)^{-1} (z - \pi^\kappa)^{-1} \, .
\]
\end{proof}

We now complete the proof of Proposition \ref{prop:outlineSpecPic} (a).
Fix $\delta > 0, \, \delta < 1/16$ and fix $t > 0$ sufficiently large so $\| R^\kappa_t\|_{C_V} \leq 1/2$ for all $\kappa \in [0,\kappa_0]$.
We first show $\sigma(\hat P^{\kappa, p}_t) \subset \{ |z| < 3/4\} \cup \{ |z - 1| < \delta\}$. 
Fix $z \in \{ |z| \geq 3/4\} \cap \{ |z - 1| \geq \delta\}$. Then 
\[
z - \hat P^{\kappa, p}_t = (z - \hat P^{\kappa}_t)(I - (z - \hat P^\kappa_t)^{-1} (\hat P^{\kappa, p}_t - \hat P^\kappa_t)). 
\]
Lemma \ref{lem:unifConvCVPto0} indicates that taking $p$ small, we can make
$\| \hat P^{\kappa, p}_t - \hat P^{\kappa}_t\|_{C_V}$ arbitrarily small. On the other hand,
by Lemma \ref{lem:resolventEst} (c), $\| (z - \hat P^\kappa_t)^{-1} \|^{-1}$ is bounded uniformly from below in terms of $\delta > 0$ above. Therefore, for any $\delta' > 0$, there exists
$p_0 > 0$ so that for all $p \in [- p_0, p_0]$, we have
$\| \hat P^{\kappa, p}_t - \hat P^{\kappa}_t\|_{C_V} < \delta'\| (z - \hat P^\kappa_t)^{-1} \|^{-1}$.

For such $p, \kappa$ and $z$, it now follows that $(z - \hat P^{\kappa, p}_t)^{-1}$ exists
and is bounded as a $C_V$ operator, hence 
\[
\sigma(\hat P^{\kappa, p}_t) \subset \{ |z| < 3/4\} \cup \{ |z - 1| < \delta\} \, .
\]
At this point, the spectral projector
\begin{align}
\pi^{p, \kappa} = \frac{1}{2 \pi i} \int_{|z - 1| = \delta} (z - \hat  P^{\kappa, p}_t)^{-1} dz  \label{eq:pipk}
\end{align}
is now defined. Repeating familiar estimates, $\pi^{p,\kappa}$ is $C_V$ close to $\pi^\kappa = \frac{1}{2 \pi i} \int_{|z - 1| = \delta} (z - \hat  P^{\kappa}_t)^{-1} dz$, and hence must be 
rank 1. We conclude that there is a unique real, positive eigenvalue $e^{- t \Lambda(p, \kappa)}$ in $\{ |z - 1 | < \delta\}$.

At this point, we have shown that for some fixed $t$ the desired spectral picture holds. Passing from continuous to discrete
time can now be carried out by repeating verbatim the arguments in the proof [Proposition 2.16 in Section 5.2 of \cite{BBPS19}].

\subsubsection{Proof of Proposition \ref{prop:outlineSpecPic}(b): Spectral picture in $C_V^1$} \label{sec:CV1twist}
Completing the  proof of Proposition \ref{prop:outlineSpecPic}(b) is by now straightforward.
From the mapping and boundedness in Lemma \ref{lem:C1VbdEtc} and the convergence in Lemma \ref{lem:unifConvC1VPto0}, coupled with the $C^1_V$ uniform spectral gap in Corollary \ref{cor:SpecGapC1V}, Lemma \ref{lem:resolventEst} holds with $C_V^1$ replacing $C_V$ on taking  $t \geq T_0$.
The desired spectral picture at any time $T$ sufficiently large now follows from the arguments given for $C_V$ in Section \ref{subsubsec:specPicCV}.

\section{Uniform spectral perturbation of twisted Markov semigroups}\label{sec:prfUniform-spec}

%
%

Our goal in this section is to complete the proof of Proposition \ref{prop:psiPUniform}.
Given Proposition \ref{prop:outlineSpecPic}, this is mainly a matter of proving the convergence of the dominant eigenvalues/functions as $\kappa \to 0$, i.e. $\Lambda(p, \kappa) \to \Lambda(p,0)$ and $\psi_{p, \kappa} \to \psi_{p,0}$ as in Proposition \ref{prop:convSpecObjectsOutline}. 

\subsection{Preliminary estimates in the limit $\kappa \to 0$}
 
Below, $\pi^{\kappa, p}$ denotes the spectral projector for $\hat P_T^{\kappa, p}$, regarded
either on $C_V$ or $C_V^1$.
The following lemma provides uniform estimates and convergence on the spectral projectors. It is a straightforward consequence of the resolvent arguments Lemma \ref{lem:resolventEst} and Section \ref{subsubsec:specPicCV} above.

\begin{lemma}\label{cor:specProjUnifCTRL}
We have
\begin{align}
\lim_{p \to 0} \sup_{\kappa \in [0,\kappa_0]} \| \pi^{\kappa, p} - \pi^\kappa\|_{C_V^1 \to C^1_V} = 0 \, .\label{eq:PikpPikConv}
\end{align}
In particular, Proposition \ref{prop:psiPUniform} (b) (ii) holds: for all $p_0,\kappa_0$ sufficiently small we have
\[
\sup_{p \in [0,p_0]} \sup_{\kappa \in [0,\kappa_0]} \| \psi_{p, \kappa}\|_{C_V^1} \lesssim 1 \, .
\]
\end{lemma}
\begin{proof}
Recall from \eqref{eq:pipk} the formula for $\pi^{\kappa,p}$. By repeating the arguments used to bound $\pi^{p,\kappa}$ in the proof of Proposition \ref{prop:outlineSpecPic} above, the convergence \eqref{eq:PikpPikConv} follows from Lemma \ref{lem:unifConvC1VPto0}.
\end{proof}

Obviously, a critical part of our proof has to do with the precise sense in which the semigroups 
 $\hat{P}^p_t$ and $\hat{P}^{\kappa,p}_t$ are close. For this, we start by understanding how the $\kappa$ projective process $(x^\kappa_t,v^\kappa_t)$ and the $\kappa =0$ process $(x_t,v_t)$ converge to each other in a suitable sense.

\begin{lemma}\label{lem:proj-kap-convergence}
The following estimate holds for each $t>0$:
\[
	\widetilde{\E}\sup_{s\in[0,t]}d(x^\kappa_s,v^\kappa_s;x_s,v_s) \leqc \sqrt{\kappa t} \exp\left({\int_0^t\|\nabla u_s\|_{\infty}\ds}\right).
\]
\end{lemma}
\begin{proof}
This follows from the fact that 
\[
\sup_{s\in[0,t]}d(x^\kappa_s,v^\kappa_s;x_s,v_s) \leqc \int_0^t \|\nabla u_s\|_{\infty}\sup_{r\in[0,s]}d(x^\kappa_r,v^\kappa_r;x_r,v_r)\ds + \sqrt{2\kappa}\sup_{s\in[0,t]}|\widetilde{W}_s|
\]
Taking expectation with $\widetilde{\E}$, using $\widetilde{\E}\sup_{s\in[0,t]}|\widetilde{W}_s| \leqc t^{1/2}$, and applying Gr\"onwall's lemma gives the result.
\end{proof}

Next, we show the continuity in the strong operator topology of $P^{p,\kappa}_t\phi$ in $C_V$ as $\kappa \to 0$. Below, $V = V_{\beta, \eta}$ as in Lemma \ref{lem:Lyapu}.

\begin{lemma}\label{lem:SOT-convergence}
Assume $\beta$ is sufficiently large. Then, there exists $\kappa_0, p_0 > 0$ so that 
for each $\psi \in \mathring{C}_V$, the following holds for any $t > 0$ fixed:
\[
	\lim_{\kappa \to 0}\sup_{p\in[-p_0,p_0]} \|P^{p,\kappa}_t\psi - P^p_t\psi\|_{C_V \to C_V} = 0.
\]
\end{lemma}
\begin{proof}
Recall that similar to our proof of strong continuity in \cite{BBPS19}, in light of the boundedness of $P^{p,\kappa}_t$ as $\kappa \to 0$, it is sufficient to show strong continuity on smooth cylinder functions $\psi \in C^\infty$. First note that for such $\psi$
\[
	\widetilde{\E}\left|\exp\left(\int_0^t H(z_s)\ds\right)(\psi(z^\kappa_t) - \psi(z_t))\right| \leqc_\psi \exp\left(p\int_0^t \|\nabla u_s\|_{\infty}\right)\widetilde{\E}d(x^\kappa_t,v^\kappa_t;x_t,v_t)
\]
and in addition 
\begin{equation}
\begin{aligned}
	&\widetilde{\E}\left | \exp\left(p\int_0^t H(z^\kappa_s)\ds\right) - \exp\left(p\int_0^t H(z_s)\ds\right)\right|\\
	&\hspace{1in}\leqc p\exp\left(\int_0^t p\|\nabla u_s\|_\infty\ds\right)\int_0^t \|\nabla^2 u_s\|_{\infty}\widetilde{\E}d(x^\kappa_s,v^\kappa_s;x_s,v_s)\ds.
\end{aligned}
\end{equation}
Applying Lemma \ref{lem:proj-kap-convergence} gives
\[
	|P_t^{p,\kappa}\psi - P_t^{p}\psi| \leqc_{\psi} \sqrt{\kappa t}(1+p) \E_u \exp\left((1+p)\int_0^t \|\nabla u_s\|_{\infty}\ds\right) \sup_{s\in [0,t]} \|u_s\|_{H^\sigma}.
\]
The proof is complete upon using Lemma \ref{lem:TwistBd} and sending $\kappa \to 0$. Note that, in fact, the above estimates are
uniform over compact time intervals $t \in [0,T]$.
\end{proof}

\subsection{Proof of Proposition \ref{prop:convSpecObjectsOutline}:  Convergence of $\set{\psi_{p,\kappa}}$ and $\Lambda(p,\kappa)$} \label{sec:ProofConv}



We are now ready for what is in some sense the crucial step in extending the work of \cite{BBPS19} to prove Theorem \ref{thm:unifTwoPtMixing}: passing to the limit in the eigenfunction/value relation for $\psi_{p,\kappa}$ as stated in Proposition \ref{prop:convSpecObjectsOutline}.

\begin{remark}\label{rmk:diffDomainCpct}
First, note that all the arguments we have made hold for arbitrary $\sigma \in (\alpha-2(d-1), \alpha  - \tfrac{d}{2})$.
Moreover, the corresponding $\Lambda(p, \kappa)$ are the same and $\psi_{p, \kappa} \in C_V (\Hbf^{\sigma'} \times P \T^d)$ agree on $\Hbf^\sigma \times P \T^d$
for $\sigma' < \sigma$ with $\sigma',\sigma \in (\alpha-2(d-1), \alpha  - \tfrac{d}{2})$. 
See [Remark 5.6; \cite{BBPS19}] for related discussions. 
\end{remark}

The first step is to use the uniform bound $\norm{\psi_{p,\kappa}}_{C^1_V} \lesssim 1$ to apply the Arzela-Ascoli theorem in classes of observables to extract limit points of $\set{\psi_{p,\kappa}}_{\kappa \in (0,\kappa_0]}$. This is a little subtle due to the interplay between regularity in $H^\sigma$ vs $H^{\sigma'}$ and regularity in the space of observables, $C_V$ vs $C_V^1$.

\begin{lemma} \label{lem:compact}
There exists $p_0 , \kappa_0 > 0$ such that the following holds. For any $p \in [0,p_0]$ and 
any sequence $\{ \kappa_n\}_{n=1}^\infty \subset (0,\kappa_0]$, $\kappa_n \to 0$, 
there exists an subsequence $\{ \kappa_{n'}\}_{n'=1}^\infty \subseteq \{ \kappa_n\}_{n=1}^\infty$
and a nonnegative, continuous function $\psi_{p, \ast} : \Hbf \times P \T^d \to \R_{\geq 0}$ such that
for any $R > 0$, we have
\begin{align}
\lim_{n' \to \infty} \sup_{\substack{z = (u,x,v) \in \Hbf \times P \T^d \\ \norm{u}_\Hbf \leq R}} \abs{\psi_{p, \kappa_{n'}}(z) - \psi_{p,\ast}(z)} = 0. 
\end{align}
\end{lemma}

\begin{proof} 
To start, fix $R > 0$. Let $\sigma' < \sigma$ and regard $\psi_{p, \kappa} \in C_V(\Hbf^{\sigma'} \times P \T^d)$
for all $p \in [0,p_0], \kappa \in [0,\kappa_0]$ as in Remark \ref{rmk:diffDomainCpct}. 
By Corollary \ref{cor:specProjUnifCTRL}, there exists $C_{\sigma'} > 0$ so that
\[
\norm{\psi_{p,\kappa}}_{C^1_V(\Hbf^{\sigma'} \times P \mathbb T^d)} \leq C_{\sigma'} \, .
\]
Note that the set $\mathscr D_R := \{ (u,x,v) : u \in \Hbf^{\sigma'}, \| u \|_{\Hbf^\sigma} \leq R, (x,v) \in P \T^d\}$ 
is compact in $\Hbf^{\sigma'} \times P \T^d$. 
By the uniform $C^1_V(H^{\sigma'}\times P\mathbb T^d)$ bound, it follows that
 the set $\set{\psi_{p,\kappa_n}|_{\mathscr D_R}}$ is uniformly bounded and $\Hbf^{\sigma'}$-equi-continuous on the $\Hbf^{\sigma'}$-compact set $\mathscr D_R$. 
Therefore, by Arzela-Ascoli, there is a subsequence $\kappa_{n'} \to 0$ and a ($\Hbf^{\sigma'}$-uniformly continuous) function $\psi_{p;R} : \mathscr D_R \to \R_{\geq 0}$ such that
\begin{align}
\lim_{n \to \infty} \sup_{\norm{z}_{\Hbf} \leq R}\abs{\psi_{p,\kappa_n}(z) - \psi_{p;R}(z)}. 
\end{align}
By diagonalization, we may refine the subsequence $\{ \kappa_{n'}\}$ to find
 a limiting function $\psi_{p;\ast}$ defined over the entire $H^\sigma \times P\mathbb T^d$ and 
 continuous in this same topology 
 (note that continuity in $\Hbf^{\sigma'}\times P\mathbb T^d$ is stronger than continuity in $\Hbf^\sigma \times P\mathbb T^d$ if $\sigma'<\sigma$) such that $\psi_{p,\kappa_n}$ converges uniformly to $\psi_{p;\ast}$ on bounded sets.
The fact that $\abs{\psi_{p;\ast}(z)} \lesssim V(u)$ follows from this convergence and the $\kappa$-uniform estimates on  $\|\psi_{p,\kappa}\|_{C_V} $. 
\end{proof}

With Lemma \ref{lem:compact}, we can now pass to the limit in the eigenvalue. 

\begin{lemma}\label{lem:stabilityPsikappap}
We have $\lim_{\kappa \to 0} \Lambda(p, \kappa) = \Lambda(p,0)$.
\end{lemma}
\begin{proof}
Let $\psi_{\ast} = \lim_{n \to \infty} \psi_{p, \kappa_n}$ 
be a cluster point of $\{ \psi_{p, \kappa} \}_{\kappa > 0}$ as in Lemma \ref{lem:compact}. 

First we show that $\psi_{\ast}$ cannot be identically zero. By Corollary \ref{cor:specProjUnifCTRL},
for $p$ small enough the the spectral projectors $\pi^{p, \kappa}$ are $\kappa$-uniformly close to $\pi^\kappa$ in $C_V^1$.
 Since $\psi_{p, \kappa} = \pi^{p, \kappa}({\bf 1})$ and $\pi^\kappa(\mathbf{1}) = \mathbf{1}$, we conclude that
 $\sup_{\kappa \in [0,\kappa_0]}\| \psi_{p, \kappa} - {\bf 1}\|_{C_V} \ll 1$ for $p$ small enough. Therefore,
 for $p_0$ fixed and sufficiently small, we have that there exists $\delta_0,R_0 > 0$ so that
  $\psi_{p, \kappa} > \delta_0$ on $\{ \| z \|_{\Hbf} \leq R_0\}$. This lower estimate passes to $\psi_{\ast}$, hence
  it cannot vanish identically.
 
Next, we show that $\psi_{*} = c\psi_p$ for some $c > 0$. For this, notice that the uniform boundedness in Lemma \ref{lem:unifConvCVPto0} (with the uniform bound $\norm{\psi_{p,\kappa_n}}_{C_V} \lesssim 1$) and the convergence in Lemma \ref{lem:SOT-convergence} imply that 
\[
	\lim_{n\to\infty} \norm{\hat{P}_{t}^{p,\kappa_n}\psi_{p,\kappa_n} - \hat{P}_t^{p}\psi_{*}}_{C_V} = 0
\]
for fixed $t > 0$. Therefore 
\[
\hat P_t^p \psi_{*} = \lim_{\kappa_n \to 0} \hat P^{p, \kappa_n}_t \psi_{p, \kappa_n} = \lim_{\kappa_n \to 0} e^{- \Lambda(p, \kappa_n)t} \psi_{p, \kappa_n} = \left( \lim_{\kappa_n \to 0} e^{- \Lambda(p, \kappa_n) t}  \right) \psi_{\ast} \, .
\]
In the last equality, we have used the fact that $\psi_\ast > 0$ to deduce that the limit 
$e^{- t \Lambda_*} :=  \lim_n e^{- \Lambda(p, \kappa_n) t}$ exists. Therefore $\psi_{\ast}$ is an eigenfunction
of $\hat P^p_t$ with eigenvalue $e^{- \Lambda_* t}$. By Corollary \ref{cor:specProjUnifCTRL}, the limit $-\Lambda_* = -\lim_n \Lambda(p, \kappa_n)$ is strictly larger than $\log c_0$ (where $c_0$ is as in Proposition \ref{prop:outlineSpecPic} (a) for $\kappa = 0$, proved in \cite{BBPS19}) for $\forall p$ sufficiently small, by Proposition \ref{prop:outlineSpecPic} (a) in the $\kappa = 0$ case, we conclude that in fact $\Lambda_* = \Lambda(p, 0)$ and $\psi_{\ast} = c \psi_{p, 0}$ for some $c > 0$.
Moreover, the convergence $\Lambda(p, 0) = \lim_n \Lambda(p, \kappa_n)$ holds independently of the subsequence $(\kappa_n)$, and so we deduce
$\lim_{\kappa \to 0} \Lambda(p, \kappa) = \Lambda(p,0)$ as desired.
\end{proof}

It remains to show $\psi_{p, \kappa} \to \psi_p$ in the $C_V$ norm. We start
by checking $\kappa$-uniform convergence of the following limit formula for $\psi_{p, \kappa}$. 

\begin{lemma}\label{lem:uniformCV-limit}
The $C_V$ limit
\[
\psi_{p, \kappa}(u,x,v) = \lim_{t \to \infty} e^{\Lambda(p,\kappa)t} \hat P^{\kappa,p}_t{\bf 1}
\]
is uniform over $\kappa \in [0,\kappa_0]$.
\end{lemma}
\begin{proof}
Consider the operator
\[
R_t^{\kappa, p} := \hat P^{\kappa, p}_t \circ (I - \pi^{\kappa, p}) = (\hat P^{\kappa, p}_t - \hat P^{\kappa}_t)  \circ
(I - \pi^{\kappa, p}) + \hat P^\kappa_t \circ (\pi^{\kappa} - \pi^{\kappa, p}) + \hat P^\kappa_t \circ (I - \pi^\kappa) \, .
\]
Fix $t > 0$ so that $R_t^\kappa := \hat P_t^\kappa$ has $C_V$ norm $\leq 1/3$. 
Take $p$ sufficiently small (independently of $\kappa \in [0,\kappa_0]$
such that the above first and second terms are each $< 1/6$ (the first term
estimated as in Lemma \ref{lem:unifConvCVPto0} and the second as in
Section \ref{subsubsec:specPicCV}). Therefore $\| R_t^{\kappa, p}\|_{C_V} \leq 2/3$
uniformly in $\kappa$.  This implies the desired estimate.
\end{proof}

\begin{remark}
Note that by the same arguments as those applied to $\psi_p$ in [Lemma 5.7; \cite{BBPS19}], we deduce that $\psi_{p,\kappa} \geq 0$ for all $p,\kappa$ sufficiently small. 
\end{remark}

We now use this to show that the limits $\psi_{p, \kappa_n} \to \psi_{p, \ast}$ actually coincide
with $\psi_{p}$ (independent of the subsequence $\kappa_n \to 0$).
\begin{lemma}\label{lem:psitopsi-kappa}
For each $p\in[0,p_0]$,
\[
\lim_{\kappa\to 0} \|\psi_{p,\kappa} - \psi_{p}\|_{C_V} = 0.
\]
\end{lemma}
\begin{proof}
For each $t > 0$, we have
\begin{align*}
	 \|\psi_{p,\kappa} - \psi_p\|_{C_V} & \leq  \|\psi_{p,\kappa} - e^{\Lambda(p,\kappa)t} \hat P^{p,\kappa}_t{\bf 1}\|_{C_V} \\ 
	& + \|\psi_p - e^{\Lambda(p)t}\hat P^{p}_t{\bf 1}\|_{C_V} + \|e^{\Lambda(p,\kappa)t}\hat P^{p,\kappa}_t{\bf 1} - e^{\Lambda(p)t} \hat P^{p}_t{\bf 1}\|_{C_V} \, .
\end{align*}
Combining Lemma \ref{lem:stabilityPsikappap} and \ref{lem:SOT-convergence}, we see that
\[
\lim_{\kappa \to 0}\|e^{\Lambda(p,\kappa)t}P^{p,\kappa}_t{\bf 1} - e^{\Lambda(p)t}\hat P^{p}_t{\bf 1}\|_{C_V} = 0
\]
for each $t$ fixed, hence
\[
\limsup_{\kappa \to 0}  \|\psi_{p,\kappa} - \psi_p\|_{C_V} \leq \sup_{\kappa \in [0,\kappa_0]} ( 
 \|\psi_{\kappa,p} - e^{\Lambda(p,\kappa)t} \hat P^{p,\kappa}_t{\bf 1}\|_{C_V} + \|\psi_p - e^{\Lambda(p)t}\hat P^{p}_t{\bf 1}\|_{C_V}
)
\]
Sending $t \to \infty$ and applying Lemma \ref{lem:uniformCV-limit} completes the proof.
\end{proof}







The proof of Proposition \ref{prop:psiPUniform} is largely complete, save for
the uniform positive lower bounds on $\psi_{p, \kappa}$ on bounded sets as in item (b)(iii). 

\begin{lemma}\label{lem:unifLowerBdPsiPKappa}
For each $R > 0$, and $p\in [0,p_0]$ there exists $\kappa_0$ small enough such that 
\[
\inf_{\kappa \in [0,\kappa_0]} \inf_{\substack{(u,x,v) \in \Hbf \times P \T^d \\ \| u \| \leq R} } 
\psi_{p,\kappa} (u,x,v) > 0 \, .
\]
\end{lemma}

\begin{proof}
For $p_0$ sufficiently small, by [Lemma 5.7; \cite{BBPS19}], $\forall R > 0$, there exists $c = c_{R} > 0$ so that for all $p \in [0,p_0]$ 
on $\{V(u)\leq R\}$, we have $\psi_p \geq c$. Therefore, on $\{ V(u) \leq R\}$ we have
\[
	\psi_{p,\kappa} \geq \psi_p - \|\psi_{p,\kappa} - \psi_p\|_{C_V} V \geq c - \|\psi_{p,\kappa} - \psi_p\|_{C_V} R \, .
\]
Applying Lemma \ref{lem:psitopsi-kappa} and choosing $\kappa_0$ small enough depending on $R$ and $c$ gives $\psi_{\kappa,p} \geq \frac{1}{2}c$.

\end{proof}

\section{Geometric ergodicity for the two-point process}

The goal of this section is to apply Theorem \ref{thm:Harris} to deduce Theorem \ref{thm:unifTwoPtMixing}, namely the geometric ergodicity of $P_t^{(2,\kappa)}$.
The main difficulty is the construction of an appropriate drift condition with suitable $\kappa$ independent constants. This is done in Section \ref{sec:Udrift} below with the help of the uniform spectral theory deduced in Sections \ref{sec:prfSpecPic} and \ref{sec:prfUniform-spec}. First, in Section \ref{sec:P2kMap} we record basic properties of the semigroup $P_t^{(2),\kappa}$ of the two-point $\kappa$-regularized Lagrangian motion, namely that it is a $C_0$ semi-group on an appropriate separable Banach space. In Section \ref{sec:USFetc} we prove the uniform strong Feller and topological irreducibility needed to apply Proposition \ref{prop:minorize} to deduce the minorization condition \eqref{eq:minorize}.
Both Sections \ref{sec:P2kMap} and \ref{sec:USFetc} follow very similarly to analogous arguments in \cite{BBPS19} and Section \ref{sec:prfSpecPic}, hence some of proofs are only sketched with the reader encouraged to consult \cite{BBPS19} for more details.

\subsection{$C_0$-semigroup property} \label{sec:P2kMap}
Define the function 
\[
	\hat{V}(u,x,y) := d(x,y) ^{-p}V(u),
\]
where $p>0$ is small and fixed. 
Let $\mathring{C}_{\hat{V}}$ be the the $C_{\hat{V}}$-norm closure of smooth cylinder functions
\begin{align}
\mathring C_0^\infty (\Hbf \times \cD^c) := \{ \psi | \psi(u,x,y) = \phi(\Pi_{\mathcal K} u, x, v), \mathcal K \subset \mathbb K, 
  \phi \in C_0^\infty\}. 
\end{align}
The first step is to check that $P_t^{(2),\kappa}$ is uniformly bounded on $C_{\hat V}$ and maps the subspace $\mathring{C}_{\hat{V}}$ to itself.
\begin{lemma} \label{lem:C02pt}
For all $p \in (0,p_0)$, $\beta \geq 0$, $\eta \in (0,\eta^\ast)$, $P_t^{(2),\kappa}$ extends to a bounded linear operator on $C_{\hat{V}}$ and there exists a $C > 0$ such that 
for all $t> 0$ and $\kappa \in (0,1)$, 
\begin{align}
\|P_t^{(2),\kappa} \varphi\|_{C_{\hat{V}}} \leq e^{Ct} \norm{\varphi}_{C_{\hat{V}}}. 
\end{align}
Moreover, for all $t > 0$ and $\kappa \in (0,1)$, $P^{(2),\kappa}_t(\mathring{C}_{\hat{V}})\subseteq \mathring{C}_{\hat{V}}$.
\end{lemma}
\begin{proof}
Uniform boundedness follows as in [Lemma 6.11; \cite{BBPS19}] and the $\mathring{C}_{\hat{V}}$ mapping property follows as in [Proposition 6.12; \cite{BBPS19}] (which itself is analogous to [Proposition 5.5; \cite{BBPS19}]). 
\end{proof}
We will also find the following uniform-in-$\kappa$ strong continuity property for $P^{(2),\kappa}_t$ useful.
\begin{lemma}\label{lem:SOT-convergence-2pt}
Assume $\beta \geq 1$ is sufficiently large universal constant.
Then, there exists $\kappa_0 > 0$ so that for each $\varphi \in \mathring{C}_{\hat V}$, the following holds
\[
	\lim_{t \to 0}\sup_{\kappa\in [0,\kappa_0]}\|P^{(2),\kappa}_t\varphi - P^{(2)}_t\varphi\|_{C_{\hat V}} = 0.
\]
In particular, $\set{P^{(2),\kappa}_t}_{t \geq 0}$ defines a $C_0$-semigroup on $C_{\hat{V}}$. 
\end{lemma}
\begin{proof}
The argument is essentially the same as that applied for Lemma \ref{lem:SOT-convergence} above, hence the proof is omitted for brevity. 
\end{proof}

\subsection{Uniform strong Feller and irreducibility} \label{sec:USFetc}

The first lemma we need to verify is a uniform strong Feller property as in Lemma \ref{lem:ProjQSF} above. 
As in [Section 6.1.2; \cite{BBPS19}] it is convenient to define the following metric: for $z^1,z^2 \in \Hbf \times \cD^c$, define
\begin{align}
d_b(z^1,z^2) := \inf_{\gamma: z^1 \to z^2} \int_0^1 d(x_s,y_s)^{-b} (1 + \norm{u_s}_{\Hbf})^b \norm{\dot{\gamma}_s}_{\Hbf \times \Real^{2d}} ds, 
\end{align}
where the infimum is taken over all differentiable curves $[0,1]\ni t \mapsto \gamma_t = (u_t,x_t,y_t)$ in $\Hbf\times\mathcal{D}^c$ connecting $z^1$ and $z^2$. It is not hard to see that the metric $d_b(\cdot,\cdot)$ generates the $\Hbf\times \mathcal{D}^c$ topology since the extremal trajectories avoid the diagonal $\mathcal{D}$.

Using this metric, we obtain the following uniform strong Feller result; as the proof is essentially a combination of the arguments therein and those found in [Proposition 6.5; \cite{BBPS19}], we omit the proof for the sake of brevity. 
\begin{lemma}\label{lem:2ptQSF}
There exists $a,b > 0$ such that, there exists a continuous, monotone increasing, concave function $X:[0,\infty) \to [0,1]$ with $X(r) = 1$ for $r > 1$ and $X(0) = 0$ such that the following holds uniformly in $\kappa < 1$, $d_{b}(z^1,z^2) < 1$, $t \in (0,1)$, 
\begin{align}
\abs{P^{(2),\kappa}_t \varphi(z^1) - P^{(2),\kappa}_t \varphi(z^2)} \leq X\left( \frac{d_{b}(z^1,z^2)}{t^a} \right) (1 + \norm{z^1}_{\Hbf}^b) \norm{\varphi}_{L^\infty}. 
\end{align}
\end{lemma}
Next, we verify the uniform topological irreducibility away from the diagonal. Specifically, combining the methods used to prove Lemma \ref{lem:projIR} above with those of [Proposition 2.7; \cite{BBPS19}] we prove the following. The details are again omitted for brevity.

\begin{lemma} \label{lem:2ptIR}
Fix an arbitrary $z_\ast \in \Hbf \times \cD^c$.
For all $R > 0$ sufficiently large, $\forall \eps > 0$, $\forall T > 0$, $\exists \kappa_0' = \kappa_0'(\eps,T,R)$ and $\exists \eta > 0$ such that for all $\kappa \in [0,\kappa_0']$ and $z \in \Hbf \times \cD^c$ with $\max(\norm{u}_{\Hbf} + d(x,y)^{-1}, \norm{u_\ast}_{\Hbf} + d(x_\ast,y_\ast)^{-1}) < R$ (denoting $z = (u,x,y), z_\ast = (u_\ast,x_\ast,y_\ast)$
\begin{align}
\hat P^{(2),\kappa}_T \left(z , B_\eps(z_\ast) \right) > \eta, 
\end{align}
where we denote $B_\eps(z_\ast)$ the $\eps$-ball in $\Hbf \times \cD^c$. 
\end{lemma}

Lemmas \ref{lem:2ptQSF} and \ref{lem:2ptIR} are sufficient to apply Proposition \ref{prop:minorize} to deduce the minorization condition \eqref{eq:minorize}.

\subsection{Uniform drift conditions} \label{sec:Udrift}
As mentioned, the main effort of this section is to deduce a drift condition on the semi-group $P^{(2),\kappa}_{t}$ associated with the $\kappa$-two point motion $(u_t,x^\kappa_t,y_t^\kappa)$. As discussed in Section \ref{sec:outline}, it is natural to consider a Lyapunov function of the form
\[
	\Vc_\kappa(u,x,y) = h_{p,\kappa}(u,x,y) + V_{\beta+1,\eta}(u)
\]
where
\begin{equation}\label{eq:h_pk-def}
	h_{p,\kappa}(u,x,y) = \chi(|w|)|w|^{-p}\psi_{p,\kappa}\left(u,x,\frac{w}{|w|}\right),
\end{equation}
and $w = w(x,y)$ is the minimum displacement vector from $x$ to $y$, $\psi_{p,\kappa}$ are the positive eigenfunctions obtained in Proposition \ref{prop:psiPUniform} for a particular choice of $p\in(0,1)$ (sufficiently small) and $\chi(r)$ is a smooth cut-off equal to $1$ for $0\leq r<1/10$ and $0$ for $r>1/5$. The choice is $\beta >0$ above is fixed arbitrary, sufficiently large by the steps used to construct $\psi_{p,\kappa}$. 

Our goal is to prove the following drift condition for $\Vc_\kappa$.
\begin{proposition} \label{prop:drift}
There exists a $K \geq 1$ independent of $\kappa$ such that for all $\kappa > 0$ and $p\in(0,1)$ small enough
\[
	P^{(2),\kappa}_t \Vc_\kappa \leq e^{-\Lambda(p,\kappa)t}\Vc_{\kappa} + K.
\]
\end{proposition}
\begin{remark}
In light of the fact that $\Lambda(p,\kappa) \to \Lambda(p)$ as $\kappa \to 0$ we see that for $\kappa$ small enough, $P^{(2),\kappa}_t$ satisfies a uniform drift condition in the sense of Definition \ref{defn:drift}, with constants $\gamma$ and $K$ that independent of $\kappa$.
\end{remark}

Let $\mathcal{L}_{(2),\kappa}$ denote the generator of $P^{(2),\kappa}_t$ as a $C_0$ semi-group on $\mathring{C}_{\hat V}$. For convenience we will work with the coordinates $(u,x,w)$ where $w=w(x,y)$ is the minimum displacement vector from $x$ to $y$. The two point motion can then equivalently be written in these coordinates $(u_t,x_t^\kappa, w_t^\kappa)$, where
\[
	w_t^\kappa = w(x_t^\kappa,y_t^\kappa).
\]
Note that $w_t^\kappa$ is not directly subject to white-in-time forcing since $x_t^\kappa$ and $y_t^\kappa$ are driven by the same Brownian motion. Formally, in this new $(u,x,w)$ coordinate system, one expects the generator $\mathcal{L}_{(2),\kappa}$ to take the form
\[
\mathcal{L}_{(2),\kappa}\varphi = \mathcal{L}_{(1),\kappa}\varphi + (u(x+w)-u(x))\cdot \nabla_w\varphi.
\]
where $\mathcal{L}_{(1),\kappa}$ is the generator for the Lagragian process $(u_t,x_t^\kappa)$. Note that $\kappa>0$ is a singular perturbation at the level of the generator $\mathcal{L}_{(1),\kappa}$ since it corresponds to the addition of a $\kappa \Delta$. Naturally, the strategy is to relate $\mathcal{L}_{(2),\kappa}$ to the generator $\mathcal{L}_{p,\kappa}$ of the twisted Markov semi-group $P^{p,\kappa}_t$, which we know has a good uniform in $\kappa$ spectral gap, implying
\[
	\mathcal{L}_{p,\kappa}\psi_{p,\kappa} = -\Lambda(p,\kappa)\psi_{p,\kappa}.
\]
In order to do this, we must approximate the displacement process $w_t^\kappa$ with the linearized process
\[
	w_t^{*,\kappa} := D\phi^tw,\quad w = w(x,y).
\]
This can only be made sense of when $x$ and $y$ are suitably close, so the cut-off $\chi$ is necessary. Using that $\psi_{p,\kappa}$ is the dominant eigenfunction for $\mathcal{L}_{p,\kappa}$ we can show that $h_{p,\kappa}$ is an {\em approximate} eigenfunction of $P^{(2),\kappa}$ with error contributions coming from the cut-off $\chi$ and the approximation error made by approximating $w_t$ with $w_t^*$. This is made precise in the following key Lemma.
\begin{lemma} \label{lem:hpkEst}
For all $p\in(0,p_0)$, $\kappa \in [0,\kappa_0]$, $\eta \in (0,\eta^*)$ and $\beta\geq 1$ taken large enough, $h_{p,\kappa}$ belongs to $\mathrm{Dom}(\mathcal{L}_{(2),\kappa})$ on $\mathring{C}_{\widehat{V}_{p,\beta,\eta}}$ the following formula holds
\begin{align}
	\mathcal{L}_{(2),\kappa}h_{p,\kappa} = - \Lambda(p,\kappa)h_{p,\kappa} + \mathcal{E}_{p,\kappa} + \Sigma\cdot \nabla_w h_{p,\kappa} \label{eq:hpkform}
\end{align}
where
\[
\mathcal{E}_{p,\kappa}(u,x,y) = H\left(u,x,\frac{w}{|w|}\right)|w|^{1-p}\psi_{p,\kappa}\left(u,x,\frac{w}{|w|}\right)\chi^\prime(|w|),
\]
with $H(u,x,v) = \langle v,\nabla u(x)v\rangle$ and $\Sigma(u,x,w)  = u(x+w) - u(x) - Du(x)w$.
\end{lemma}

As in \cite{BBPS19}, the strategy to justifying \eqref{eq:hpkform} (and $h_{p,\kappa} \in \mathrm{Dom}(\mathcal{L}_{(2),\kappa})$) is to approximate $P^{(2),\kappa}_th_{p,\kappa}$ by the semi-group 
\[
TP_t^{\kappa}h_{p,\kappa}(u,x,w) = \E_{u,x,w}h_{p,\kappa}(u_t,x_t^\kappa,w^{*,\kappa}_t)
\]
for the linearized dynamics and write
\[
	\frac{P^{(2),\kappa}_th_{p,\kappa} - h_{p,\kappa}}{t} = \frac{TP_t^{\kappa}h_{p,\kappa} - h_{p,\kappa}}{t} + \frac{P^{(2),\kappa}_th_{p,\kappa} - TP^\kappa_th_{p,\kappa}}{t}.
\]
Showing that each term on the right-hand side has a limit as $t\to 0$ in $C_{\hat{V}_{p,\beta,\eta}}$. First, let us obtain the analogue of [Lemma 6.14; \cite{BBPS19}], which shows that the generator of the linearized semi-group $TP_t^{\kappa}$  behaves well applied to $h_{p,\kappa}$. 
\begin{lemma} \label{lem:Gen1}
For $p\in (0, p)$,  $\kappa \in [0,\kappa_0]$ and $\beta>0$ large enough, the following limit holds in $C_{\hat V_{p,\beta,\eta}}$
\[
	\lim_{t\to 0} \frac{TP_t^\kappa h_{p,\kappa} - h_{p,\kappa}}{t} = -\Lambda(p,\kappa) h_{p,\kappa} + \mathcal{E}_{p,\kappa}.
\]
\end{lemma}
\begin{proof}
Fix $\beta_0>0$ so that $\psi_{p,\kappa} \in \mathring{C}_{V_{\beta_0,\eta}}$.  The proof is almost the same as that of [Lemma 6.14; \cite{BBPS19}], with some small differences. Indeed, using here the fact that $\abs{w}^{-p} \psi_{p,\kappa}$ is an eigenfunction for $TP_t^\kappa$ with eigenvalue $e^{-\Lambda(p,\kappa)t}$, we find
\[
	\frac{TP_t^{\kappa}h_{p,\kappa} - h_{p,\kappa}}{t} = \frac{e^{-\Lambda(p,\kappa)t} - 1}{t}h_{p,\kappa} + \mathcal{E}_{p,\kappa} + \E R_t,
\]
where the remainder $R_t$ takes the form
\[
	R_t = |w_t^{*,\kappa}|^{-p}\psi_p(u_t,x_t^\kappa,v_t^\kappa)\frac{1}{t}\int_0^t|w^{*,\kappa}_s|H(u_s,x_s^\kappa,v_s^\kappa)\chi^\prime(|w^{*,\kappa}_s|)\ds - \mathcal{E}_{p,\kappa}.
\]
The goal is therefore to show that $R_t \to 0$ in $C_{\hat{V}_{p,\beta,\eta}}$ for some $\beta$ large enough as $t \to 0$. Note that, even though $|w_t^{*,\kappa}|$ depends on $\kappa$ it has the following formula
\begin{equation}\label{eq:w*-eq}
|w_t^{*,\kappa}| = \exp\left(\int_0^t H(u_s,x^\kappa_s,v^\kappa_s)\ds\right)|w|,
\end{equation}
and therefore is bounded independently of $\kappa$. Just as in [Lemma 6.14 , \cite{BBPS19}], using the fact that $\psi_{p,\kappa}$ is in $\mathring{C}_{V_{\beta_0,\eta}}$ and using a density argument to approximate it by cylinder functions $\psi^{(n)}_{p,\kappa}$, we can bound the remainder by
\[
	|R_t| \leqc |w|^{1-p}\exp\left(C_p\int_0^t \|u_s\|_{H^{r}}\ds\right)\sup_{s\in(0,t)} V_{\beta_0+1,\eta}(u_s)\left(C_{n,\kappa} \rho_t + \|\psi_{p,\kappa} - \psi_{p,\kappa}^{(n)}\|_{C_{V_{\beta_0,\eta}}}\right),
\]
for $r \in (1+d/2,3)$, where $C_n$ depends badly on $n$ and $D\psi_{p,\kappa}^{(n)}$ and
\[
	\rho_t = \sup_{s\in(0,t)}\left(\|u_s-u\|_{H^r} + d_{\T^d}(x_s^\kappa,x) + d_{P^{d-1}}(v_s^\kappa,v)\right).
\]
At this stage, the only significant difference from the proof in \cite{BBPS19} is that $d(x^\kappa_s,x)$ is influenced by the Brownian motion $\sqrt{\kappa}\widetilde{W}_t$ and is therefore given by
\[
	d_{\T^d}(x_t^\kappa,x) \leq \int_0^t \|u_s\|_{L^\infty}\ds + \sqrt{\kappa}|\widetilde{W}_t|, 
\] 
so that by the Burkholder-Davis-Gundy inequality and the fact that $\E\sup_{s\in(0,t)} \|u_s\|_{L^\infty}^2 \leqc e^{Ct} \|u\|_{L^\infty}^2$, we obtain for $t\leq 1$
\[
	\E\sup_{s\in (0,t)} d_{\T^d}(x_t^\kappa,x)^{2} \leqc_{\kappa} (1+\|u\|_{\Hbf})^2t
\]
Both $\|u_s-u\|_{H^r}$ and $d_{P^{d-1}}(v_s^\kappa,v)$ are dealt with exactly as in \cite{BBPS19}. Consequently , we obtain a bound on $R_t$ of the form
\[
	\E|R_t| \leqc |w|^{1-p}V_{\beta_1,\eta}(u)(C_{n,\kappa}t^{1/2} + \|\psi_{p,\kappa} - \psi_{p,\kappa}^{(n)}\|_{C_V}).
\]
for some constant depending on $n$ and $\kappa$ and $\beta_1 > \beta_0 +1$ large enough. Sending $t\to 0$ first and then sending $n\to\infty$ still gives the result.
\end{proof}
We similarly have the analogue of [Lemma 6.15; \cite{BBPS19}], which shows the error made in approximating $P^{(2),\kappa}_t$ by the linearized dynamics $TP^\kappa_t$.
\begin{lemma}\label{lem:Gen2}
For $p\in(0,p_0)$, $\kappa \in [0,\kappa_0]$ and $\beta >0$ large enough, the following limit holds in $C_{\widehat V_{p,\beta,\eta}}$
\[
	\lim_{t\to 0} \frac{P^{(2),\kappa}_t h_{p,\kappa} - TP_t^\kappa h_{p,\kappa}}{t} = \Sigma \cdot \nabla_w h_{p,\kappa}.
\]
\end{lemma}
\begin{proof}
Again, the proof is almost identical to the proof in [Lemma 6.15; \cite{BBPS19}] due to the fact that the approximation is happening on the process $w_t$, which does not have noise directly driving it (the Brownian motion on $x_t$ and $y_t$ cancel). The main difference is the appearance of some terms due to It\^{o}`s formula, which can easily be dealt with. We recall a sketch of the proof here. 
As in [Lemma 6.15, \cite{BBPS19}], we introduce the events (see \cite{BBPS19} for a motivation for the definition of these sets)
\[
	A_t := \left\{t \sup_{s\in(0,t)}\|\nabla u_s\|_{\infty} \leq \frac{1}{100}\right\}, \quad B_t := \left\{t \sup_{s\in(0,t)}(\|\nabla u_s\|_{\infty}(|w_s^\kappa| +|w_s^{*,\kappa}|))\leq \frac{|w_t^{*,\kappa}|}{2}\right\}.
\]
Note that for each $\delta >0$
\begin{equation}\label{eq:complement-set-est}
	\1_{A^c_t\cup B^c_t} \leqc t^{1+\delta}\exp\left(2(1+\delta)\int_0^t\|u_s\|_{H^r}\ds\right)\sup_{s\in(0,t)}\|u_s\|_{H^r}^{1+\delta},
\end{equation}
for $r \in (1+d/2,3)$, so that by Lemma \ref{lem:TwistBd} we have $\lim_{t\to 0}\P(A_t\cap B_t) = 1$. The first step is to write
\[
\begin{aligned}
	\frac{P^{(2),\kappa}_th_{p,\kappa} - TP^\kappa_th_{p,\kappa}}{t} &= \P(A_t\cap B_t)\,\Sigma\cdot \nabla_w h_{p,\kappa}+ \E (R_t^1 + R^2_t + R^3_t),
\end{aligned}
\]
where the remainders $R_t^1$,$R^2_t$ and $R_t^3$ are given by
\begin{align}
    R_t^1&= \frac{1}{t}\1_{A_t^c\cup B_t^c}(h_{p,\kappa}(u_t,x_t^\kappa,w_t^\kappa) - h_{p,\kappa}(u_t,x_t^\kappa,w_t^{*,\kappa}))\\
	R_t^2 &= \1_{A_t\cap B_t}\int_0^1\nabla_w h_{p,\kappa} (u_t,x_{t}^\kappa,w^{\theta,\kappa}_t)\dee \theta \cdot \left(\frac{w_t^\kappa - w_{t}^{*,\kappa}}{t} - \Sigma\right)\\
	R^3_t &= \1_{A_t\cap B_t}\left(\int_0^1 \nabla_w h_{p,\kappa}(u_t,x^\kappa_t,w^{\theta,\kappa}_t)\dee \theta- \nabla_w h_{p,\kappa}\right)\cdot \Sigma
\end{align}
and $w^{\theta,\kappa} := \theta w_t + (1-\theta)w_t^{*,\kappa}$.

In light of the fact that $\P(A_t\cap B_t) \to 1$, it suffices to show that $\E R^1_t,\E R_t^2$ and $\E R_t^3$ converge to $0$ in $C_{\hat{V}_{p,\beta,\eta}}$ for suitable choices of $\beta$ and $p$. Indeed, an easy application of \eqref{eq:complement-set-est} and Lemma \ref{lem:TwistBd} gives
\[
	\E |R_t^1| \leqc t^{\delta} |w|^{-p}\E \exp\left(C_{p,\delta}\int_0^t \|u_s\|_{H^r}\ds\right)\sup_{s\in (0,t)}V_{\beta_0+1,\eta}(u_s) \leqc t^\delta \hat{V}_{p,\beta_0+1,\eta}
\]
which implies $\E R_t^1 \to 0 $ in $C_{\hat{V}_{p,\beta_0+1,\eta}}$. Also, a similar argument to that in [Lemma 6.15 \cite{BBPS19}] using properties of the sets $A_t$ and $B_t$ gives
\[
	|R_t^2| \leqc \|\psi_{p,\kappa}\|_{C^1_V} V_{\beta,\eta}(u_t) |w_t^{*,\kappa}|^{-p-1} \rho^1_t,
\]
where
\begin{equation}\label{eq:rho-1-est}
	\rho_t^1 = \sup_{s\in(0,t)}\left(|u(y_s^\kappa) - u(x_s^\kappa) - u(y) + u(x)| + \|u_s-u\|_{H^r}|w| + \|u_s\|_{H^r}|w^{*,\kappa}_s - w|\right).
\end{equation}
In order to estimate $\rho_t^1$, the main difference this proof and the one in \cite{BBPS19} is that the quantity
\[
	|u_s(y_s^{\kappa}) - u_s(x_s^\kappa) - u(y) + u(x)|
\]
now has to be estimated using It\^{o}'s formula, which gives rise to a new terms of the form $|\kappa \Delta u_t(x_t)- \kappa \Delta u_t(y_t)|$, specifically using It\^{o}`s formula and that fact that $u_s$ is evaluated along Lagrangian trajectories gives
\begin{equation}
\begin{aligned}
u_s(y_s^\kappa) - u_s(x_s^\kappa) - u(y) + u(x) &= \int_0^s B(u_\tau,u_\tau) (x^\kappa_\tau) - B(u_\tau,u_\tau)(y^\kappa_\tau)\, \dee\tau\\
&+\sum_{m\in \mathbb{K}}q_m \int_0^s (e_m(y_\tau^\kappa) - e_m(x_\tau^\kappa))\dee W^m_\tau\\  
&+\int_0^s (u_\tau\cdot \nabla u_\tau)(y^\kappa_\tau) - (u_\tau\cdot\nabla u_\tau)(x^\kappa_\tau)\, \dee\tau\\
&+ \frac{1}{2}\kappa\int_0^s \Delta u_\tau (y^\kappa_\tau) - \Delta u_{\tau}(x^\kappa_\tau)\,\dee \tau\\
&+ \sqrt{\kappa}\int_0^s (Du_\tau(y_\tau^\kappa) - Du_\tau(x_\tau^\kappa))\dee \widetilde{W}_\tau. 
\end{aligned}
\end{equation}
However, since $\sigma$ is large enough, all the velocity fields are regular enough to bound the differences on the right-hand-side above by $(1 + \|u_s\|_{\Hbf}^2)|w_s^\kappa|$. Applying the BDG inequality and that fact that
\[
	|w_s^\kappa| \leq |w|\exp\left(\int_0^s \|u_\tau\|_{H^r}\dee \tau \right)
\] 
for $r \in (1+d/2,3)$, implies that for $t\leq 1$
\begin{equation}
\begin{aligned}
	\left(\E \sup_{s\in(0,t)} |u_s(y_s^{\kappa}) - u_s(x_s^\kappa) - u(y) + u(x)|^2\right)^{1/2} &\leqc t^{1/2}|w|\E \sup_{s\in(0,t)} \exp\left(\int_0^s \|u_\tau\|_{H^r}\dee \tau\right)(1+\|u_s\|_{\Hbf}^2)\\
	& \leqc t^{1/2}|w|V_{1,\eta}(u). 
\end{aligned}
\end{equation}
The terms $\|u_s-u\|_{H^r}|w|$ and $\|u_s\|_{H^r}|w^{*,\kappa}_s - w|$ in \eqref{eq:rho-1-est} are treated similarly with the help of the cut-off $\mathbb 1_{A_t}$ giving (using also Lemma \ref{lem:TwistBd}),
\[
	\left(\E (\rho^1_t)^2\right)^{1/2} \leqc t^{1/2} |w| V_{1,\eta}(u).
\] Combining this (along with the formula \eqref{eq:w*-eq} for $w^{*,\kappa}_t$) gives by Cauchy-Schwartz that
\[
	\E |R_t^2| \leqc t^{1/2} \|\psi_{p,\kappa}\|_{C^1_V} |w|^{-p} V_{\beta_1,\eta}(u),
\]
implying that $\E|R_t^2| \to 0$ in $C_{\hat{V}_{p,\beta_1,\eta}}$ as $t\to 0$ for some $\beta_1$ big enough.

Finally, to estimate $R_t^3$, as in \cite{BBPS19} we approximate $\psi_{p,\kappa}$ by smooth cylinder functions $\psi_{p,\kappa}^{(n)}$ in $C^1_V$, a straight-forward computation using the cut-off $\1_{B_t}$ shows that
\[
	\abs{R_t^3} \leqc |w|^{-p}\exp\left(C\int_0^t \|u_s\|_{H^r}\ds\right)\left(\sup_{s\in(0,t)} V_{\beta_1,\eta}(u_s)\right)\left(C_{n,\kappa}\rho^2_t + \|D_v\psi_{p,\kappa} - D_v \psi_{p,\kappa}^{(n)}\|_{C_{V}}\right),
\]
for $r \in (1+d/2,3)$ and some $\beta_2$ large enough, where $C_{n,\kappa}$ depends badly on $n$ and $D^2_v\psi_{p,\kappa}^{(n)}$ and
\[
	\rho_t^2 = \sup_{s\in(0,t)} \left(\|u_s -u\|_{H^{r}} + d_{\T^d}(x_s^\kappa,x) + \1_{A_s}|w^\kappa_s - w| + \1_{A_s}|w_s^{*,\kappa} - w|\right).
\]
Again, very similarly to the proof of Lemma \ref{lem:Gen1} $\rho_t^2$ can be estimated by BDG to conclude that
\[
	\E|R_t^3| \leqc |w|^{-p}V_{\beta_3,\eta}(u)(C_{n,\kappa}t^{1/2} + \|D_v\psi_{p,\kappa} - D_v\psi_{p,\kappa}^{(n)}\|_{C_V}).
\]
for some large enough $\beta_3$. Sending $t\to 0$ and then $n \to \infty$ implies that $\E|R_t^3|\to 0$ as $t\to \infty$ in $C_{\hat{V}_{p,\beta_3,\eta}}$. 
 \end{proof}

As explained above, Lemmas \ref{lem:Gen1} and \ref{lem:Gen2} are sufficient to complete the proof of Lemma \ref{lem:hpkEst}.

\begin{proof}[\textbf{Proof of Proposition \ref{prop:drift}}]

Given a $V_{\beta,\eta}$ and $p$ from Lemma \ref{lem:hpkEst} using Taylor expansion allows us to bound ( c.f. [Lemma 6.13; \cite{BBPS19}])
\[
	\abs{\mathcal{E}_{p,\kappa} + \Sigma\cdot \nabla_w h_{p,\kappa}} \leqc |w|^{1-p}V_{\beta+1,\eta} \|\psi_{p,\kappa}\|_{C^1_{V_{\beta,\eta}}}.
\]
Since we can take $p< 1$ and have uniform-in-$\kappa$ bounds on $\psi_{p,\kappa}$ in $C^1_{V_{\beta,\eta}}$ we obtain the estimate
\begin{equation}\label{eq:2pt-drift-error}
	\mathcal{L}_{(2),\kappa}h_{p,\kappa} \leq -\Lambda(p,\kappa)h_{p,\kappa} + C^\prime V_{\beta + 1,\eta},
\end{equation}
for some $\kappa$ independent constant $C^\prime$. The rest of the argument proceeds as in [Proposition 2.13; \cite{BBPS19}]. We briefly recall the sketch of the argument for the readers' convenience. 
Using the super Lyapunov property it was shown in [(6.13), \cite{BBPS19}] that the following holds for all $\zeta > 0$, (denoting $P_t$ the semi-group of the Navier-Stokes equations),
\begin{equation}\label{eq:EEV}
	e^{\Lambda(p,\kappa) t}P_t V_{\beta+1,\eta} - V_{\beta+1,\eta} \leq \int_0^t e^{\Lambda(p,\kappa)s}P_s\left((\Lambda(p,\kappa)- \zeta) V_{\beta+1,\eta}(u_s) + C_\zeta\right)\ds. 
\end{equation}
Then the estimate \eqref{eq:2pt-drift-error} on $\mathcal{L}_{(2),\kappa}h_{p,\kappa}$ implies the following
\begin{align}\label{2pt-hpk-drift}
e^{\Lambda(p,\kappa)t} P_t^{(2),\kappa} h_{p,\kappa} - h_{p,\kappa}\leq C^\prime \int_0^t e^{\Lambda(p,\kappa) s} P_s V_{\beta+1,\eta} ds. 
\end{align}
By choosing $\zeta - \Lambda(p,\kappa)$ sufficiently large and adding \eqref{eq:EEV} to \eqref{2pt-hpk-drift}, the desired drift condition follows. This same argument is carried out in more detail in [Proposition 2.13; \cite{BBPS19}]. 
\end{proof}

\section{Enhanced dissipation} \label{sec:ED}

We now turn to the proof of enhanced dissipation Theorem \ref{thm:ED}. We begin by proving an enhanced dissipation result for initial data $g\in H^1$.


\begin{lemma}
Let $\gamma$ and $D_\kappa$ be as in Theorem \ref{thm:UniMix} for $p \geq 2$ and $s =1$.
Then, for any mean-zero scalar $g \in H^1$, and associated $(g_t)$ solving \eqref{def:AD}, there holds 
\begin{align}\label{ineq:ED1}
\norm{g_t}_{L^2}^2 \leq \min\left( \norm{g}_{L^2}^2, \gamma D^{2}_\kappa(u,\omega) \kappa^{-1} \left(e^{2\gamma t}-1 \right)^{-1} \norm{g}_{H^1}^2 \right).
\end{align}
\end{lemma}

\begin{proof}
Note that because $\norm{g}_{L^2} \leq \norm{g}_{H^{-1}}^{1/2} \norm{g}_{H^1}^{1/2}$, by Theorem \ref{thm:UniMix} we have
\begin{align}
\frac{d}{dt}\norm{g_t}_{L^2}^2= -2\kappa \norm{\grad g_t}_{L^2}^2 \leq -2\kappa \frac{\norm{g_t}_{L^2}^4}{\norm{g_t}_{H^{-1}}^2}\leq -2\kappa \frac{\norm{g_t}_{L^2}^4}{D^2_\kappa(u,\omega) \norm{g}_{H^{1}}^2}  e^{2\gamma t}. 
\end{align}
Re-arranging gives
\begin{align}
-\frac{d}{dt}\left(\frac{1}{\norm{g_t}^2}\right) = \frac{1}{\norm{g_t}_{L^2}^4}\frac{d}{dt}\norm{g_t}_{L^2}^2 \leq -\kappa \frac{2}{D^2_\kappa(u,\omega) \norm{g}_{H^{1}}^2}  e^{2\gamma t},
\end{align}
and hence 
\begin{align}
\frac{1}{\norm{g}_{L^2}^2} - \frac{1}{\norm{g_t}_{L^2}^2} \leq -\kappa \frac{1}{\gamma D^2_\kappa(u,\omega) \norm{g}_{H^{1}}^2}  \left(e^{2\gamma t} - 1 \right). 
\end{align}
Rearranging again gives
\begin{align}
\norm{g_t}_{L^2}^2   \leq \frac{\norm{g}_{L^2}^2}{1 + \kappa \frac{\norm{g}_{L^2}^2}{\gamma D^2_\kappa(u,\omega) \norm{g}_{H^{1}}^2}  \left(e^{2\gamma t} - 1 \right)} \leq \gamma \kappa^{-1} D_\kappa^2(u,\omega) (e^{2 \gamma t} - 1)^{-1} \| g \|_{H^1}^2 \, .
\end{align}
\end{proof}

\begin{remark}\label{rmk:HsInstead}
Note that in the above proof, we could replace the $H^1$ norm of $g$ with any $H^s$ norm, $s \in (0,1)$, 
using instead $H^{-s}$-decay in Theorem \ref{thm:UniMix} and the interpolation for mean-zero $f$ $\norm{f}_{L^2} \leq \norm{\grad f}_{H^1}^{1-\theta} \norm{f}_{H^{-s}}^\theta$ for suitable $\theta = \theta(s)$. 
\end{remark}

We can complete the proof of Theorem \ref{thm:ED} and extend to any $L^2$ initial data using parabolic regularity. Indeed, for any mean-zero scalar $g \in H^1$, and associated $(g_t)$ solving \eqref{def:AD}, there holds by standard parabolic regularity arguments, for $r \in (\frac{d}{2} +1,3)$
\begin{align}
\norm{g_t}_{H^1} \leq  \underbrace{C \exp\left(Ct + \int_0^t \norm{u_s}_{H^r} ds \right) \sup_{0 < \tau < t} \norm{u_\tau}_{H^1}}_{(*)} \frac{\norm{g}_{L^2}}{\sqrt{\kappa}} \, ,  \label{ineq:paraReg} 
\end{align}
where $C > 0$ is a constant. 
For initial $u \in \Hbf$ and random noise paths $\omega \in \Omega$, 
define $\widetilde D(\omega, u)$ to be the quantity $(*)$ above with $t = 1$. 
By Lemma \ref{lem:TwistBd}, we have that $(\EE (\widetilde D(u, \omega))^p)^{1/p} \lesssim_{p,\eta} V_{\beta,\eta}(u)$ for all $\beta$ sufficiently large and all $\eta \in (0,\eta^\ast)$. 

By \eqref{ineq:ED1} for $t \geq 1$, there then holds
\begin{align}
  \norm{g_t}_{L^2} & \leq \min(\norm{g}_{L^2},  \sqrt{2\gamma} \kappa^{-1/2} D_\kappa(u_1,\theta_1\omega) e^{-\gamma t} \norm{g_1}_{H^1}) \\
  & \leq \kappa^{-1} \underbrace{  \widetilde D(u,\omega) D_\kappa(u_1,\theta_1\omega) e^{\gamma } }_{=:D_\kappa^\prime(u, \omega) } e^{- \gamma t} \norm{g}_{L^2}. \label{ineq:gtL2}
\end{align}
Above, $\theta_1 \omega(t) = \omega(t+1) - \omega(1)$ refers to the standard Wiener shift on paths in $C(\R_+;\Lbf^2)$.
This is precisely the inequality \eqref{ineq:ED}. It remains to 
estimate the $p$-th moment of $D_\kappa^\prime$.  

Let $V = V_{\beta, \eta}$ as in Lemma \ref{lem:Lyapu} for $\eta \in (0,\eta_*)$ arbitrary. When $\beta$ is taken sufficiently large, we have that
\begin{align}
\EE (D_\kappa'(u, \omega))^p & \lesssim \left(\EE (\widetilde D(u,\omega))^{2p}\right)^{1/2} \EE \left((D_\kappa)^{2p}(u_1,\theta_1\omega)\right)^{1/2} \\
& = \left(\EE (\widetilde D(u,\omega))^{2p}\right)^{1/2} \EE \big( \EE \left((D_\kappa)^{2p}(u_1,\theta_1\omega)| \mathscr{F}_{1}\right) \big)^{1/2}  \\
& \lesssim V^{p/2}(u) \left(\EE V^{p}(u_1) \right)^{1/2} \\
& \lesssim V^p(u) \,
\end{align}
where we used that fact that $u_1$ is $\mathscr{F}_1$ measurable and $\theta_1\omega$ is independent of $\mathscr{F}_1$.

\subsection{Optimality of the $O(\abs{\log \kappa})$ dissipation time-scale} \label{sec:EDopt}
We complete this section with the proof of Theorem \ref{thm:optTimeScale}, the optimality of the 
timescale $t = O(| \log \kappa|)$ for enhanced $L^2$ dissipation. To start, by the standard $H^1$ norm growth bound on \eqref{def:AD}, any solution satisfies the following lower bound on the time derivative of $\norm{g_t}_{L^2}$: 
\begin{align}
\frac{\dee}{\dt} \norm{g_t}_{L^2}^2 = -\kappa \norm{\grad g_t}_{L^2}^2 \geq -\kappa\exp\left( \int_0^t \norm{\grad u_\tau}_{L^\infty} \dee\tau \right)\norm{g}_{H^1}^2. \label{ineq:ExLwBd}
\end{align}
By a straightforward application of Lemma \ref{lem:TwistBd} and Borel-Cantelli (or, alternatively, the Birkhoff ergodic theorem), we observe the following almost sure growth bound. 
\begin{lemma} \label{lem:BVH1explo}
There exists a $\lambda >0$ and a random constant $\overline{D} : \Hbf \times \Omega \to [1,\infty)$, independent of $\kappa$,  such that
\[
	\exp\left( \int_0^t \norm{\grad u_\tau}_{L^\infty} d\tau \right) \leq \overline{D} e^{\lambda t} \, .
\]
Moreover, for any $\eta > 0$ with $p\eta \in (0,\eta^\ast)$ and $\beta \geq 1$, we have
$\E \overline{D}^p \leqc_p V^p(u)$ for $V = V_{\beta, \eta}$.
\end{lemma}
\noindent
Lemma \ref{lem:BVH1explo} and \eqref{ineq:ExLwBd} together imply the lower bound
\[
\norm{g_t}_{L^2}^2 \geq \|g\|_{L^2}^2 - \kappa \|g\|_{H^1}^2 \overline{D} \lambda^{-1} ( e^{\lambda t} - 1) 
\geq \|g\|_{L^2}^2 - \kappa \|g\|_{H^1}^2 \overline{D} t e^{\lambda t} 
\]
It follows that for each $\delta \in (0,1)$
\[
	\norm{g_{\delta |\log{\kappa}|}}_{L^2}^2 \geq \|g\|_{L^2}^2 \left( 1 - \delta | \log \kappa| \kappa^{1-\lambda \delta} \frac{\overline{D} \|g\|_{H^1}^2}{\| g \|_{L^2}^2} \right) . 
\]
Choosing 
\[
	\delta(g,u,\omega) := 
	\min\left\{\frac{\|g\|_{L^2}^2}{\|g\|_{H^1}^2 \overline{D}(u,\omega)},\frac{1}{2\lambda}\right\}
\]
gives
\[
	\norm{g_{\delta |\log{\kappa}|}}_{L^2}^2 \geq (1-|\log  \kappa_0| \kappa_0^{1/2})\|g\|_{L^2}^2 \, .
\]
Choosing $\kappa_0$ small enough so that $|\log \kappa_0| \kappa_0^{1/2} \leq 3/4$ implies 
$\tau_* \geq \delta | \log \kappa|$, where $\tau_*$ is the enhanced dissipation time
$\tau_* = \inf\{ t \geq 0 : \| g_t\|_{L^2} < \frac12 \| g \|_{L^2} \}$.  This completes the proof of
Theorem \ref{thm:optTimeScale}. 

\bibliographystyle{abbrv}
\bibliography{bibliography}

\end{document}